\documentclass[11pt]{amsart}
\usepackage{amssymb, amscd, amsmath, amsthm, epsf, epsfig, latexsym,color} 
\usepackage{hyperref}
\usepackage{comment}

\usepackage{mathtools}
\usepackage{pb-diagram}

\newcommand{\tu}{\tilde{u}}

\newcommand{\barP}{\overline P}

\newcommand{\TryPackage}[3]{\IfFileExists{#1.sty}{\usepackage{#1}#2}{#3}
}
\TryPackage{mathrsfs}{\renewcommand{\mathcal}{\mathscr}}{%
        \TryPackage{eucal}{}{}}

\newcommand{\ep}{\epsilon}

\newcommand{\Si}{\Sigma}

\newcommand{\eA}{\epsilon_A}
\newcommand{\eB}{\epsilon_B}

\newcommand{\bbi}{{{\bf i}}}

\newcommand{\bbk}{{{\bf k}}}

\newcommand{\HH}{{\mathbb H}}
\newcommand{\FF}{{\mathbb F}}
\newcommand{\ZZ}{{\mathbb Z}}
\newcommand{\RR}{{\mathbb R}}

\newcommand{\tA}{{\widetilde A}}
\newcommand{\tB}{{\widetilde B}}
\newcommand{\tC}{{\widetilde C}}
\newcommand{\tD}{{\widetilde D}}

\newcommand{\hC}{{\widehat C}}
\newcommand{\hA}{{\widehat A}}
\newcommand{\hB}{{\widehat B}}
\newcommand{\hD}{{\widehat D}}

\newcommand{\hb}{{\hat b}}

\newcommand{\Hom}{\operatorname{Hom}}

\newcommand{\Mas}{\operatorname{Mas}}
\newcommand{\nat}{\natural}

\newcommand{\Real}{\operatorname{Re}}

\graphicspath{ {figures/} }

\theoremstyle{definition}

\newtheorem{df}{Definition}[section]

\theoremstyle{plain}

\newtheorem{theorem}{Theorem}
\newtheorem{thm}[df]{Theorem}

\newtheorem{cor}[df]{Corollary}
\newtheorem{lem}[df]{Lemma}
\newtheorem{prop}[df]{Proposition}

\newtheorem{rem}[df]{Remark}

\newtheorem{conj}[df]{Conjecture}

\date{}

\thanks{ }

\author{Matthew Hedden}
\address{Department of Mathematics, Michigan State University, East Lansing, MI 48824} 
\email{mhedden@math.msu.edu}

\author{Christopher M. Herald}
\address{Department of Mathematics,   University of Nevada, Reno, Reno, NV 89557} 
\email{herald@unr.edu}

 \author{Paul Kirk}
\address{Department of Mathematics, Indiana University, Bloomington, IN 47405} 
\email{pkirk@indiana.edu}


\begin{document}

\title[The pillowcase and traceless representations of knot groups II] {The pillowcase and traceless representations of knot groups II: a Lagrangian-Floer theory in the pillowcase }

\begin{abstract}  We define an elementary  relatively $\ZZ/4$ graded Lagrangian-Floer chain complex   for  restricted immersions  of   compact 1-manifolds into the pillowcase, and apply it to the  intersection diagram obtained by taking traceless $SU(2)$ character varieties of 2-tangle decompositions of   knots.  Calculations for torus knots are  explained in terms of pictures in the punctured plane.  The relation to the reduced instanton homology of knots is explored. \end{abstract}

 \maketitle

 \section{Introduction}
 In \cite{KM1}, Kronheimer and Mrowka introduced a powerful invariant of a knot or link in a  3-manifold $L\subset X$ called {\em singular instanton knot homology}.  Denoted  $I^\nat(Y,L)$, their invariant is defined in the context of gauge theory.  Roughly, the theory associates to $L$ a chain group $C^\nat(X,L)$ generated by flat $SO(3)$ connections on  $X\setminus L$ which have a prescribed singularity near $L$. This group is endowed with a differential that counts anti-self-dual instantons on $X \times \mathbb{R}$ which limit to given flat connections on the ends.   Singular instanton knot homology has an important computational tool called the {\em skein exact triangle}. This is a long exact sequence relating the homology groups of links which agree outside of a small 3-ball, where they differ in a simple way.   Iterated application of the exact triangle using a collection of 3-balls leads to a spectral sequence which converges to $I^\nat(X,L)$. This spectral sequence, when applied to a collection of 3-balls containing all the crossings of a diagram for a link in the $3$-sphere, has $E_2$ page isomorphic to the well-known combinatorial knot invariant  Khovanov homology \cite{KM-khovanov}.  The existence of this spectral sequence, together with a non-triviality result for $I^\nat(S^3,L)$ coming from its relation to another knot invariant (sutured instanton knot homology) \cite{KM-s}, allowed Kronheimer and  Mrowka to prove the striking result that Khovanov homology detects the unknot.  Despite this triumph,  $I^\nat(X,L)$ remains rather mysterious.  This is due in large part to the fact that computations are extremely scarce.  Initially, the only route for computation was through  the aforementioned spectral sequence, but, aside from the instances where it collapses for simple reasons at Khovanov homology, little headway has been made in this direction  (though see \cite{KM-filtrations,LZ} for more sophisticated computations using the spectral sequence).
 
 Motivated by a desire for a more explicit understanding of the singular instanton chain complexes, we began a project in \cite{HHK} which aims to make concrete  direct calculations of $I^\nat(X,L)$.  This is not so easy, and a serious initial sticking  point arises from  the fact that the flat connections which generate $C^\nat(X,L)$ are never isolated.  Indeed, aside from the case of the unknot in $S^3$, the  spaces of flat connections studied by the theory are always positive dimensional varieties.  For this reason, one needs to perturb the Chern-Simons functional that gives rise to $I^\nat(X,L)$ through its Morse homology.  The robust holonomy perturbations used to set up the general theory destroy the concrete algebraic interpretation of the generating set for $C^\nat(X,K)$ in terms of certain traceless representations of the fundamental group of $X\setminus L$, and the goal of our first paper was to retain such an interpretation by way of explicit local, and in some sense minimal, holonomy perturbations.  The main idea from \cite{HHK} was to  pick a particular distinguished 3-ball in $X$ which intersects $L$ in a trivial 2-stranded tangle.   We then performed an explicit holonomy perturbation to the  Chern-Simons functional in the neighborhood of a curve living in this ball.  Using this perturbation allowed us, in a  variety of examples, to perform computations of singular instanton chain groups for many knots (e.g., many torus knots)  which implied that the spectral sequence from Khovanov homology necessarily had large rank higher differentials.  In many of these cases we produced  perfect complexes, i.e.,  complexes with trivial differential, so that our computation determined the singular instanton homology despite the fact that the spectral sequence from Khovanov homology was not understood.
 
 The key perspective for the computations of \cite{HHK} is that the generators of the singular instanton chain groups can be interpreted as the intersections of two immersed 1-manifolds in a 2-dimensional orbifold, {\em the pillowcase}, which  arises as the quotient of the torus by the hyperelliptic involution (see Section \ref{RLITP} for more details).   This perspective results from the  observation that the choice of trivial 2-stranded tangle $(D,U)\subset (X,L)$ where we perform the perturbation results in a decomposition of the link
 $$  (X,L)  =   (X\setminus D, L\setminus U)  \cup_{(S^2,\{ 4\ \text{points}\})} (D,U)$$
We let $(Y,T):= (X\setminus D, L\setminus U)$ denote  the complementary tangle.  Now $C^\nat(X,L)$ is generated by certain conjugacy classes of  perturbed traceless representations  $\rho:\pi_1(X\setminus L)\rightarrow SU(2)$.  Using the above decomposition, they can be viewed as the intersection of the restrictions 
$$
\begin{diagram}\node{R(Y,T)}\arrow{se}\node[2]{R^\nat_\pi(D, U)}\arrow{sw}\\
\node[2]{R(S^2,\{4 \ \text{points}\})\simeq P}
\end{diagram}
$$
where $R(Y,T)$ denotes the traceless $SU(2)$ character variety of the complementary tangle, ${R^\nat_\pi(D, U)}$ denotes the perturbed traceless character variety of the trivial tangle (suitably summed with the Hopf link to allow for a non-trivial bundle), and  
 $R(S^2,\{4 \ \text{points}\})$ is the traceless character variety of the $4$-punctured 2-sphere where the tangles intersect.  The latter variety is  isomorphic to the pillowcase, which we denote by $P$ \cite[Proposition 3.1]{HHK}.  In \cite[Theorem 1]{HHK} we calculated the restriction map $L_0:{R^\nat_\pi(D, U)}\rightarrow P$, showing that for certain perturbation data its image was an immersed circle with exactly one double point, a ``figure eight".  Provided that the image of the restriction $L_1:R(Y,T) \rightarrow P$ is an immersed 1-manifold transverse to this figure eight, there is a bijection  between generators for $C^\nat(X,L)$ and intersections of the images of $L_0$ and $L_1$:
 $$  C^\nat(X,L)  =  \bigoplus_{x\in \{\text{Image}\ L_0\ \cap\  \text{Image}\ L_1\}} \mathbb{Z}/2 \langle x\rangle$$ 
    This perspective allows for the computation of $C^\nat(X,L)$ for an arbitrary 2-bridge or torus knot and for certain pretzel knots (see \cite[Sections 10 and 11]{HHK} and  \cite{FKP}). 

Though it made  progress towards our goal of making the singular instanton complexes more computable, the approach of \cite{HHK} had two serious drawbacks.  The first is that it was not clear whether, given a  link $(X,L)$, a trivial tangle $(D,U)\subset (X,L)$ can be found for which  $L_1: R(Y, T)\rightarrow P$ is an immersed 1-manifold transverse to the image of $L_0$.  The second is that even when such a tangle can be found, we had no way to compute the instanton differential  on the resulting chain group.  The  purpose of the present article is to address this second issue.  

A hint towards a possible understanding of the differential is gleaned by  viewing our setup through the lens of an ever growing body of conjectured or established relationships between gauge theoretic and symplectically defined Floer theories (e.g., \cite{Atiyah1,Atiyah2,Taubes,DS,wehrheim,WW,Duncan} discuss relationships between Yang-Mills gauge theory and symplectic invariants).  These relationships are often described as ``Atiyah-Floer Conjectures",  and  our description of $C^\nat(X,L)$ suggests looking for a differential on the instanton complex in terms of the symplectic geometry of the pillowcase.  In fact such a differential exists, and in the first half of this paper  we introduce an elementary $\ZZ/4$ relatively graded Lagrangian-Floer type chain complex  for appropriate  immersions  of   compact 1-manifolds ({\em restricted immersed Lagrangians}) into the pillowcase.  That is, we define  a complex generated by intersections of the images of immersed 1-manifolds, whose boundary operator is defined by counting the analog of holomorphic disks connecting them (in this low-dimensional setting, it will suffice to count orientation preserving immersions of disks into the pillowcase connecting intersections of $L_0$ and $L_1$).  In pursuit of our chain complexes, we draw liberally from the foundational work of Abouzaid \cite{abou}   on Floer homology for immersed Lagrangians in Riemann surfaces and de Silva-Robbin-Salamon \cite{SRV} for combinatorial and homotopy-theoretic aspects of Lagrangian Floer homology in this setting. Our work here can be viewed both as generalization and specialization of the existing literature, and our primary contribution is clarifying invariance proofs and properties of immersed Lagrangian Floer homology in the $\mathbb{Z}/4$ graded setting, and when some of the Lagrangians are immersed arcs (as opposed to circles).  The main result towards this end is Theorem \ref{rhi}, which can be paraphrased as follows.

\begin{theorem} Let $(L_0,L_1)$ be a pair of restricted  immersed $1$-manifolds in the pillowcase such that at least one of $L_i$ consists only of circles.  Then there is a well-defined  Floer homology group $HF(L_0,L_1)$ whose relatively $\mathbb{Z}/4$ graded isomorphism type depends only on the free homotopy type of the pair $(L_0,L_1)$.
\end{theorem}
 
 The second half of the article applies this construction to the situation described above, when $L_0:{R^\nat_\pi(D, U)}\rightarrow P$ is the restriction map from the perturbed traceless $SU(2)$ character variety of the trivial tangle and $L_1:R_\pi(Y,T)\rightarrow P$ is the restriction  of the perturbed traceless character variety of a  tangle $T$ in a homology 3-ball $Y$ (e.g.,  the complementary tangle to an embedded trivial tangle in a pair $(X,L)$ as above).  In favorable circumstances, such as when $(Y,T)$ is a certain tangle naturally associated to a 2-bridge or torus knot, the map $L_1:R_\pi(Y,T)\to P$ is a restricted Lagrangian without perturbations, so that the chain complex  $C^\nat(Y,T,\pi):=CF(L_0,L_1)$ and its homology $H^\nat(Y,T,\pi)$, which we refer to as the {\bf pillowcase homology}  of $(Y,T)$, are defined.     We then calculate the pillowcase homology for a number of examples, and show that it agrees with the singular instanton homology    in cases where the latter is known (e.g., 2-bridge knots and many torus knots).  More generally, our computations of pillowcase homology agree with conjectures for  $I^\nat(S^3,K)$  in many other cases.  We make the following Atiyah-Floer type conjecture:
 
 \medskip

\noindent{\bf Conjecture.} {\em Given a knot $K$ in a homology sphere $X$, there exists a 2-tangle decomposition $(X,K)=(Y,T)\cup (D,U)$   with  $(D,U)$ a trivial 2-tangle, and arbitrarily small  perturbations $\pi$, so that $L_1:R_\pi(Y,T)\to P$ is a restricted immersed 1-manifold for which the resulting pillowcase homology $H^\nat(Y,T,\pi)$ is isomorphic to $I^\nat(X,K)$ as a  $\mathbb{Z}/4$ relatively graded group.}

\medskip

We should note that the holonomy perturbations $\pi$ that we use to make the traceless character varieties regular and transverse are compatible with the perturbations used to make the moduli spaces regular in the construction of $I^\nat(S^3,K)$.   We should also note that while we have in some sense dealt with the second drawback from our first paper, in the sense that we have constructed a differential, the present results still leave us quite far from achieving our goal of computing $I^\nat(X,K)$.  Indeed, we do not yet know how to construct the general perturbations necessary to even define the pillowcase  homology (nor do we know how to pick the embedded trivial tangle $(D,U)$).  Moreover, even in the cases that we can find perturbations which make the pillowcase homology well-defined, it is not true that any such perturbations will yield a complex whose homology agrees with $I^\nat(X,K)$.

We hope to remedy  these concerns for the case of links in the 3-sphere in a subsequent article, which will  develop a spectral sequence from Khovanov homology to  pillowcase homology.  This will rely on picking a particular trivial tangle with which to apply our construction, and then iterating a skein exact triangle satisfied by pillowcase homology  in a similar manner to Kronheimer and Mrowka's construction.  This will rely on picking diagrammatically defined perturbation curves for the diagram of the complementary tangle $(Y,T)$.    Provided that we can construct this spectral sequence, it will have the advantage of providing an algebraic route to proving not only that pillowcase homology can be defined for arbitrary links in the 3-sphere, but also for showing that it is an invariant of the isotopy type of the link (in reality, we will   prove the stronger result that the quasi-isomorphism type of a certain twisted complex built from Lagrangian immersions  is a tangle invariant living in an appropriately defined Fukaya category of the pillowcase).   Moreover, it will provide a mechanism for calculating the higher differentials in the spectral sequence, since they will be combinatorially computable via the Riemann mapping theorem for polygons.  While much remains to be done to achieve this goal, we are given hope from the fact that we have already established the skein exact triangle for pillowcase homology in our work-in-progress.

\medskip

\noindent{\bf Outline:}  We now summarize the results of the article. In Section \ref{LFT} we recall  and extend work of Abouzaid \cite{abou} to construct a Lagrangian-Floer theory for curves   in surfaces, and outline the basic properties of the Maslov index.

In Section \ref{RLITP}
we recall the construction of the pillowcase $P$ as a quotient of $\RR^2$ by $\ZZ^2\ltimes \ZZ/2$. This is a 2-sphere with four singular points  {\em (corners)}.  Motivated by the main result of our previous article \cite{HHK}, we fix a  family $\{L^{\ep,g}_0\}_{\ep,g}$ of immersed circles with one double point in a certain regular homotopy class.   We define a {\em restricted immersed 1-manifold in $P$} (roughly)  to be an immersion $L_1:R\to P$, where either $R$ is  a circle and $L_1$ misses the corners of $P$, or $R$ is an arc with endpoints mapping to the corners; see Definition \ref{RLP}.   Choosing   $L_0=L_0^{\ep,g}$  transverse to a restricted lagrangian $L_1$, we then define a chain complex $(C^\nat(L_0,L_1),\partial)$   with differential $\partial$ determined by immersed bigons in the smooth part of $P$ with boundary lying on $L_0 $ and $L_1$,   following Floer \cite{Floer} and Abouzaid \cite{abou}.  
We show how to endow this complex with a  relative $\ZZ/4$ grading, a variant of an idea due to Seidel \cite{seidel1}.

In Section \ref{homotopy}, we prove that the resulting Floer homology $HF(L_0,L_1)$ depends only on the homotopy classes of the restricted immersed curves $L_0,L_1$ (the  result paraphrased as the theorem above). This result,  together with some basic observations described in Section \ref{special}, provides a set of tools to calculate $ HF(L_0,L_1)$. 

In Section \ref{2strand}, we apply this construction to traceless character varieties of knots.  We first recall  that the traceless character variety of the pair $(S^2,\{a,b,c,d\})$ is the pillowcase $P$. The main theorem of \cite{HHK} shows that if $(X,K)$ is a knot in a 3-manifold and $(S^2,\{a,b,c,d\})\subset (X,K)$ is a  2-sphere which separates $(X,K)$ into a trivial 2-tangle in the 3-ball $(D,U)$ and its complement $(Y,T)$, then generators of Kronheimer-Mrowka's reduced instanton knot complex can be identified with the intersection of $L_0^{\ep,g}$ and $L_1:R(Y,T)\to P$.   Hence, in favorable circumstances,  to such a decomposition  and an appropriate holonomy perturbation $\pi$ we can assign the corresponding Floer homology $HF(L_0^{\ep,g},L_1)$, which we denote by $H^\nat(Y,T,\pi)$.  This leads us to make the Atiyah-Floer conjecture stated above (see Conjectures \ref{con1} and \ref{con3} for more precise statements).

We show in Section \ref{2b} that this conjecture holds for 2-bridge knots (where all differentials are zero in both complexes). Sections \ref{general} and    \ref{chris}  establish   some general properties of $R_\pi(Y,T)$, such as identifying the two boundary points, showing they are stable under holonomy perturbations, and map to the corners of the pillowcase. We also examine the effect of applying holonomy perturbations in a collar neighborhood of the separating 2-sphere; in particular, this is used to make $L_0$ and $L_1$ transverse. 

In Sections \ref{torussec} and \ref{examples} we turn to calculations for torus knots, which display a rich and complicated collection of examples.  We find two appropriate perturbation curves in a useful tangle decomposition for any $(p,q)$ torus knot, and prove (Theorem \ref{torusknotsgood}) that there exist perturbations $\pi$ so that $R_\pi(Y,T)$ is a compact 1-manifold with two boundary points. 
We extend the work of \cite{FKP}, identifying $R_\pi(Y,T)$ and its image in the pillowcase,  to give many calculations of $H^\nat(Y,T,\pi)$ for tangles associated to torus knots. We give examples where different tangle decompositions and perturbations of a knot yield the same Lagrangian-Floer homology, which agrees with reduced instanton homology.   We give examples with non-zero differentials.   The reader is encouraged to examine Figures \ref{fig34} through 
\ref{fig47pert2} to get a feel for how calculations are carried out.

The upshot of our calculations is that the conjecture stated above holds for all the calculations of $H^\nat(Y,T)$ for which the  corresponding instanton homology $I^\nat(S^3,K)$ is known, and is consistent with the conjectured equality of ranks of $I^\nat(S^3,K)$ and the Heegaard knot Floer homology of $K$ when the instanton homology is unknown.

 \medskip

As an important final remark, we should say that while this article is rather lengthy we believe the results are quite natural and can be relatively easily understood through examples. Thus for the benefit of  the  reader we have
included a running example which is illustrated in Figures \ref{re1}, \ref{re2}, \ref{PP}, and \ref{re3},  and with details of the resulting calculations in Section \ref{examplecalc}. Understanding this example, and   its relation to the traceless character variety associated to a 2-tangle decomposition of the $(5,11)$ torus knot (Section \ref{5-11tk})  should      make our ideas quickly accessible.

 \bigskip
 
 \noindent{\bf Acknowledgements}  The first author gratefully acknowledges support from NSF grant DMS-0906258,  NSF CAREER grant DMS-1150872, and an Alfred P. Sloan Research Fellowship.  The third author gratefully acknowledges support from NSF  grant DMS-1007196 and Simons foundation collaboration grant  278714.
 
 \medskip
 
 The authors thank Juanita Pinzon-Caicedo  and Yoshihiro Fukumoto for helpful discussions and assistance with writing the computer programs to produce the data and figures used in Section \ref{examples}. They also thank Nikolai Saveliev, Matt Hogencamp for providing   very useful insights, and  Paul Seidel for pointing us to some  influential references.

 \section{Immersed Lagrangian Floer theory on a surface}\label{LFT}
 
 In \cite{abou}, Abouzaid constructs a Lagrangian-Floer theory for  {\em unobstructed immersed} curves in an oriented  surface.  In this section we recall his contruction, adapted slightly for our purposes. We also recall and relate various versions of the Maslov index for curves and $n$-gons in a surface equipped with a line field.

     \begin{figure}[h]
\begin{center}
\def\svgwidth{2.4in}
 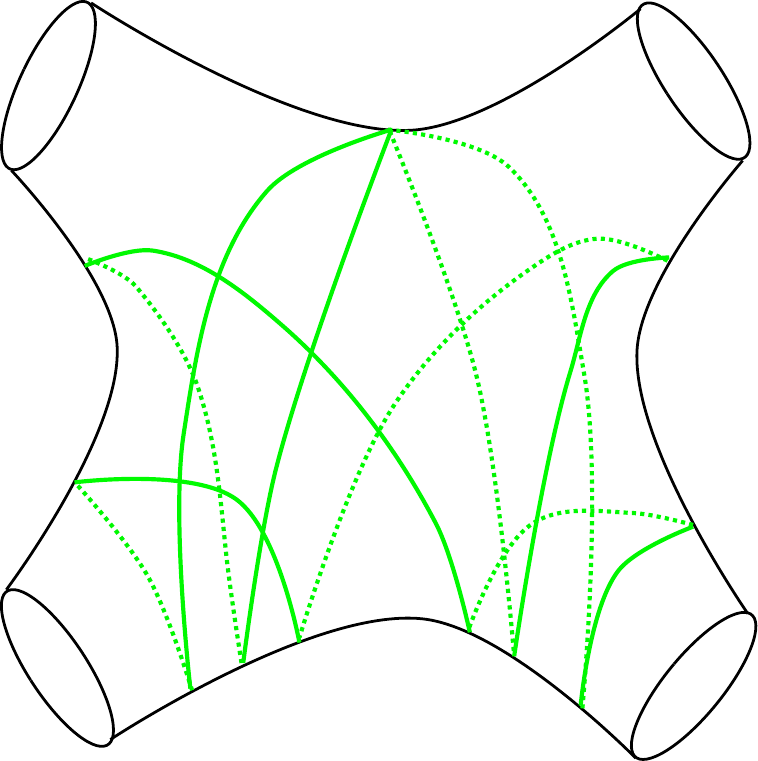
 \caption{An immersed circle $L_1$ in the 4-punctured 2-sphere.\label{re1} }
\end{center}
\end{figure}

\subsection{Unobstructed immersed curves}  \medskip
Let $S$ be a compact oriented surface, possibly with boundary, with infinite fundamental group.      
     

\begin{df}\label{unobstructed}
  An {\em unobstructed immersed arc}  is an immersion $L:[0,1]\to S$ which maps the endpoints to distinct points in the boundary of $S$, which is transverse to the boundary at its endpoints, and so that  some (and hence every) lift of $L$ to  the universal cover of $S$ is  embedded.

An {\em unobstructed immersed  circle} is an immersion of a circle $L:S^1\to S$ so that  each lift of the composite $\RR\xrightarrow{e}S^1\xrightarrow{L} S$ to the universal cover of $S$ is a properly embedded line.  Here $e(t)=\exp(2\pi\bbi t)$.

Either one of these is called an {\em unobstructed immersed curve}.
\end{df}

  An immersion of a  1-manifold $R$ to $S$    is said to {\em contain a fishtail} if there is an interval $I\subset R$ whose endpoints are sent to the same point in $S$ and so that the resulting loop is nullhomotopic.    Lemma 2.2 of \cite{abou} shows that an immersed circle is unobstructed if and only if it is homotopically essential and contains no fishtails.  Similarly, an immersed arc is unobstructed if and only if it contains no fishtails.

The reader can easily verify  that the immersed circle illustrated in Figure \ref{re1} is essential and contains no fishtails, hence is unobstructed.

\subsection{Intersection points} 
Let $L_0:R_0\to S$  and $L_1:R_1\to S$  be unobstructed immersed curves which  intersect transversely.   

\begin{df}\label{intpt}
Define an {\em intersection point} of $L_0$ and $L_1$ to be a pair $p=(r_0,r_1)\in R_0\times R_1$ where $L_0 (r_0)=L_1(r_1)$. 
\end{df}

By transversality and compactness, there are only finitely many intersection points.    
We will frequently abuse notation and write ``$L_0\cap L_1$''  for the set of intersection points of $L_0$ with $L_1$, or confuse $p=(r_0,r_1)$ with its image $L_0 (r_0)=L_1(r_1)$  in $S$. 

\bigskip

 Define 
$C(L_0,L_1)$ to be the  $\FF_2$ vector space generated by the  intersection points of $L_0$ and $L_1$.
\medskip

\subsection{Line fields and the Maslov index}  We recall some well known facts about the Maslov index  in the 2-dimensional setting  for the convenience of the reader, but also to make our conventions precise.

First, suppose that $\alpha:[0,1]\to S^1$ is a smooth map.  Define {\em the degree of $\alpha$}  to be the sum of the local degrees of $\alpha$  over the preimages of a regular value  $e^{\delta\bbi}$ for any small $\delta\ge0$ chosen so that $\alpha(0),\alpha(1)\not\in\{e^{t\delta\bbi}~|~t\in (  0,1]\}$: 
$$\deg(\alpha):=\sum_{x\in \alpha^{-1}(e^{\delta\bbi})}\deg_x(\alpha)$$
($\deg_x(\alpha)$ denotes the sign of $d\alpha$ at $x$).
Typically $\alpha(0)\ne 1\ne \alpha(1)$ so that we may take $\delta=0$. The integer $\deg(\alpha)$ has the property that it is additive under composition of paths, and invariant under homotopy of $\alpha$ rel endpoints.  Define the degree of a continuous path to be the degree of any smooth approximation with the same endpoints.

Next, suppose that ${\bf P}\to B$ is a principal circle bundle over a space $B$ and $\ell_0, \ell$ are two  sections.   The section $\ell$ defines a trivialization ${\bf P}\cong_{\ell} B\times S^1$  sending $\ell$ to $1\in S^1$. Given a path $\alpha:[0,1]\to B$,  define {\em the Maslov index of $\alpha$ with respect to $\ell_0$ and $\ell $}, denoted $\mu(\ell_0,\ell)_\alpha$, to be the degree of the composite 
$$ [0,1]\xrightarrow{\alpha}B\xrightarrow{\ell_0}\bf P\cong_{\ell} B\times S^1\xrightarrow{\pi}S^1.$$

\bigskip  
For the rest of this section we consider a pair $(S,\ell)$, where $S$ is an oriented   Riemannian  surface equipped with a line field $\ell$  (so $S$ is either a torus or else $S$ is not closed).    Formally, $\ell$ is a section of the projective tangent bundle ${\bf P}(T_*S)$, which we consider as a principal  $S^1=SO(2)/(\ZZ/2)$ bundle. Hence $\ell$ defines a trivialization ${\bf P}(T_*S)\cong_{\ell} S\times S^1$.
 
An immersion of a 1-manifold $R$ in $S$ defines a line field along $R$ and hence a Maslov index for any path in $R$. More precisely, given an immersion $L_0:R\to S$, the  subspaces $dL_0(T_{r}R)\subset T_{L_0(r)} S $ of $L$    define a section, which we denote  $\ell_0$, of the pullback bundle  $L_0^*({\bf P}(T_*S))\to R$.
Taking account of  the pullback of the line field $\ell$, this yields the composite map 
\begin{equation}
\label{mas1}
R\xrightarrow{\ell_0}L^*({\bf P}(T_*S))\cong_{\ell} R\times S^1\xrightarrow{\pi}S^1
\end{equation}

\begin{df}\label{Masdef} 
 Given an immersion $L_0:R\to S$ of a 1-manifold and a path $\alpha:[0,1]\to R$, the {\em Maslov index $\mu(L_0,\ell)_\alpha,$} is defined to be $\mu(\ell_0,\ell)_\alpha$, the  Maslov index  of $\alpha$ with respect to $\ell_0$ and $\ell$. When clear from context this will be denoted $\mu(\alpha,\ell)$  or simply $\mu(\alpha)$.
\end{df}

As   explicit examples (and to explain our conventions), if $S=\RR^2$, $\ell$ is the horizontal line field, and $R$ is the parabola $y=x^2$, then the path $\alpha:[-1,1]\to R$ given by $t\mapsto (t,t^2)$ satisfies $\mu(\alpha)=1$. If $\beta:[0,1]\to R$ is given by $\beta(t)=(t,t^2)$ and $\gamma:[-1,0]\to R $ is given by $\gamma(t)=(t,t^2)$, then  $\mu(\beta)=1$ and $\mu(\gamma)=0$. 

The basic properties of $\mu$, including its dependence on the choice of the background line field $\ell$ on $S$, are well known and easily understood using obstruction theory. We summarize the facts we need in the following Proposition. 

\begin{prop}\label{field}\hfill
 \begin{enumerate}
\item If $L_0:R\to S$ is an immersion of a 1-manifold and $\alpha, \beta:[0,1]\to R$ are continuous paths with $\alpha(1)=\beta(0)$, then $\mu(L_0,\ell)_{\alpha}$ and $\mu(L_0,\ell)_{\beta}$ depend  only on the homotopy classes of $\alpha $ and $\beta$ rel boundary. Moreover,
 $\mu(L_0,\ell)_{\alpha *\beta}=\mu(L_0,\ell)_{\alpha }+\mu(L_0,\ell)_{\beta}$. 
 \item If $R= S^1$ and $\alpha(t)=e^{2\pi \bbi t}$,  $0\leq t \leq 1$,  then $\mu(L_0,\ell)_\alpha$ is unchanged by a regular homotopy of $L_0$. More generally, if $R$ is any 1-manifold and $\alpha:[0,1]\to R$ arbitrary, then $\mu(L_0,\ell)_\alpha$ is unchanged by a regular homotopy of $L_0:R\to S$ which leaves $\alpha(0)$ and $\alpha(1)$  and the tangent spaces of $L_0$ at these points stationary.
 \item If $\ell'$ is any other line field, let $z\in [S,S^1]=H^1(S;\ZZ)=\Hom(H_1(S),\ZZ)$ denote the homotopy class of the difference map (i.e., $\ell'(s)=z(s)\ell(s)$).  Here,  we identify $S^1$ with $\RR{\bf P}^1$, so that one rotation corresponds to a rotation of a line through an angle of $\pi$.   If $\alpha:[0,1]\to R$ is a loop, then $\mu(L_0,\ell')_\alpha=\mu(L_0,\ell)_\alpha+z( L_0 \circ    \alpha)$. 
 \end{enumerate}

\end{prop}

\bigskip

We next define the triple index $\tau(\ell_0,\ell_1,\ell)$.  
\begin{df}\label{taudef} Suppose that $s\in S$.   Given a pair $\ell_0,\ell_1$ of transverse 1-dimensional subspaces of $T_sS$, let $\ell_t,t\in [0,1]$ be the {\em shortest clockwise path from $\ell_0$ to $\ell_1$}. Then define the {\em triple index}
$$\tau(\ell_0,\ell_1,\ell)=-\mu(\ell_t,\ell)_{[0,1]}\in\{0,1\}$$
\end{df}

When $\ell_0,\ell_1$ and $\ell$ are pairwise transverse,
 $\tau(\ell_0,\ell_1,\ell)$ is equal to $1$ if $\ell_0$ passes through $\ell$ when rotating $\ell_0$  negatively (clockwise) to $\ell_1$; otherwise  $\tau(\ell_0,\ell_1,\ell)=0$.  See Figure \ref{taufig}.

    \begin{figure}[h]
\begin{center}
\def\svgwidth{3in}
 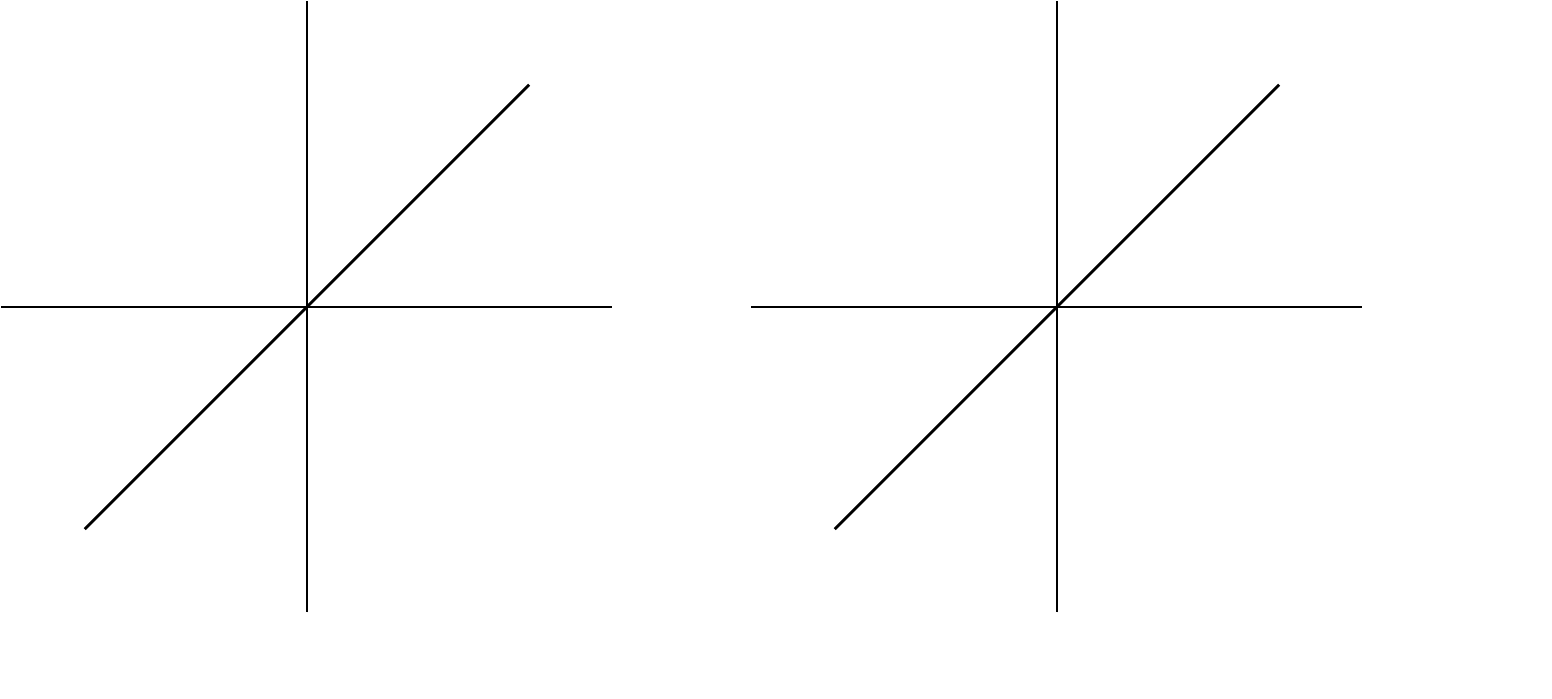
 \caption{ \label{taufig} }
\end{center}
\end{figure} 

The triple index has the following properties: if $\ell_0,\ell_1,\ell_2$ are pairwise transverse, then 
\begin{equation}\label{symtau}
\tau(\ell_1,\ell_2,\ell_0)= \tau(\ell_0, \ell_1,\ell_2) \text{ and } \tau(\ell_1,\ell_0,\ell)=1-\tau(\ell_0,\ell_1,\ell_2)
\end{equation}
\bigskip
  
 Let $L_k:R_k\to S$, $k=0,1,2,\dots, n-1$,  be a sequence of pairwise transverse  unobstructed immersed curves. Let $p_k$   be an intersection point of $L_{k-1}$ and $L_k$ for $k=1,\dots,n$.
 We consider   the indices cyclically ordered, so that $L_n=L_0$ and $p_0=p_n$.
 Figure \ref{5gonfig} illustrates the notation. 
  
     \begin{figure}[h]
\begin{center}
\def\svgwidth{3.3in}
 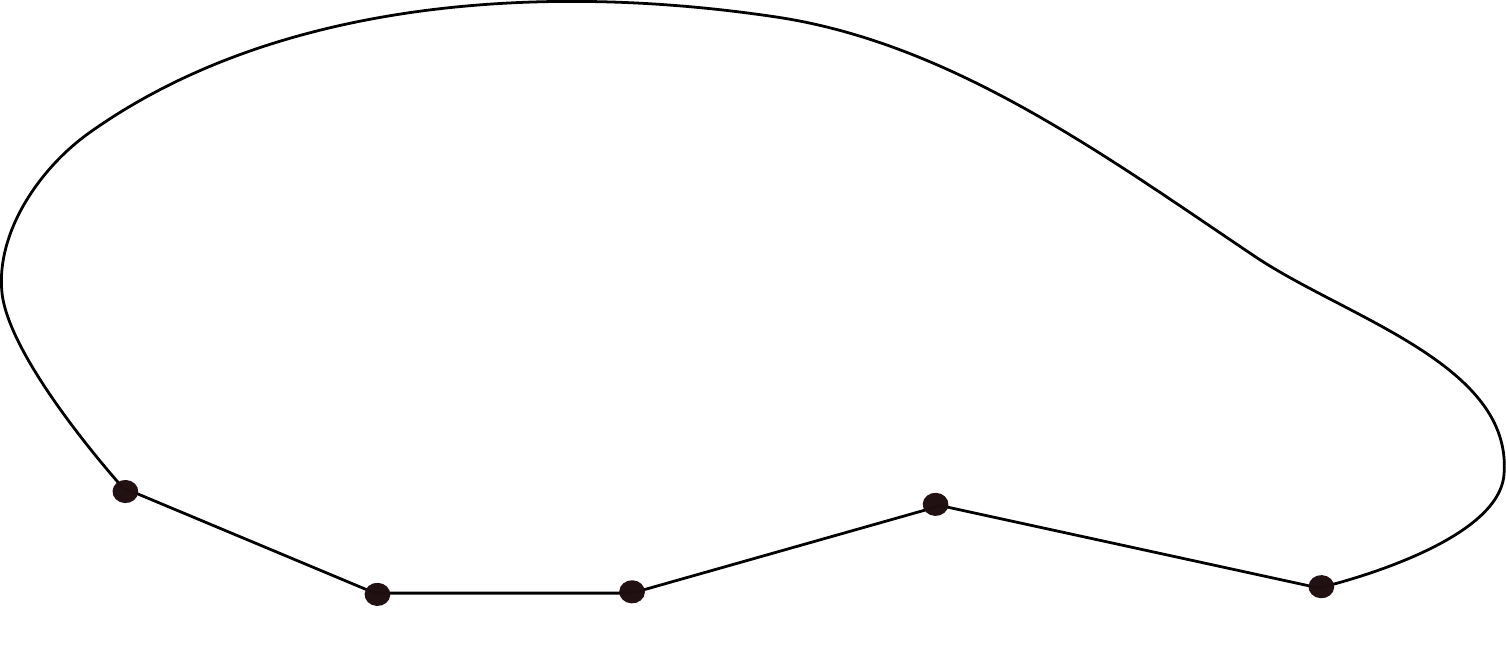
 \caption{A 5-gon\label{5gonfig} }
\end{center}
\end{figure}

\begin{df}\label{defpi1} 
Given an ordered list $L_0,L_1,\dots, L_{n-1}$ of pairwise transverse unobstructed immersed curves and a sequence of intersection points $p_n=p_0, p_1, \dots, p_{n-1}$ with $p_i\in L_{i-1} \cap L_i$,  define $\pi_1(p_1,\dots, p_{n})$ to be the  homotopy rel boundary classes of $n$-tuples of paths $(\gamma_0, \dots,\gamma_{n-1})$, where  $\gamma_k$ is a path in $R_k$ from (the first coordinate of the) intersection point $p_{k}$ to $p_{k+1}$.  
\end{df} 
The reader will notice that $p_n=p_0$ is placed last in the notation, in contrast to the order of the $L_i$ (and $\gamma_i$) where $L_0$ is placed first. This is motivated   by the fact that it is often useful to treat $p_0=p_n$   differently than the rest of the $p_k$, for example in   the definition of the differential  and, more generally, the $A_{\infty}$ structure on the Fukaya category of $S$ (see Definition \ref{munmaps}  below).  In particular, this ensures that  notation to be introduced below  for an $n$-gon from $(p_1,\dots,p_{n-1})$ to $p_n=p_0$ 
is as simple as possible.

\begin{df}\label{defpi2}  Define $\pi_2(p_1, \dots, p_{n}) $ to be homotopy classes of 
 pairs $(u,(\gamma_0, \dots,\gamma_{n-1 }))$  where $(\gamma_0,\dots,\gamma_{n-1})\in \pi_1(p_1,\dots, p_n)$, and $u:D^2\to S$ is a continuous map so that the (counterclockwise) boundary of $D^2$ is sent to the loop  
 $$\big(   \left( L_0 \circ \gamma_0\right )*\cdots*\left( L_{n-1}\circ \gamma_{n-1}\right)\big)^{-1}.$$
   There is a well defined forgetful map $\pi_2(p_1, \dots, p_{n})\to \pi_1(p_1, \dots, p_{n}) $ obtained by sending  the equivalence class of $(u,(\gamma_0, \dots,\gamma_{n-1}))$ to  that of  $ (\gamma_0, \dots,\gamma_{n-1})$. 
 \end{df}
 
 There is a well defined function 
 \begin{equation*}\label{nloop}
 \theta:\pi_1(p_1, \dots, p_{n})\to H_1(S), ~
\theta(\gamma_0,\dots, \gamma_{n-1})= -\left[\left( L_0 \circ \gamma_0\right)*\cdots*\left( L_{n-1} \circ\gamma_{n-1}\right)\right]
  \end{equation*}
which vanishes on the image of  the forgetful map $\pi_2(p_1, \dots, p_{n})\to \pi_1(p_1, \dots, p_{n}) $. 

We  use a ``label clockwise, but orient the boundary of a disk counterclockwise'' convention. This is why the inverse of the loop  $ L_0(\gamma_0)*\cdots*L_{n-1}(\gamma_{n-1})$  appears in Definition \ref{defpi2}, and that the negative sign   is used in the definition of $\theta$.

 Given an intersection point $p_k$ of $L_{k-1}$ and $L_k$, $\pi_2(p_k,p_k)$ is a group.  
If $\pi_2(p_1, \dots, p_{n})$  is non-empty, then $\pi_2(p_k,p_k)$ acts on $\pi_2(p_1, \dots, p_{n})$ by attaching $\tau\in \pi_2(p_k,p_k)$ to $\phi\in \pi_2(p_1, \dots, p_{n})$ along the vertex $p_k$ to form a new map of the disk $\tau\cdot\phi\in\pi_2(p_1, \dots, p_{n})$.  

In the case of a pair $(L_0,L_1)$, it is easy to check that $\pi_2(p_1,p_1)$ acts freely and transitively on $\pi_2(p_1,p_2)$ for any $p_2$ for which $\pi_2(p_1,p_2)$ is non-empty.

\medskip

\begin{df}\label{Masbi}
Given an ordered list $L_0,L_1,\dots, L_{n-1}$ of pairwise transverse unobstructed immersed curves,  intersection points $p_k\in L_{k-1}\cap L_k$,  and an $n$-tuple of paths $(\gamma_0,\dots, \gamma_{n-1})$ representing  a class in $\pi_1(p_1, \dots, p_{n})$, define  
 \begin{equation*}
\Mas_\ell(\gamma_0, \dots,\gamma_{n-1})
=1- \sum_{k=0}^{n-1} \big(\mu(L_k,\ell)_{\gamma_k} + \tau(L_{k},L_{k-1},\ell)_{p_{k}}\big)\end{equation*}
\end{df}

We describe a few equivalent formulas for $\Mas_\ell$.
First notice that letting $\alpha_k$ denote  the reverse of the path $\gamma_k$,
$$\alpha_k(t) =\gamma_k(1-t), $$
then $\mu(L_k,\ell)_{\gamma_k}=-\mu(L_k,\ell)_{\alpha_k}$.
Next, $-\tau(L_{k },L_{k-1},\ell)_{p_{k}}$ is equal to $\mu(M_{k}(t),\ell)_{[0,1]}$, where $M_{k}(t)$ is the path of lines in $T_{ p_{k}}S$ obtained as the shortest clockwise rotation of $T_{ p_{k} }L_{k} $ to $T_{ p_{k} }  L_{k-1} $.  Hence 
\begin{equation}
\label{goodeq}\Mas_\ell(\gamma_0, \dots,\gamma_{n-1})=1+\sum_{k=0}^{n-1} \big(\mu(L_k,\ell)_{\alpha_k} +\mu(M_{k}(t),\ell)_{[0,1]}\big) 
\end{equation}
 which, by path additivity of the Maslov index, is one greater than the Maslov index with respect to the line field $\ell$ of the continuous loop  in the projective tangent bundle of $S$: 
$$ \mathbb{M}:=TL_{n-1}|_{\alpha_{n-1}}*M_{n-1}*
TL_{n-2}|_{\alpha_{n-2}}*M_{n-2}*\cdots*TL_0|_{\alpha_0} *M_0,$$
 i.e., $$\Mas_\ell(\gamma_0, \dots,\gamma_{n-1})=1+\mu(\mathbb{M},\ell).$$

 As mentioned before, $p_0$ plays a special role, which motivates taking the  continuous loop obtained by  rotating  clockwise  at $p_k,k>0$, but rotating {\em counterclockwise} at $p_0$:
 \begin{equation*}
\mathbb{A}:=TL_{n-1}|_{\alpha_{n-1}}*M_{n-1}*
TL_{n-2}|_{\alpha_{n-2}}*M_{n-2}*\cdots*TL_0|_{\alpha_0} *N_0 \end{equation*}
where $N_{0}(t)$ is the shortest counterclockwise rotation of $T_{ p_{0} }L_{0} $ to $T_{ p_{0} }  L_{n-1} $ in $T_{p_0}S$.
Then using (\ref{symtau}), 
 \begin{equation}\label{Mattsversion}\Mas_\ell(\gamma_0, \dots,\gamma_{n-1})=  \mu(\mathbb{A},\ell),\end{equation}
justifying the use of the notation $\Mas_\ell$.

\begin{prop} \label{basicmas} Given an ordered set $L_0,\dots ,L_{n-1}$ of pairwise transverse unobstructed immersed curves and intersection points $p_k$ of $L_{k-1}$ and $L_k$,  then $\Mas_\ell$ defines a function $\pi_1(p_1,\dots,p_{n })\to \ZZ$.   If $\ell'$ is another line field and $z\in H^1(S)=[S,S^1]$ is the difference class, then 
$$\Mas_{\ell'}(\gamma_0, \dots,\gamma_{n-1})-\Mas_{\ell}(\gamma_0, \dots,\gamma_{n-1})=z(\theta(\gamma_0, \dots,\gamma_{n-1})).$$ In particular, $\Mas_\ell$ depends only on the homotopy class of $\ell$, and the composite
$$\pi_2(p_1,\dots,p_{n })\to \pi_1(p_1,\dots,p_{n})\xrightarrow{\Mas_\ell} \ZZ$$ is independent of the choice of line field $\ell$.

Furthermore,  $\Mas_\ell$ has the properties:
\begin{enumerate}
\item {\em (Splicing)}  If $q$ is another intersection point of $L_0$ and $L_k$, and $\gamma_0', \gamma_0'', \gamma_k', \gamma_k''$ are paths so that $\gamma_0'(1)=q=\gamma_0''(0)$, $\gamma_k'(1)=q=\gamma_k''(0)$, $\gamma_0=\gamma_0'*\gamma_0''$, and  $\gamma_k=\gamma_k'*\gamma_k''$, then 
$$\Mas_\ell(\gamma_0,\dots, \gamma_{n-1})=\Mas_\ell(\gamma_0'',\gamma_1,\dots, \gamma_{k-1},\gamma_k')+ \Mas_\ell(\gamma_k'',\gamma_{k+1},\dots, \gamma_{n-1}, \gamma_0').
 $$
 \item {\em (Path reversal)} Let $\alpha_k(t)=\gamma_k(1-t)$. Then
 $$\Mas_\ell(\alpha_{n-1},\alpha_{n-2},\dots, \alpha_0)= 2-n-\Mas_\ell (\gamma_0,\gamma_1,\dots,\gamma_{n-1}). 
 $$
\item {\em (Cyclic invariance)} $\Mas_\ell (\gamma_0,\gamma_1,\dots,\gamma_{n-1})=\Mas_\ell(\gamma_1,\dots,\gamma_{n-1},\gamma_0)$.
\end{enumerate}

 \end{prop}
\begin{proof}  The homotopy invariance property of the Maslov index $\mu$ and   Equation (\ref{Mattsversion}) shows that   $\Mas_\ell(\gamma_0, \dots,\gamma_{n-1})$ depends only on the homotopy class of $\ell$ and the class  of
$(\gamma_0, \dots,\gamma_{n-1})$ in $\pi_1(p_1,\dots, p_{n})$.

If $\ell, \ell'$ are two different line fields and $z\in H^1(S)$ is their difference class, Proposition \ref{field} gives 
$$ \Mas_{\ell'} (\gamma_0, \dots,\gamma_{n-1})-\Mas_{\ell}(\gamma_0, \dots,\gamma_{n-1}) = \mu(\mathbb{A},\ell')-\mu(\mathbb{A},\ell)=z(\theta( \gamma_0, \dots,\gamma_{n-1})).$$
If $(\gamma_0, \dots,\gamma_{n-1})$ lies in the image of  $\pi_2(p_1,\dots, p_{n})\to \pi_1(p_1,\dots, p_{n})$, then $\theta( \gamma_0, \dots,\gamma_{n-1})=0$ in $H_1(S)$, and hence  $\Mas_{\ell} (\gamma_0, \dots,\gamma_{n-1})$ is independent of $\ell$.

The three properties are easily checked using path additivity of   $\mu(-,\ell)$ and the identity $\tau(\ell_0,\ell_1,\ell)+\tau(\ell_1,\ell_0,\ell)=1$. \end{proof}

\subsection{Immersed polygons}
Define an {\em $n$-gon in $\RR^2$} to be a pair $(D, (\beta_0,\beta_1,\dots, \beta_{n-1}))$ where $D\subset \RR^2$ is a closed topological disc, and  $(\beta_0,\beta_1,\dots, \beta_{n-1})$ is a sequence  of smooth embeddings of the unit interval in $\RR^2$ with image in the boundary of $D$ so that  
\begin{enumerate}
\item $\beta_k(1)=\beta_{k+1}(0)$ for $k=0,\dots, n$ (with $\beta_n=\beta_0$),
\item the composite path $\beta_0*\cdots*\beta_{n-1}$ forms an embedded simple closed curve in $\RR^2$, which forms the {\em clockwise}  boundary of $D$,  
\item the $\beta_k$ meet transversely at their endpoints. 
\end{enumerate}

\begin{df}\label{ngondef}  Given a pair $(S,  (L_0,\dots ,L_{n-1}))$, where $S$ is an oriented surface and $L_k:R_k\to S$, $k=0,\dots, n-1$ are pairwise transverse immersions of 1-manifolds into $S$,
define an  {\em an immersed $n$-gon in $S$ for the ordered $n$-tuple $(L_0,\dots,L_{n-1})$ through the points $(p_1,\dots, ,p_{n})$} to be   a triple consisting of 
\begin{enumerate}
\item an $n$-gon in $\RR^2$, $(D, (\beta_0,\beta_1,\dots, \beta_{n-1}))$,
\item a representative $n$-tuple of paths $(\gamma_0,\dots,\gamma_{n-1})$ for a class in $\pi_1(p_1,\dots,p_{n})$,
\item an orientation preserving immersion $u:D\to S$ satisfying $u\circ \beta_k=L_k\circ \gamma_k$.
\end{enumerate}
We  use the brief notation   $u$ for the triple   $$((D,(\beta_0,\dots,\beta_{n-1})),(\gamma_0,\dots,\gamma_{n-1}), u:D\to S),$$
and we call 
this {\em an immersed $n$-gon for the ordered $n$-tuple $(L_0,\dots,L_{n-1})$ from $(p_1,\dots, p_{n-1})$ to $p_0=p_n$.}

Define the {\em Maslov index of an immersed $n$-gon in $S$} by
 $$\Mas(u)=\Mas_\ell( \gamma_0,\dots,\gamma_{n-1})$$ for any line field $\ell$. Since an immersed $n$-gon represents an element of $\pi_2(p_1,\dots,p_{n})$, it follows from Proposition \ref{basicmas} that $\Mas(u)$ is independent of the choice of line field $\ell$.

\end{df}

 Figure \ref{conventionfig} indicates the corresponding model examples of 2-gons and 3-gons in $S=\RR^2$:

    \begin{figure}[h]
\begin{center}
\def\svgwidth{3.9in}
 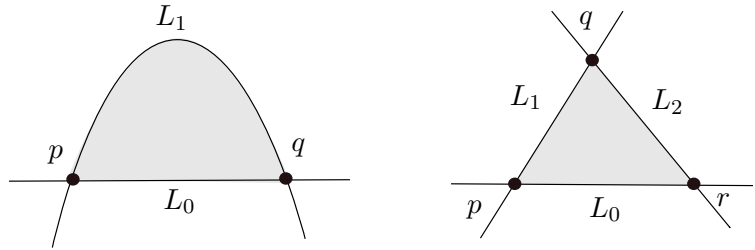
 \caption{A 2-gon of  Maslov index 1 from $p$ to $q$ for the ordered pair $(L_0,L_1)$ and a 3-gon of Maslov index 0   from $(p,q)$ to $r$ for the ordered triple $(L_0,L_1,L_2)$.\label{conventionfig} }
\end{center}
\end{figure}

Notice that an $n$-gon in $\RR^2$ is just a special (embedded) case of an immersed $n$-gon in the surface $S=\RR^2$. As such, its Maslov index can be easily computed by taking, for example, any line field of constant slope. For example, the 5-gon of Figure \ref{5gonfig} has Maslov index $-1$.  One can see this by taking $\ell$ to be the vertical  line field. Then $\mu(L_k,\ell)_{\gamma_k}=0$ for $i=0,1,2,3$ and $\mu(L_4,\ell)_{\gamma_4}=-2.$  Also $\tau(L_{k},L_{k-1},\ell)_{p_k}$ equals 1 for $k=0,2,3,4$ and $\tau(L_{1},L_0,\ell)_{p_1}=0$, so $\Mas=1-(-2+4)=-1$.  This calculation easily generalizes to yield the following proposition.
\begin{prop} \label{convex}Let $(D,\beta_0,\dots,\beta_{n-1})$ be an $n$-gon in $\RR^2$. Let $\kappa(D)\in\{0,\dots, n\}$ denote the number of non-convex corners of $D$.

Then any immersed $n$-gon   $u:D\to S$  satisfies
$$\Mas(u)=3-n+\kappa(D).
$$
 \qed
\end{prop}

\subsection{The differential and $A_n$ maps.} 
We recall  the definitions of the differential and maps $\mu_n$ next. Our goal is to relate  a Lagrangian-Floer theory to singular instanton homology, which is  the homology of a chain complex rather than a cochain complex, and hence our orientation conventions differ slightly from the similar constructions in the literature which typically produce cochain complexes.

\begin{df}\label{eqngons} Fix an ordered $n$-tuple $(L_0,L_1,\dots, L_{n-1})$ of distinct unobstructed pairwise transverse immersed curves in $S$ and intersection points $p_k$ of $ L_{k-1}$ and $L_k$.

  Two immersed $n$-gons in $S$ through  $(p_1,\dots, p_{n})$, $u:D\to S, u':D'\to S$ are called {\em equivalent} if there is an  orientation preserving   diffeomorphism $\psi:D\to D'$ so that $u=u'\circ\psi$. 
  
  The set of equivalence classes of immersed $n$-gons with Maslov index $3-n$   for the ordered $n$-tuple $(L_0,\dots,L_{n-1})$   through $(p_1,\dots, p_{n})$  is denoted by $\mathcal{M}_{L_0,\dots,L_{n-1}}( p_1,\dots,  p_{n})$, or simply  by $\mathcal{M}( p_1,\dots,  p_{n})$  when the order of the $L_k$ is clear from context.   \end{df}
  
  When $n\ge 3$, the list $(p_1,\dots,p_{n})$ determines the order $(L_0,\dots, L_{n-1})$.    Cyclically permuting the $n$-tuple $(L_0,\dots, L_{n-1})$ and the points $(p_1,\dots, p_{n})$  preserves immersed bigons and the Maslov index, and hence 
  $$\mathcal{M}_{L_0,L_1,\dots,L_{n-1}}( p_1,\dots,  p_{n})=\mathcal{M}_{L_1,\dots,L_{n-1},L_0}( p_2,\dots,  p_{n},p_1).$$

When $n=2$, care must be taken with the ordering since the ordered pair $(p_0,p_1)$ does not determine the order of $L_0,L_1$. In particular, $\mathcal{M}_{L_0,L_1}(p,q)=\mathcal{M}_{L_1,L_0}(q,p)$, but these are different from $\mathcal{M}_{L_1,L_0}(p,q)=\mathcal{M}_{L_0,L_1}(q,p)$.

\bigskip
Given a finite set $X$, let  $\# X\in \FF_2$ denote the number of elements of $X$ mod 2. 
Recall that $C(L_0,L_1)$ is defined to be the free $\FF_2$ vector space on the intersection points of $L_0$ and $L_1$.

\begin{df}\label{munmaps} Fix an ordered $n$-tuple $(L_0,\dots, L_{n-1})$ of unobstructed pairwise transverse curves in $S$.   Suppose $n-1$ intersection points $p_k$ of $ L_{k-1}$ and $L_k$ are given for $k=1,\dots,n-1$, with the property that for every intersection point  $q$ of $L_0$ and $L_{n-1}$, 
$\mathcal{M}_{L_0,\dots,L_{n-1}}( p_1,\dots,  p_{n-1},q)$ is finite.  

 Define
 $$\mu_{n-1}(p_1, \dots, p_{n-1})=\sum_{q\in L_{0}\cap L_{n-1}}
\big(\# \mathcal{M}_{L_0,\dots ,L_{n-1}}( p_1,\dots,p_{n-1},q)\big)~ q \ \text{~in~} C(L_{0},L_{n-1}).$$

If $\mathcal{M}_{L_0,\dots,L_{n-1}}(p_1,\dots,  p_{n})$ is finite for all choices of intersection points $(p_1,\dots,  p_{n-1})$ {\em and} $p_n$, then $\mu_{n-1}$ defines a linear map:
$$\mu_{n-1}:C(L_{0},L_{1})\otimes C(L_{1},L_{2})\otimes\cdots\otimes C(L_{n-2},L_{n-1})\to C(L_{0},L_{n-1}).$$
 \end{df}

  Most important for us is the map $\mu_1$, which we also denote by $\partial$.  Explicitly
  \begin{equation}
\label{diff} \partial:C(L_0,L_1)\to C(L_0,L_1),~\partial p=\sum_{q\in L_0\cap L_1} \#\mathcal{M}_{L_0,L_1}(p,q)~ q
\end{equation}
Recall that we call  representatives of $\mathcal{M}_{L_0,L_1}(p,q)$ 2-gons from $p$ to $q$ for $(L_0,L_1)$, so $\partial p$ is the linear combination of intersection points $q$ weighted by the mod 2 count of  2-gons with   Maslov index 1 from $p$ to $q$.  

Since they occur frequently, we call a 2-gon with Maslov index 1 a {\em bigon}.
Note that a bigon from $p$ to $q$ for the ordered pair $(L_0,L_1)$ is also a bigon from $q$ to $p$ for $(L_1,L_0)$, but is not a bigon from $q$ to $p$ for $(L_0,L_1)$.    See the paragraph following Definition \ref{eqngons}.

\medskip

Figure \ref{re2}  shows the   curve  $L_1$ of Figure \ref{re1} and another unobstructed curve $L_0$. These intersect in eight points.  Two intersection points $p,q$ are indicated, and the image of a bigon from $p$ to $q$ for the pair $(L_0,L_1)$ is shaded. The reader should check that there is precisely one other bigon between $L_0$ and $L_1$, joining a different pair of intersection points.
      \begin{figure}[h]
\begin{center}
\def\svgwidth{2.4in}
 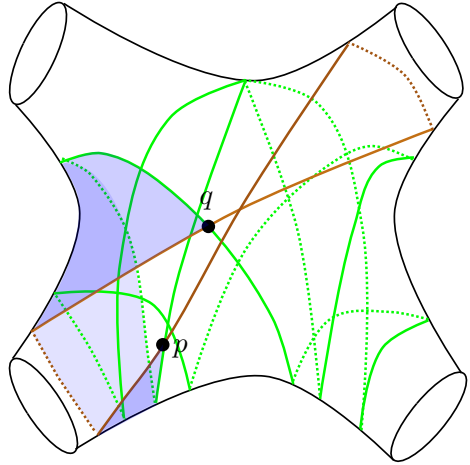
 \caption{A bigon from $p$ to $q$.\label{re2} }
\end{center}
\end{figure}

 \medskip

  The following is proved in  Abouzaid's article \cite{abou}.  In that article, coefficients are taken in a Novikov ring (over $\ZZ$) to account for the possibility that $\mathcal{M}_{L_0,L_1}(q,p)$ is infinite. Since this will not be the case in our applications, we set the  Novikov variable $t$ equal to 1 and reduce the coefficients from $\ZZ$ to $ \FF_2$. 

\begin{thm}[Abouzaid \cite{abou}]  \label{abouza} Let $(L_0,L_1)$ be a pair of unobstructed transverse immersed curves in $S$, and assume that $\mathcal{M}_{L_0,L_1}(p,q)$ is finite for all intersection points $p,q$ of $L_0$ and $L_1$. Then  $\partial :C(L_0, L_1)\to C(L_0,L_1)$  satisfies $\partial^2=0$.
\qed
\end{thm}

For example, the chain complex $C(L_0,L_1)$ for the pair $(L_0,L_1)$ illustrated in Figure \ref{re2} (see also Figures \ref{PP} and \ref{re3})  is generated by the eight intersection points of $L_0$ and $L_1$. Two bigons (one of which is shaded in Figure  \ref{re2}) define a non-trivial differential $\partial$  of rank   $2$. The resulting homology has rank four.

\bigskip

More generally, Abouzaid   proves in \cite{abou} that the $\mu_n,~n\ge 2$ satisfy the {\em $A_n$ relations} for all $n$ when $(L_0,\dots, L_{n-1})$ are pairwise transverse unobstructed immersed curves with no triple points. 

We will only use the $A_2$  and $A_3$ relations in the present article and so we write them out explicitly. (We refer to  \cite{abou, Auroux1} for the formulas for the $A_n$ relations.)
The $A_2$ relation says
\begin{equation}
\label{mu2rel}\mu_2(\mu_1(x),y)+\mu_2(x,\mu_1(y))+\mu_1(\mu_2(x,y))=0,\end{equation} 
and the $A_3$ relation says
\begin{equation}
\label{mu3rel}
\begin{multlined}
  \mu_3(\mu_1(x),y,z)+\mu_3(x,\mu_1(y),z)+\mu_3(x,y,\mu_1(z)))  +\mu_2(\mu_2(x,y),z)\\
  +\mu_2(x,\mu_2(y,z))
  +\mu_1(\mu_3(x,y,z))=0\qquad\end{multlined}
\end{equation}
These hold when all the sets  of equivalence classes of $n$-gons of Maslov index $3-n$, $n=2,3,4$ which appear in the formulas defining each term in Equation (\ref{mu2rel}) or (\ref{mu3rel}) are finite.

\section{Restricted Lagrangians in the pillowcase}\label{RLITP}
 In this section, we will apply the constructions of Section \ref{LFT} to the pillowcase.

 \subsection{The pillowcase}
 
 The {\em pillowcase} $P$ is the quotient of the torus by the hyperelliptic involution. It is a topological 2-sphere with four singular points corresponding to the four fixed points of the involution. For concreteness, define $P$ to be the  quotient of $\RR^2$ by the group of orientation preserving isometries generated by the maps 
 $$(\gamma,\theta)\mapsto (\gamma+2\pi,\theta),~(\gamma,\theta)\mapsto (\gamma,\theta+2\pi),~(\gamma,\theta)\mapsto (-\gamma,-\theta)$$  (this group is a semi-direct product $\ZZ^2\ltimes \ZZ/2$).  
 The quotient map is a branched covering
 \begin{equation}
\label{brcover} \RR^2\to P.
\end{equation}
A fundamental domain for the action is  given by the rectangle $(\gamma,\theta)\in [0,\pi]\times [0,2\pi]$.     We will frequently specify  a point  in $P$ by giving its coordinates $(\gamma,\theta)\in \RR^2$. We refer to  points in $(\pi\ZZ)^2$ as {\em lattice points}.     The four singular points  of $P$, which we call the corners, make up the image of the lattice points.  
Our theory will take place in the complement of the corners, so it is convenient to adopt the notation $P^*=\left( \RR^2 \setminus (\pi \ZZ)^2 \right)/ (\ZZ^2\ltimes \ZZ/2 )$.
Note that $P^*$ inherits an orientation and a symplectic structure  from the standard orientation and symplectic structure $d\gamma\wedge d\theta$  on $\RR^2 \setminus (\pi \ZZ)^2$ via the branched covering (\ref{brcover}).

The pillowcase $P$  is illustrated in two ways in Figure \ref{PP}. One should view the figure on the left as obtained by folding the fundamental domain  $[0,\pi]\times[0,2\pi]$  for the branched cover  (\ref{brcover}), illustrated on the right, along $ [0,\pi]\times \{\pi\}$ and making  identifications along the edges as indicated. The front face is the image of $[0,\pi]\times[0,\pi]$ and the back face is the image of $[0,\pi]\times[\pi,2\pi]$, upside down.  In Figure \ref{PP} we have also indicated the immersed circle $L_1$ of Figure \ref{re1}.

    \begin{figure}[h]
\begin{center}
\def\svgwidth{3.5in}
 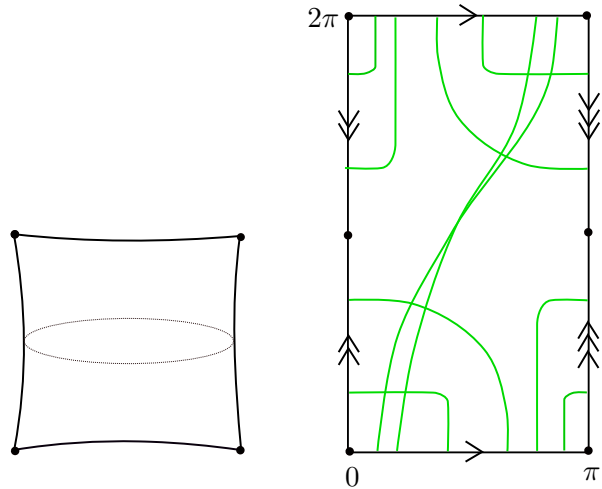
 \caption{Two depictions of the pillowcase $P$. The immersed circle $L_1$ of Figure \ref{re1} is indicated on the right.\label{PP} }
\end{center}
\end{figure}

 \subsection{A line field on the pillowcase}  To apply the Maslov index constructions  described in Section \ref{LFT}, a line field on $P^*$ is needed.    We fix a line field $\ell_{\operatorname{inst}}$ in a particular homotopy class, so that the  $\ZZ/4$ grading we construct below using  the Maslov index matches the $\ZZ/4$ grading on singular instanton knot homology.  The connection to   gauge theory is explained in Section \ref{2strand}.   
 
 The line field $\ell_{\operatorname{inst}}$ is somewhat complicated to depict or calculate with, as it twists along the edges of the pillowcase.   Our approach in calculations is to keep track of  the mod 4 Maslov index information determined by $\ell_{\operatorname{inst}}$ by instead using a pair $(\ell,z)$, where $\ell$ is a  simple (constant slope) line field and $z\in H^1(P^*;\ZZ/4)$ keeps track of the extra twisting of $\ell_{\operatorname{inst}}$  relative to $\ell$, as described in Proposition \ref{field}.
 
 Any  constant slope line field  on $\RR^2$ is invariant under the $\ZZ^2\ltimes\ZZ/2$ action and hence its restriction 
  to $\RR^2\setminus(\pi\ZZ)^2$ descends to a line field  on $P^*$.  Call such a line field a {\em  constant slope} line field on $P^*$.  We will make frequent use of 
 the slope one line field on $\RR^2$, and hence we give it the label $\ell_1$.

%
%

\begin{df}\label{theclassz}
 Let $z\in H^1(P^*;\ZZ/4)=\Hom(H_1(P^*),\ZZ/4)$ denote the unique cohomology class which assigns $1\in \ZZ/4$ to each small loop circling a corner counterclockwise.  

\end{df}
  
 Immersed circles $\gamma:S^1\to P^*$ satisfy $\mu(\gamma,\ell_1)\equiv z(\gamma)$ mod 2.  They need not be equal modulo 4, however.   For example, if $\gamma$ is  the boundary of a smoothly embedded disk in $P^*$, then $\mu(\gamma,\ell_1)=\pm2$ and $z(\gamma)=0$.  For a small embedded loop encircling one corner of $P$ counterclockwise,  $\mu(\gamma,\ell_1)=1=z(\gamma)$. For the curve  $L_1$ illustrated in Figure  \ref{re1} and the right in Figure \ref{PP}, $\mu(L_1,\ell_1)=0=z(L_1)$.  For the curves $L_0^{\ep,g}$ depicted in Figure \ref{Lzero}, $\mu(L_0^{\ep,g}, \ell_1)=0=z(L_0^{\ep,g})$.

\medskip

\begin{df}\label{instanton line field}    Fix a  map $\tilde z:P^*\to S^1$ so that its class $\tilde z\in H^1(P^*;\ZZ)=[P^*, S^1]$  is a lift of the class $z$.  Think of $S^1$ as acting freely and  transitively on lines in $\RR^2$.  Define the {\em instanton line field} by:
\begin{equation}
\label{ilf} 
\ell_{\operatorname{inst}}(p)=\tilde z(p)  \ell_1(p).
\end{equation}
 
\end{df}
 
The  homotopy class of the line field $\ell_{\operatorname{inst}}$ depends on choice of lift $\tilde z$, as do Maslov indices computed using $\ell_{\operatorname{inst}}$, however, \begin{equation}\label{ivs1}
\mu(L,\ell_{\operatorname{inst}})\equiv \mu(L,\ell_1)+z(L)\mod{4}.
\end{equation}

 \subsection{Proper arcs in $P$}

\begin{df}\label{pimm}
Define a {\em proper  immersion}   of an interval $L:I\to P$ to be the image under the branched cover (\ref{brcover}) of a smooth immersion $\widetilde L:I\to \RR^2$ which  takes the two endpoints of the interval  to $(\pi\ZZ)^2$ and the interior  to $\RR^2\setminus (\pi\ZZ)^2$.  We call the slopes of $\widetilde L$ at the  endpoints  (which  are determined by $L$) the {\em limiting slopes of $L$}.   Note that a proper immersion cannot spiral infinitely many times as it limits to a corner. \end{df}

\medskip

In order to easily apply the results of \cite{abou}, it is convenient to work in a compact surface with boundary.   It will suffice for our purposes to simply remove a   small neighborhood  of the corners. More precisely, given some small $\delta>0$, let $\barP_\delta\subset P$ denote the image under the branched cover (\ref{brcover}) of the subspace of $\RR^2$ obtained by removing open $\delta$ neighborhoods of the lattice points.

If $R$ is a compact 1-manifold with boundary 
and $L_1:R\to P$ is  a proper immersion (as defined above) on each arc, and maps each circle of $R$ into $P^*$, then for  $\delta>0$  small enough so that the $\delta$ disks miss the circle components,  $L_1(R)\cap \barP_\delta$ is a properly immersed compact 1-manifold in the compact surface $\barP_\delta$.  In the following, we will typically write $\barP$ instead of $\barP_\delta$, with the understanding that $\delta$ is chosen small enough to  miss circle components and result in  arc components transverse to the boundary.

\subsection{Perturbation functions and a family of isotopies of $P$}

Let $$\mathcal{X}=\{ f\in C^\infty(\RR,\RR)~|~ f(x+2\pi)=f(x), f\text{ odd}\}.$$ We call this the {\em space of perturbation functions}.  It is a vector space, and is preserved by pre-composition by $x\mapsto x+\pi$. In particular, $f(\pi)=0$ for all $f\in \mathcal{X}$.  The sine function   is a member of  $\mathcal{X}$.

 The usual terminology  in the literature describes perturbation data as a choice of an embedded solid torus and a  conjugation invariant function on $SU(2)$, which together are used to define a gauge invariant perturbation of the   Chern-Simons functional.  In our notation, an element $f\in \mathcal{X}$ is the derivative of such a conjugation invariant function on $SU(2)$, restricted to the maximal torus.  The function $f$ determines the effect on the critical set of Chern-Simons function.  More precisely, $f$ determines which flat connections on the complement of  the perturbation solid torus extend to be perturbed flat on the solid torus, so it is more convenient for us to refer to these functions in our pertubation data.  See Section \ref{perturbationsection} for more details.

We associate,  to each  perturbation function $g\in \mathcal{X}$, an isotopy  of the pillowcase by: 
\begin{equation}
\label{iso}c_{g}:P\times I\to P,~c_{g}((\gamma,\theta),s)=(\gamma,\theta+sg(\gamma)). 
\end{equation}
 Since $c_{-g}(c_g(p,s),s)=p$, $c_g(-,s)$ is a homeomorphism and hence $c_g$ is an isotopy starting at the identity. Notice that $c_g$ fixes the left and right edges of the pillowcase.

The formula (\ref{iso}) shows that $c_g$ lifts to a Hamiltonian isotopy of $\RR^2$ which is $\ZZ \ltimes\ZZ/2$ invariant  and fixes the vertical lines $\{x=n\pi\}$.  In particular, we can think of $c_g$ as a Hamiltonian isotopy of $P^*$, or of the orbifold $P$.

\subsection{The family $L_0^{\ep,g}$ of immersed circles in the pillowcase}
In the applications to singular instanton homology in Section \ref{2strand}, we show that   a 2-tangle decomposition of a knot gives rise to two unobstructed immersed curves $L_0,L_1$ in $P^*$,  which in turn define a chain complex $C(L_0,L_1)$ as in Section \ref{LFT}.  Identification of the  immersed circle $L_0$, which depends on a parameter $\epsilon\ne 0$, was accomplished in \cite[Theorem 7.1]{HHK}. In order to ensure we can choose $L_0$ transverse to $L_1$ we enlarge the family of $L_0$ to include the isotopies described above. 

\bigskip

 Let $\Delta\subset P$ denote the arc of slope one, i.e., the diagonal arc  
\begin{equation}
\label{Delta}
\Delta=\{(\gamma,\gamma)~|~\gamma\in[0,\pi]\}.
\end{equation}

\begin{df} \label{elzero} Fix an $\epsilon>0$ and $g\in \mathcal{X}$. Let $L_0^{\ep, g}:S^1\to P^*$ denote the immersion given as the composite of the map
\begin{equation}\label{tilfixed}\tilde L_0^{\ep,g}(t)=
 (t+ \epsilon \sin(t)+ \tfrac\pi 2, t- \epsilon  \sin(t)+ \tfrac\pi 2+g(t+\epsilon\sin(t)+\tfrac{\pi} 2) ),~ t\in [0,2\pi]\end{equation}
 with the branched cover $\RR^2\to P$  of Equation (\ref{brcover}), so 
 \begin{equation}\label{fixed}
L_0^{\ep,g}:[0,2\pi]\xrightarrow{\tilde L_0^{\ep,g}}\RR^2\to P^*.
\end{equation}

 \end{df}

The   image of $L_0^{\ep, g}$ in $P^*$ for $g=0$ and $\ep$ small is illustrated in Figure \ref{Lzero} and also in Figure \ref{re2}.  As $\ep$ and $g$ approach zero, $L_0^{\ep, g}$ limits to a generically 2-1 map onto $\Delta$, with two points mapping to corners.  
\medskip
 
    \begin{figure}[h]
\begin{center}
\def\svgwidth{1.6in}
 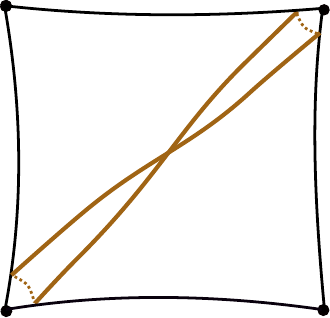
 \caption{The curve $L_0^{\ep,0}$ in $P$.\label{Lzero} }
\end{center}
\end{figure}

Note that  $L_0^{\ep ,g}(t)=c_g(L_0^{\ep,0}(t),1)$, so that  $L_0^{\ep ,g}$ is isotopic to $L_0^{\ep,0}$. 
In particular, the family of immersions $L_0^{\ep, g}$  for $\epsilon>0$  small  are self-transverse with one double point. Furthermore, $L_0^{\ep, g}$   is an unobstructed circle in the sense of Definition \ref{unobstructed}.
 
 \medskip
The following easily proved genericity lemma  says that  given any unobstructed immersed curve $L_1$,   arbitrarily small  $\ep,g$ can be found so that $L_1$ and $L_0^{\ep,g}$ are transverse.

\begin{lem}\label{generic} Given an unobstructed immersed circle or arc $L_1$, there exist $\ep>0$ and $\delta>0$ arbitrarily close to zero so that,  with $g(x)=\delta\sin(x)$, 
 $L_0^{\ep,g}$ and $L_1$
are transverse.  \qed
\end{lem}
 
 \color{black}
\subsection{Restricted immersed arcs and circles}

 To complete the construction of a $\ZZ/4$ relatively graded chain complex we refine the notion of an unobstructed curve. Experts will recognize this  notion as a $\ZZ/4$ variant of Seidel's notion of a {\em graded Lagrangian} (\cite{seidel1,Auroux1}).

\begin{df}\label{RLP}\hfill

A {\em restricted immersed circle in $P^*$} is an unobstructed immersed circle
$ L_1:S^1\to P^*$ 
which satisfies  $\mu(L_1(S^1),\ell_{\operatorname{inst}})\equiv0 \mod 4$, or, equivalently,
 $\mu(L_1(S^1),\ell_1 )+ z(L_1(S^1)) \equiv0 \mod 4$.

A {\em restricted  immersed arc in $P$} is a proper immersion on an interval  (in the sense of  Definition \ref{pimm})  $ L_1:I\to P$ such that $L_1(I)\cap \barP_\delta$ is unobstructed for small $\delta>0$.
 
A {\em restricted immersed curve} is either a restricted immersed circle or a restricted immersed arc.
\end{df}

The curves $L_0^{\ep,g}$ and the curve $L_1$ of Figure \ref{re1}  are restricted immersed circles. An embedded circle $L$ encircling one corner of $P$ counterclockwise  is unobstructed but not  restricted since $\mu(L,\ell_1)+z(L)=2$.   The image of any straight line segment  in $\RR^2$ joining two lattice points whose interior misses the lattice is mapped via the branched cover (\ref{brcover}) to a restricted immersed arc.

\bigskip

\subsection{A relative $\ZZ/4$ grading}\label{z4}
 
 \medskip
 
 We revisit the notation and constructions of Section \ref{LFT} in the context of restricted immersed curves.  Recall that for simplicity we write $L_k\cap L_j$ for the set of intersection points of $L_k$ with $L_j$ (see Definition \ref{intpt}).

\begin{df} \label{grading}  Given an ordered list $(L_0,\dots, L_{n-1})$ of pairwise transverse restricted immersed curves, 
define
$$gr_{L_0,L_1,\dots, L_{n-1}}:(L_{0}\cap L_{1})\times (L_{1}\cap L_{2})\times\cdots\times (L_{n-2}\cap L_{n-1})\times (L_{n-1}\cap L_0)\to \ZZ/4$$
by 
$$ gr_{L_0,L_1,\dots, L_{n-1}}(p_1,\dots, p_{n})= \Mas_{\ell_{\operatorname{inst}}}(\gamma_0,\dots,\gamma_{n-1}) \mod 4$$
for any choice of $(\gamma_0,\dots,\gamma_{n-1})\in \pi_1(p_1,\dots,p_{n}).$
  
\end{df}

\begin{prop}\label{grade}  Given an ordered list $(L_0,\dots, L_{n-1})$ of pairwise transverse restricted immersed curves,
\begin{enumerate}
\item $gr_{L_0,L_1,\dots, L_{n-1}}(p_1,\dots, p_{n})$ is   independent of the choice of $(\gamma_0,\dots,\gamma_{n})$  and is   invariant under simultaneous cyclic permutations of   $L_0,L_1,\dots, L_{n-1}$ and $p_1, \dots,p_{n-1}, p_{n}=p_0$.

\item If $q$ is another  intersection point of $L_0$ and $L_{k}$, then
\[\begin{multlined}
gr_{L_0,L_1,\dots, L_{n-1}}(p_1,\dots, p_{n})\\
=
gr_{L_0,L_1,\dots L_{k}}(p_1,p_2,\dots,p_{k},q)+
gr_{L_0,L_{k},L_{k+1},\dots L_{n-1}}(q,p_{k+1},\dots,p_{n-1},p_n).
\end{multlined}\]

\item $gr_{L_{n-1},L_{n-2},\dots,L_0}(p_{n-1},p_{n-1},\dots, p_1,p_n)=
2-n-gr_{L_0,L_1,\dots, L_{n-1}}( p_1,p_2,\dots, p_{n-1},p_n)$.
\end{enumerate}

In particular,
$$gr_{L_0,L_1}(p,r)=gr_{L_0,L_1}(p,q)+gr_{L_0,L_1}(q,r),~  gr_{L_0,L_1}(p,p)=0,$$ 
$$   gr_{L_1,L_0}(q,p)= gr_{L_0,L_1}(p,q)=-gr_{L_0,L_1}(q,p),$$
and if $\mathcal{M}_{L_0,L_1}(p,q)$ is non-empty, then $gr_{L_0,L_1}(p,q)=1$.

 \end{prop}
\begin{proof} The assumption that the $L_k$ are restricted immersed curves implies that the mod 4 reduction of $\Mas_{\ell_{\operatorname{inst}}}(\gamma_0,\dots,\gamma_{n-1})$  is independent of the choice of paths $\gamma_k$ and therefore
$gr_{L_0,L_1,\dots, L_{n-1}}(p_1, \dots, p_n)$ is well defined. 

The remaining assertions follow from their counterparts in Proposition \ref{basicmas}. 
\end{proof}

The function 
$gr_{L_0,L_1}:(L_0\cap L_1)^2\to \ZZ/4$ is called 
  {\em the relative $\ZZ/4$ grading on $C(L_0,L_1)$.}  Proposition \ref{grade} and Equation (\ref{diff}) imply that the differential (if defined) $\partial:C(L_0,L_1)\to C(L_0,L_1)$ lowers the relative grading by $1$, i.e.,  $(C(L_0,L_1),\partial)$ is a chain (rather than a cochain) complex.

\medskip

We thank Matt Hogancamp for formulating   the following corollary. Its  proof   follows quickly from Proposition \ref{grade}, and we omit  it. 
\begin{cor} \label{hogansheroes} Given an ordered list $(L_0,\dots, L_{n-1})$ of pairwise transverse restricted immersed curves,
$$gr_{L_0,L_1,\dots, L_{n-1}}(p_1,\dots,p_n)-gr_{L_0,L_1,\dots, L_{n-1}} (q_1,\dots,q_n)=\sum_{k=1}^n gr_{L_{k-1},L_k}(p_k,q_k).$$
Moreover, if there exists a Maslov index $k_p$ immersed $n$-gon through $(p_1,\dots, p_{n})$ and  a Maslov index $k_q$ immersed $n$-gon through $(q_1,\dots, q_{n}) $, then 
$$gr_{L_{0},L_{n-1}}(p_n,q_n)=k_q-k_p+\sum_{k=1}^{n-1} gr_{L_{k-1},L_k}(p_k,q_k)$$\qed
\end{cor}

  \medskip

  In the next lemma, we provide  a practical formula for $gr_{L_0,L_1}$ in terms of the slope 1 line field  and the reversed paths $\alpha_k(t)=\gamma_{k}(1-t)$.     We find this formula to be the simplest to remember, and most of the subsequent calculations of relative gradings in this paper are obtained using this formula, without referring back to Maslov index definitions and conventions.   The omitted proof consists applying Equation (\ref{symtau}) and Proposition \ref{basicmas} to the difference class $z$.

\begin{lem}The relative $\ZZ/4$ grading on $C(L_0,L_1)$  is given as follows. 

Let $p,q$ be intersection points of $L_0$ with $L_1$, let $\alpha_0$ be a path in $L_0$ from $p$ to $q$, let $\alpha_1$ be a path in $L_1$ from $q$ to $p$, $\tau(L_0,L_1,\ell_{1})_{p}$ and $\tau(L_0,L_1,\ell_{1})$ the triple indices with respect to the slope 1 line field $\ell_1$. Then 
\begin{equation}\label{gradeeq}gr_{L_0,L_1}(p,q)=\mu(L_0,\ell_{1})_{\alpha_0}+ \mu(L_1,\ell_{1})_{\alpha_1}+ \tau(L_0,L_1,\ell_{1})_{p}-\tau(L_0,L_1,\ell_{1})_q + z(L_0(\alpha_0)*L_1(\alpha_1)).\end{equation} 
\qed
\end{lem}

 When the order is clear from context, we   write $gr(p,q)$ rather than $gr_{L_0,L_1}(p,q)$ for the relative $\ZZ/4$ grading on $C(L_0,L_1)$.

\subsection{Finiteness of bigons} 

  When $(L_0,L_1)$ is an transverse pair of restricted immersed curves, we have constructed a relative $\ZZ/4$ grading  on the vector space $C(L_0,L_1)$ spanned by the intersection points of $L_0:R_0\to P^*$ and $L_1:R_1\to P$. To  show that $C(L_0,L_1)$ is a chain complex, we must show that $\mathcal{M}(p,q)=\mathcal{M}_{L_0,L_1}(p,q)$ is finite for any pair intersection points $p,q$.  To this end we introduce the notion of an admissible pair.  
  
\begin{df}\label{admis}  A pair 
 $$(L_0:R_0\to P,L_1:R_1\to P)$$ of  restricted immersed curves in $P$ is called an {\em admissible pair} provided:
\begin{enumerate}
\item at least one of $L_0$ or $L_1$ is a restricted immersed circle,   
\item if $\alpha_0:S^1\to R_0$ and $\alpha_1:S^1\to R_1$   are loops so that  $L_0\circ \alpha_0$ and $L_1\circ \alpha_1  $ are freely homotopic, then both $ \alpha_0 $ and $ \alpha_1 $ are nullhomotopic (this holds automatically if  one of $L_0,L_1$ is a restricted immersed arc, since restricted immersed circles are essential),
\item $L_0$ and $L_1$ intersect transversely.
\end{enumerate}
 
\end{df}
 
 If we   put a complete hyperbolic metric on $P^*$, then the second assumption is equivalent to the requirement that the unique geodesic representatives of   the homotopy classes of  $L_0$ and $L_1$ are transverse.

 \medskip

Given an admissible pair of restricted curves $L_0,L_1$, to each element of $u\in \pi_2(p,q)$ one can assign a {\em local degree} function $f_u$, which is an integer valued function with domain the set  of complementary regions of $L_0\cup L_1$, i.e., the path components of $P^*\setminus(L_0\cup L_1)$. Its value on a complementary region is the signed number of preimages of a regular value of any smooth representative of $u$.

\begin{lem}\label{pi2} For each pair $(p,q)$ of intersection points, $\pi_2(p,q)$ is either empty or contains a unique element.
 \end{lem}
\begin{proof} 
Since $\pi_2(p,q)$ is either empty or else $\pi_2(p,p)$ acts transitively on $\pi_2(p,q)$, it suffices to prove that $\pi_2(p,p)$ contains a unique class, namely the class of the constant map.

Write $p=(p_0,p_1)\in R_0\times R_1$.  Given $(u,(\gamma_0,\gamma_1))\in\pi_2(p,p)$, $\gamma_0$  (resp. $\gamma_1$) is a loop  in $R_0 $ (resp. $R_1$)  based at $p_0$ (resp. $p_1$). 
If $R_1$ is an arc,   then $L_1\circ \gamma_1$ is homotopic rel endpoints to the constant path, and hence so is $\gamma_0$.
If $R_1$ is a circle, then  the second assumption of Definition \ref{admis}  implies that $\gamma_0$ and $\gamma_1$ are nullhomotopic loops.   Either way, the images $L_0\circ \gamma_0$ and $L_1\circ \gamma_1$   in $P^*$ are homotopic loops based at $p$.

 By the homotopy extension property, $(u,(\gamma_0,\gamma_1))$ may be homotoped in $\pi_2(p,p)$ so that $\gamma_0$ and $\gamma_1$ are constant.  But then $u$ sends the entire boundary of the bigon to $p$, and hence represents a class in $\pi_2(P^*)=0$.  Thus we may further homotop $u$ rel boundary to the constant map and so $\pi_2(p,p)=0$, as desired.  
\end{proof}

\begin{cor}\label{finite} Given an admissible pair $(L_0,L_1)$, each set $\mathcal{M}(p,q)$ is either empty or contains one equivalence class of bigons. 
 \end{cor}
\begin{proof}    Fix intersection points $p$ and $q$, and suppose that $\mathcal{M}(p,q)$ is non-empty.    
Choose two immersed bigons $u,u'$ from $p$ to $q$. 
Lemma  \ref{pi2}    implies that $\pi_2(p,q)$ contains a unique element, so their local degree functions are equal. Moreover, since $u$ (and $u'$) are immersed by an orientation preserving immersion, $f_u=f_{u'}$ takes only non-negative values.

Standard arguments  now show that $u$ and $u'$ can be reconstructed from the data of their local degrees up to reparameterization, so that $u$ and $u'$ are equivalent.  For example, see
 \cite[Theorem 6.8]{SRV}, whose proof applies verbatim to our setting by passing to a compact simply connected submanifold of the universal cover of $P^*$.
 \end{proof}

\begin{rem}\label{finiteeasy}
 {\rm 
That  $\mathcal{M}(p,q)$ is finite (which is all we require for the assertions in the present article) when $L_0$ and $L_1$ are self-transverse immersions can be shown even more easily, as follows. Label the closure of the complementary regions of $P^*\setminus (\text{ image}(L_0)\cup \text{ image}(L_1))$ by $A_1,\dots, A_m$. Notice that the boundary of each $A_i$ is a union of arcs $\alpha_{i,j}$ meeting at convex double points.  The set of all such arcs, $\{\alpha_{i,j}\}$, can be partitioned into pairs which map to the same arc in $P^*$, and hence  each pair comes with an identification so that the surface obtained by identifying these two arcs immerses into $P^*$.

If $u\in \pi_2(p,q)$ has all local degrees $f_u(A_i)$ non-negative,  take $f_u(A_i) $ copies of $A_i$, $i=1,\dots,m$, and label the corresponding edges as $\alpha_{i,j;k}, k=1,\dots ,f_u(A_i)$.    There are finitely many ways of pairing all the arcs $\{
\alpha_{i,j;k}\}$ and gluing them to get a surface which immerses each copy of $A_i$ to its corresponding complementary region. Any immersed bigon from $p$ to $q$ must be equivalent to one of these resulting glued surfaces, hence there are finitely many bigons.}
 
 \end{rem}

 From Theorem \ref{abouza}, Proposition \ref{grade}, and Corollary \ref{finite}  we conclude the following.

\begin{thm}\label{mattmademedothis}
If $(L_0,L_1)$ is an admissible pair, then $(C(L_0,L_1),\partial)$ is a relatively $\ZZ/4$ graded chain complex with $\FF_2$ coefficients.  \qed
\end{thm}

\begin{df} Call the $\ZZ/4$ graded homology of $(C(L_0,L_1),\partial) $  the  {\em  Lagrangian-Floer homology of the  $(L_0,L_1)$} and   denote   it  by $HF(L_0,L_1)$. 
\end{df}

\subsection{Example: Calculation of $HF(L_0,L_1)$ for $L_0=L_0^{\ep,0}$ and $L_1$ the restricted immersed circle of Figure \ref{re1}}\label{examplecalc}  The pair $(L_0,L_1)$ is admissible. In Figure \ref{re3}, the eight intersection points of the Lagrangian $L_1$ of Figure \ref{re1} with $L_0^{\ep,0}$ are labeled $p,q,r,s,t,u,v,w$. There is a bigon from $p$ to $q$ and hence $gr(p,q)=1$.  Similarly, there is a bigon from $w$ to $v$ and hence $gr(w,v)=1$.    

We compute $gr(q,s)=1$ in detail next. Let $\alpha_0$ be the path in $L_0$ starting at $q$, heading down and to the left, around the lower left corner, then back up to $s$.  Let $\alpha_1$ be the short arc on $L_1$ form $s$ back to $q$.  Then  $\mu(L_0,\ell_1)_{\alpha_0}=1$ (there is one tangency at the upper right part of the figure, near $(0,2\pi)$), $\mu(L_1,\ell_1)_{\alpha_1}=0$ since the arc $\alpha_1$ is everywhere transverse to the slope 1 line field $\ell_1$. Next, $\tau(L_0,L_1,\ell_1)$ equals $0$ at $q$ and $1$ at $s$, and $z(L_0(\alpha_0)*L_1(\alpha_1))=1$ since the loop $L_0(\alpha_0)*L_1(\alpha_1)$ goes once around the lower left corner counterclockwise. Using Equation (\ref{gradeeq}) we conclude:
$$gr(q,s)=1+0+0-1+1=1.$$
An identical argument gives $gr(t,p)=1$, $gr(v,r)=1$, and $gr(u,w)=1$. 

One more calculation is required to complete the calculation of the relative grading, for example $gr(t,r)$. Take $\alpha_0$ to be the path in $L_0$ from $t$ to $r$ which heads down and to the left, around the bottom left corner clockwise, then back up to $t$. Take $\alpha_1$ the path in $L_1$ from $r$ back to $t$ which starts by heading to the right, then down and continuing along $L_1$ until it returns to $t$.  Then  $\mu(L_0,\ell_1)_{\alpha_0}=1$, $\mu(L_1,\ell_1)_{\alpha_1}=1$, $\tau(L_0,L_1,\ell_1)_t=0$,  $\tau(L_0,L_1,\ell_1)_r=1$, and $z(L_0(\alpha_0)*L_1(\alpha_1))=0.$  Thus
$$gr(t,r)=1+1+0-1+0=1.$$

From these calculations and additivity of the relative grading, we conclude that $gr(p,r)=0, gr(p,q)=gr(p,u)=1, gr(p,s)=gr(p,w)=2, gr(p,t)=gr(p,v)=3$. Hence $C(L_0,L_1)$ has rank 2 in each grading.  There are only two Maslov index $1$ bigons, and hence the differential is given by $\partial p=q$ and $\partial v=w$, and so the homology has rank $1$ in each grading.

\medskip

 We introduce a bit of notation that will simplify our descriptions of the $\ZZ/4$ gradings.  The notation $(n_0, n_1,n_2,n_3)$ with $n_i$ non-negative integers denotes the $\ZZ/4$ graded vector space (over $\FF_2$) whose dimension on grading $i$ is $n_i$. 
Thus, for this example, $$C(L_0,L_1)=(2,2,2,2)\text{~and ~}HF(L_0,L_1)=(1,1,1,1).$$

      \begin{figure}[h]
\begin{center}
\def\svgwidth{2.6in}
 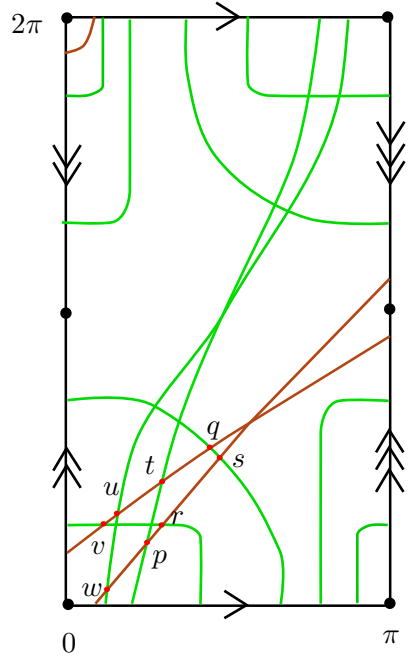
 \caption{The eight intersection points of $L_0$ and $L_1$  generating $C(L_0,L_1)$.\label{re3} }
\end{center}
\end{figure}

\section{Homotopy invariance}\label{homotopy}

 \color{black}
We  show that   the relatively  $\ZZ/4$ graded group $HF(L_0,L_1)$  depends    only on the homotopy classes of $L_0$ and $L_1$   (rel boundary)   in $P^*$.    Our argument follows the approach taken in \cite[Proposition 4.1]{abou}   and presumably will be somewhat familiar  to experts.  It is worth noting that we make  no requirement that the curves be related by a Hamiltonian isotopy.  

\begin{thm}\label{rhi}     Let $(L_0,L_1)$ and $(L_0',L_1')$  be two  admissible pairs which satisfy:
\begin{enumerate}
\item  the admissible circles $L_0, L_0'$ are freely homotopic.
\item  If $L_1,L_1'$  are immersed restricted circles  they are freely homotopic. If $L_1,L_1'$ are immersed restricted arcs,     they are homotopic rel endpoints. Assume also that near their endpoints, $L_1$ and $L_1'$ intersect only at their endpoint.
 \item   $L_0,L_1,L_0',L_1'$ are pairwise transverse and have no triple points.
\end{enumerate}

Then 
 $$HF(L_0,L_1)\cong HF(L_0',L_1').$$  
  as relatively $\ZZ/4$ graded $\FF_2$ vector spaces.
\end{thm}

\noindent{\em Proof.}     In order to avoid the proliferation of sub and superscripts, we make the following notational changes.
Set $$A:=L_0,~ B:=L_1, ~C:=L_0', ~D:=L_1'.$$
And we must show that 
$$HF(A,B)=HF(C,D).$$

Consider $A$ and $C$ as immersions of {\em the} unit circle   $A,C:S^1\to P^*$. Also,  consider $B$ and $D$ as  immersions of the real line $\RR$ to $P^*$, with the understanding that if $R_1$ is a circle, then $B$ and $D$ are $2\pi$ periodic  and,   if $R_1$ is an arc, then we identify the interior of $R_1$ (which maps by $B,D$ to $P^*$) with $\RR$.  In brief, $B$ and $D$ are immersions of $\RR$ to $P^*$ which are periodic if $R_1$ is a circle and proper if $R_1$ is an arc. The second condition in the hypotheses implies that outside some compact set in $\RR$, $B$ and $D$ are disjoint embeddings, but with the same limit points at $\pm \infty$.

Let $x=A(1)$.  The immersion   $A$ induces a homomorphism on fundamental groups.  Consider the infinite cyclic subgroup
$$Z=\text{Image} ~A_\#:\pi_1(    S^1  ,1)\to \pi_1(P^*,x)$$
and let $$f:(\Si,\hat x)\to (P^*,x)$$  denote the (non-regular) cover corresponding to $Z$.  Thus $A:S^1\to  P^*$ lifts to $\hA:S^1\to \Si$, with $\hA(1)=\hat x$.

Since $\hA:S^1\to \Si$ generates $\pi_1(\Si,\hat x)=Z\cong\ZZ$, the preimage of $\hA$ in the universal cover $\widetilde P^*\cong\RR^2$  is connected, in fact the image of an immersion $\tA:\RR\to \RR^2$.  Since $A=L_0$ is unobstructed,  $
\tA$ is an embedding, from which it follows that $\hA:S^1\to \Si$ is an embedding.  
(In the following, we use the notation $\hA,\hB,\hC,\hD$ for lifts of $A,B,C$ to $\Si$,  and $\tA,\tB,\tC,\tD$ for lifts to the universal cover $\RR^2$.)

It can be easily shown, for example using elementary hyperbolic geometry, that $\Si$ is diffeomorphic to the cylinder $S^1\times \RR$. 
Fix such a diffeomorphism and the corresponding cover
$$(S^1\times\RR,\hat x)\to \to (P^*,x).$$

 Let $F:S^1\times[0,1]\to P^*$ be a homotopy from $A$ to $C$. Let $\widehat F:S^1\times [0,1]\to S^1\times\RR$ be the unique lift of $F$ satisfying $\hat F(1,0)=\hat x$. Let $\hC(z)=\widehat F (z,1)$. 
Then $\hC$ is a lift of $C$ to $S^1\times\RR$, and, as with $\hA$, $\hC$ is an embedding.   In particular, since $\hA$ and $\hC$ are homotopic embedded curves in $ S^1\times \RR$, $\hA$ and $\hC$ are isotopic.

The following three lemmas will complete the proof of Theorem \ref{rhi}.

\begin{lem} \label{case1} 
If  $\hA$ and $\hC$ meet transversely in precisely two points, then 
 $HF(A,B)\cong HF(C,B)$   and $HF(A,D)\cong HF(C,D)$ as relatively $\ZZ/4$ graded $\FF_2$ vector spaces.
 \end{lem}

  \begin{lem} \label{case2} 
There exists a sequence $A_0,A_1,\dots, A_r$ of homotopic restricted immersed circles so that $A_0=A$, $A_r=C$, and $\hA_k$ intersects $\hA_{k+1}$ transversely in two points.
 \end{lem}
   \begin{lem} \label{case3} 
$HF(A,B)\cong HF(A,D)$ as relatively $\ZZ/4$ graded $\FF_2$ vector spaces.

 \end{lem}
 
 \noindent{\em Proof of Lemma \ref{case1}.}  
 Up to diffeomorphism of the cylinder $S^1\times \RR$, the  curves $\hA$ and $\hC$  are illustrated in Figure \ref{cylinderfig}. Also illustrated is a third curve $\hA'$ which meets $\hA$ transversely in two points. We assume that $\hA'$ is very ($C^1$) close to $\hA$, so that the preimage of $B$ is also transverse to $\hA'$ and induces a bijection between the intersection points.    Three pairs of intersection points   $\hat a, \hat b\in \hA \cap \hB$, $\hat f, \hat c\in \hA' \cap \hC$, and $\hat b, \hat d\in \hA \cap \hA'$  are illustrated. Define $A'$ to be the image of $\hA'$ under the covering map $S^1\times \RR\to P^*$, and let $a,b,\dots, f$ denote the images of $\hat a, \hat b,\dots, \hat f$ in $P^*$.   
 
  \begin{figure}[h]
\begin{center}
\def\svgwidth{2.6in}
 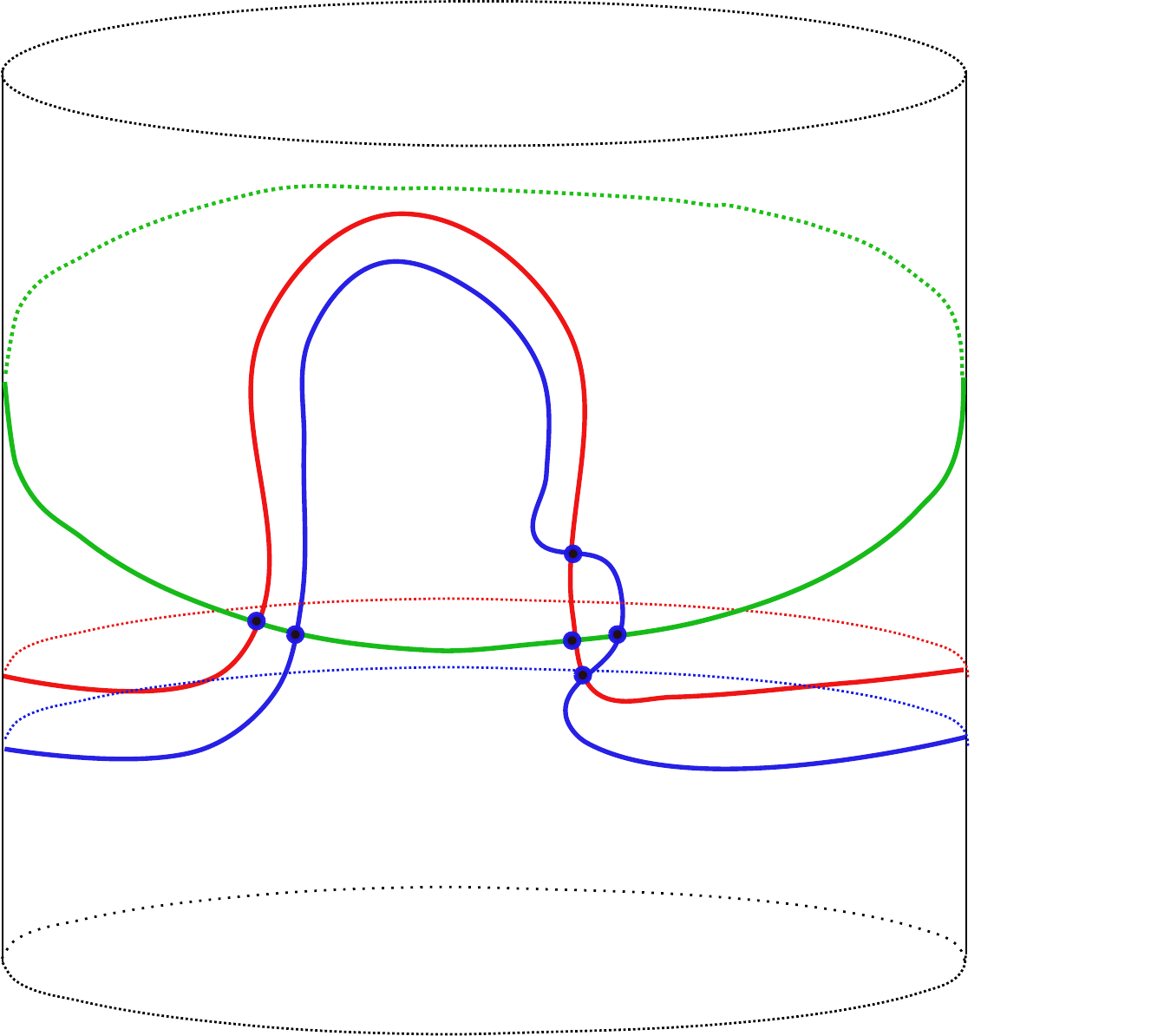
 \caption{  \label{cylinderfig} }
\end{center}
\end{figure}

  We now summarize some facts about the maps $\mu_2$ and $\mu_3$, defined in  Definition \ref{munmaps}, that will be used to complete Lemma \ref{case1}.

\begin{lem} \label{welldef} Consider $b$ as a generator of $C(A,C)$, $e$ as a generator of $C(A,A')$ and $f$ as a generator of $C(C,A')$. Then:
\begin{enumerate}
\item For any $x\in C\cap B$ and $y\in A\cap B$, the set  $\mathcal{M}_{A,C,B}(b,x,y)$ is finite. 
\item For any $x\in A'\cap B$ and $y\in C\cap B$, the set  $\mathcal{M}_{C,A',B}(f,x,y)$ is finite.
\item For any $x\in A'\cap B$ and $y\in A\cap B$, the set  $\mathcal{M}_{A,A',B}(e,x,y)$ is finite.
\item For any $x\in A'\cap B$ and $y\in A\cap B$, the set  $\mathcal{M}_{A,C,A',B}(b,f,x,y)$ is finite.
\end{enumerate}
Hence the maps
 $\mu_2(b,-):C(C,B)\to C(A,B),~\mu_2(f,-):C(A',B)\to C(C,B)$, 
 $\mu_2(e,-):C(A',B)\to C(A,B)$ and
 $\mu_3(b,f,-):C(A',B)\to C(A,B)$ are well defined.
\end{lem}
\begin{proof}

 For the first statement, first note that to each class   $\phi\in \pi_2(b,x,y)$ one can assign an  integer local degree  to each complementary region  of $A\cup C\cup B$, i.e., to each path component of $P^*\setminus (A\cup C\cup B)$ (see the proof of Corollary \ref {finite}).

The homotopy group $\pi_2(b,b)$  corresponding to the pair $A,C$ acts on $\pi_2(b,x,y)$ corresponding to the ordered triple $(A,C,B)$ by attaching a bigon from $b$ to $b$ to the vertex of a triangle with vertices $b,x,y$ at $b$. We  show this action is transitive.  From Figure \ref{cylinderfig} one sees that there exists $\tau_1\in \pi_2(b,b)$ whose two boundary loops represent  generators of $\pi_1(A)$ and $\pi_1(C)$. Let $\tau_n$ denote the $n$th power of $\tau_1$.

Suppose that $\phi_1,\phi_2\in \pi_2(b,x,y)$.  Denote the image  of $\phi_k$ in $\pi_1(b,x,y)$ by $(\alpha_k, \gamma_k, \beta_k),~k=1,2$, where $\alpha_k$ is a path in $A$ from $y$ to $b$,  $\gamma_k$ is a path in $C$ from $b$ to $x$ and $\beta_k$ is a path in $B$ from $x$ to $y$.

The loop $\gamma_1* \gamma_2^{-1}$ in $C$ based at $b$ represents some multiple  of the generator of $\pi_1(C)$. Hence, by  replacing $\phi_2$ by $\tau_n\cdot\phi_2$ for the appropriate $n$, one may assume that  $\gamma_1* \gamma_2^{-1}$ is nullhomotopic. Using the homotopy extension property we may arrange that $\gamma_1=\gamma_2$.  

 The triangles $\phi_1$ and $\phi_2$ then  glue together along $\gamma_1$ and $\gamma_2$ to provide a free homotopy of the loop 
 $\beta_1*\beta_2^{-1}$ in $B$ to the loop $\alpha_1* \alpha_2^{-1}$ in $B$. Since $A,B$ form an admissible pair (Definition \ref{admis}) it follows that $\alpha_1* \alpha_2^{-1}$  and  $\beta_1*\beta_2^{-1}$ are nullhomotopic, so that $ \alpha_1,\alpha_2$ (resp. $\beta_1$, $\beta_2$) are homotopic rel endpoints.   We may therefore   replace $\phi_2$ by another map in the same  homotopy class   so that $\alpha_1=\alpha_2$ and $\beta_1=\beta_2$.  Gluing $\phi_1$ to $\phi_2$ along the boundary arcs $\alpha_k$, $\beta_k$ and $\gamma_k$    yields 
a class in $\pi_2(P^*)=0$.  Hence $\phi_1=\phi_2$.  We have shown that  $\pi_2(b,b)$ acts transitively on $\pi_2(b,x,y)$.

A class $\tau\in \pi_2(b,b)$ determines local degrees for each complementary region  of $ A\cup C $, and hence also for each complementary region of $A\cup C\cup B$. 
Lifting to the cylinder, one sees that $\pi_2(b,b)\cong \pi_2(\hb,\hb)\cong\ZZ$, with $n\in \ZZ$ corresponding to the class $\tau_n\in \pi_2(\hb,\hb)$, whose   local degrees in $S^1\times \RR\setminus (A\cup C)$ are $n$ and $ -n$ in the two bounded regions of $S^1\times \RR\setminus (\hA\cup\hC)$, and $0$ in both unbounded regions.   Let $W_1$ and $W_3$ denote the two bounded regions  in $S^1\times \RR\setminus (\hA\cup\hC)$, and $W_2$ and $W_4$ denote the two unbounded regions,   indexed so that moving clockwise around $\hb$ they are ordered $W_1,W_2,W_3,W_4$ and so that the local degrees of $\tau_n$ are $n,0,-n, 0$. (see Figure \ref{finitenessfig}, with $n=n_1$ and $n_2=0$.)

If $U\subset P^*$ is a small  evenly covered disc neighborhood of $b$, only finitely many of the components of the preimage of $U$ in $S^1\times \RR$ meet $W_1$ and $W_3$. Suppose that $K_1$ such components meet $W_1$ and $K_2$ meet $W_3$. Let $K=K_1-K_2$.  Then in $P^*$, the local degrees of $\tau_n$ about $b$ are, in clockwise order, 
$ n +Kn,  Kn, -n +Kn, Kn $.

Fix $\phi\in \pi_2(y,x,b)$ and let $d_1,d_2,d_3,d_4$ denote local degrees of $\phi$ in the four quadrants  around $b$. The local degrees of $\tau_n\cdot\phi$ are just the sum of the local degrees of $\tau_n$ and $\phi$.   
 Thus the local degrees of $\tau_n\cdot\phi$ around $b$ are, in clockwise order, $d_1 +n +Kn, d_2+Kn, d_3-n +Kn, d_4+Kn.$

On the other hand, if $\tau_n\cdot \phi$ is represented by a Maslov index 0 immersed 3-gon, then all its corners are convex  (Proposition \ref{convex}) and so the local degrees near $b$ must take the form $r+1, r,r,r$ moving clockwise around $b$, for some non-negative integer $r$. In particular, $$|d_1-d_3+2n|=|(d_1+n+Kn)-(d_3-n+Kn)|\leq 1,$$
and so at most two of the classes $\tau_n\cdot \phi$ support Maslov index 0 immersed 3-gons.  Each such class determines a finite number of immersed 3-gons, by the same argument given in Remark \ref{finiteeasy}.  Thus $\mathcal{M}(b,x,y)$ is finite.

The proofs of the second and third assertions are the same, and we leave them to the reader. 

The last assertion has a similar proof, so we outline it, highlighting the differences.  The ordered 4-tuple $(A,C, A',B)$ gives rise to the sets $\mathcal{M}(b,f,x,y)$ and  $\pi_2(b,f,x,y)$. Since $A,C,A'$ are homotopic and $A,B$ form an admissible pair, a similar argument to that used in the triangle case shows that any two classes $\phi_1,\phi_2\in  \pi_2(b,f,x,y)$ are related by the action of $\pi_2(b,b)\times \pi_2(f,f)$: one first lets $\pi_2(b,b)$ act on $\phi_1$ to make the boundary paths of $\phi_1$ and $\phi_2$ along $A$ agree, then let $\pi_2(f,f)$ act to make the boundary paths along $C$ agree. Gluing the two rectangles along these edges yields a twice punctured sphere,   giving a free homotopy from a loop in $A'$ to a loop in $B$, and since $(A',B)$ forms an admissible pair, these loops are both nullhomotopic.  The argument then proceeds as in the 3-gon case to conclude that there exists $\tau_{n_1}\in \pi_2(b,b)$ and $\rho_{n_2}\in \pi_2(f,f)$ so that $\tau_{n_2}\cdot\rho_{n_1}\cdot \phi_1=\phi_2$.

Fix $\phi\in \pi_2(b,f,x,y)$  and assume its local multipicities in the four quadrants  clockwise around $b$ are $(d_1,d_2,d_3, d_4)$. 
In the four quadrants near $\hat b$, the multiplicities of $\tau_{n_1}\in \pi_1(b,b)$ are, in clockwise order,  $n_1, 0, -n_1, 0$ and the multiplicities of $\rho_{n_2}\in \pi_2(f,f)$ are $n_2,n_2, 0,0$. Figure \ref{finitenessfig} illustrates the contributions of $\tau_{n_1}$ and $\rho_{n_2}$ to the local multiplicities near $b$.

  \begin{figure}[h]
\begin{center}
\def\svgwidth{3.5in}
 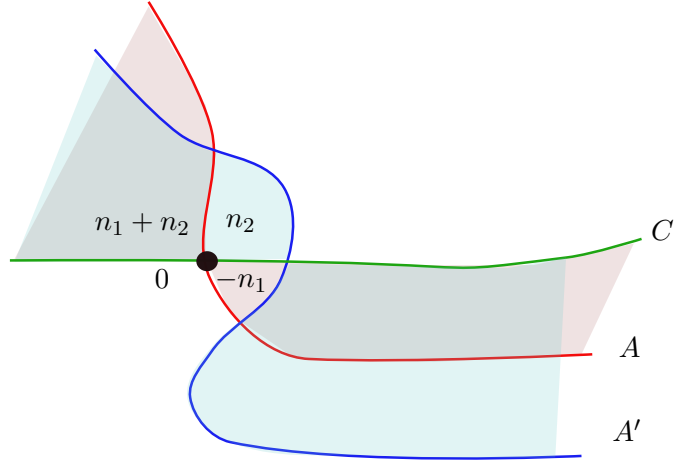
 \caption{  The local multiplicities of $\rho_{n_2}$ and $ \tau_{n_1}$ around $b$\label{finitenessfig} }
\end{center}
\end{figure} 

\medskip

 As in the triangle case,  the local multiplicities of $\tau_{n_2}\cdot\rho_{n_1}\cdot \phi$ near $b$ are, in clockwise order, $d_1+n_1+n_2 +K, d_2+K, d_3-n_1+n_2+K, d_4+K$, where $K$ is a function of $n_1$ and $n_2$ which takes into account how many of the preimages of $b$ in $S^1\times \RR$ lie in the complementary regions on which $\tau_{n_2}$ and $\rho_{n_1}$ are supported.  But to represent an orientation preserving immersed  4-gon of Maslov index $-1$, all corners are convex (Proposition \ref{convex}), and hence these numbers must be a cyclic permutation of $(r+1,r,r,r)$ for some non-negative integer $r$.  This implies that 
$$|d_1-d_3+2n_1|\leq 1 \text{ and } |d_1+n_1+n_2-d_4|\leq 1,$$
which is only possible for at most four choices of $n_1, n_2$. Thus only finitely many classes in $\pi_2(b,f,x,y)$ support immersed Maslov index $-1$ rectangles, and as  in  Remark   \ref{finiteeasy}, this implies that $\mathcal{M}(b,f,x,y)$ is finite.
\end{proof}

\noindent{\em Continuing the proof of Lemma \ref{case1}.}
 Consider $b$ as a generator of $C(A,C)$.  Any bigon from $b$ to another intersection point of $A$ with $C$ has a unique lift to the cover $S^1\times \RR$, and hence must be a bigon from $b$ to $a$. From Figure \ref{cylinderfig} one sees that there are precisely 2 such bigons up to equivalence, and hence $\partial b=2a=0$ (we are using $\FF_2$ coefficients). Thus $b$ is a cycle.    Similar arguments show that $e\in C(A,A')$ and $f\in C(C,A')$ are also cycles.

The map $\mu_2$ satisfies  the $A_2$ relation, which, applied to $b$ and an  arbitrary $x\in C(C,B)$ (writing $\partial$ instead of $\mu_1$) says:
$$0 =\mu_2(\partial b, x)+\mu_2(b,\partial x)+\partial \mu_2(b,x).$$
Since $\partial b=0$, this implies  that $\mu_2(b,-):C(C,B)\to C(A,B)$ is a chain map. 

    The cycle $f\in C(C,A')$ determines a chain map  $\mu_2(f,-):C(A',B)\to C(C,B)$. 
 The product $\mu_2(b,f)\in C(A,A')$   equals $e$, since immersed triangles lift to the cover, and there is precisely one oriented immersed Maslov index zero triangle for the ordered triple $(A,C,A')$ from  $(\hat b,\hat f)$ to $\hat x$ for an intersection point $\hat x$ of $\hA$ and $\hA'$, namely the embedded triangle  with $\hat x=\hat e$.  Similarly, the cycle $e=\mu_2(b,f)$ determines a chain map
 $$\mu_2(e,-):C(A',B)\to C(A,B).$$

 The $A_3$ relation  gives, for an arbitrary generator $x$ of $C(A',B)$:
 \begin{equation*}
\begin{multlined}
 0=\mu_3(\partial b,f,x) +\mu_3(b,\partial  f,x)+\mu_3( b,f,\partial x)  +\mu_2(\mu_2(b,f),x)+\mu_2(b,\mu_2(f,x))+\partial \mu_3(b,f,x)\\
 = H(\partial x) + \mu_2(e,x) + \mu_2(b,-)\circ \mu_2(f,x)+ \partial H(x)\quad\quad
\end{multlined}
\end{equation*}
where $H:C(A',B)\to C(A,B)$ is defined by $H(x)=\mu_3(b,f,x)$.
 In other words, $H$ is a chain homotopy from the composite
 $$C(A',B)\xrightarrow{\mu_2(f,-)}C(C,B)\xrightarrow{\mu_2(b,-)}C(A,B)$$
to $$\mu_2(e,-):C(A',B)\to C(A,B).$$

Now $\mu_2(e,-)$ is a chain isomorphism: this is an immediate consequence of the fact that $\hA$ and $\hA'$ are $C^1$ close, so that each intersection point $p'$ of $A'$ with $B$ corresponds to precisely one intersection point $p$ of $A$ with $B$, the correspondence  induced by  a unique Maslov index $0$ immersed triangle  associated to the ordered triple $(A,A',B)$ From $(e,p')$ to $ p$.
This proves that the map $\mu_2(b,-):C(C,B)\to C(A,B)$ induces a surjection $HF(C,B)\to HF(A,B)$, for any $B$.

The fact that $\mu_2(b,-)$ induces an injection $HF(C,B)\to HF(A,B)$ for any $B$ is proved by a very similar argument as surjection, using the immersed curves illustrated in Figure \ref{cylinder2fig}.
In this case,  $\hC'$ is a curve $C^1$ close to $C$, and one shows that $z\in C(C',A)$ is a cycle, $\mu_2(z,b)=w\in C(C',C)$, and the composite of chain maps
$$C(C,B)\xrightarrow{\mu_2(b,-)}C(A,B)\xrightarrow{\mu_2(z,-)}C(C',B)$$ is chain homotopic to the chain isomorphism $\mu_2(w,-):C(C,B)\to C(C',B)$. This proves $\mu_2(b,-)$ induces an injection $HF(C,B)\to HF(A,B)$ for any $B$.

  \begin{figure}[h]
\begin{center}
\def\svgwidth{2.6in}
 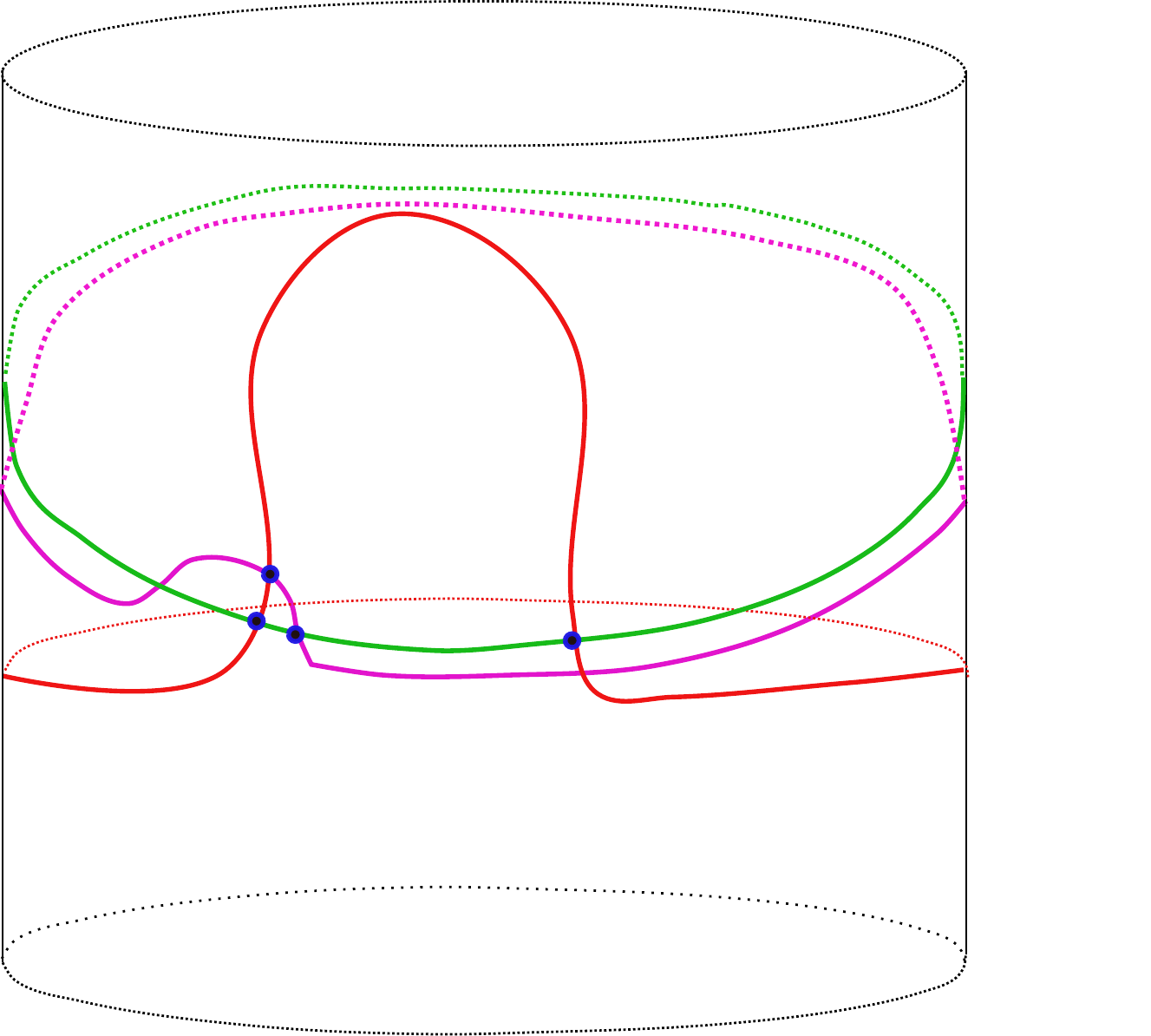
 \caption{  \label{cylinder2fig} }
\end{center}
\end{figure}

\medskip

Finally we show that  the relative $\ZZ/4$ grading is preserved.  Suppose that  $x_1, x_2\in C(C,B)$ and 
$\mu_2(b,x_1)=\sum_i m_i y_i$ and $ \mu_2(b,x_2)=\sum_j n_jz_j$ in $C(A,B)$, with $m_i, n_j$ non-zero.  Hence there exist Maslov index zero immersed 3-gons for the ordered triple $(A,C,B)$ from  $ (b, x_1)$ to $ y_i$ and from  $(b,x_2)$ to $ z_j$.

Then Corollary \ref{hogansheroes} gives
$$gr_{A,B}(y_i,z_j)= gr_{C,B}(x_1,x_2)+ gr_{A,C}(b,b)=gr_{B,C}(x_1,x_2).$$
This says that the chain map $\mu_2(b,-):C(C,B)\to C(A,B)$ preserves the relative $\ZZ/4$ grading, and hence also the induced map on homology.  
The same argument shows the chain maps $\mu_2(f,-)$ and $\mu_2(e,-)$ preserve the relative grading.  
Thus, $\mu_2(b,-)$ induces an isomorphism $HF(A,B)\to HF(A,C)$ of relatively $\ZZ/4$ graded vector spaces.
\qed\medskip

 \noindent{\em Proof of Lemma \ref{case2}.} If $\hA$ and $\hC$ intersect non-trivially (hence in an even number of points),  then the existence of the sequence $\hA_0=\hA,\hA_1,\dots,\hA_r=\hC$  with the stated property is an immediate consequence of Lemma 4.2 of \cite{abou}.         If $\hA$ and $\hC$ are disjoint, then one can take a parallel copy of $\hA$ and perform a Reidemeister 2 move that introduces a pair of intersection points with $\hA$ and a pair with $\hB$.   
\qed \medskip

 \noindent{\em Proof of Lemma \ref{case3}.} 
 If $B$ and $D$ are homotopic restricted immersed circles, then the proof follows exactly along the same lines as the proof of Lemmas \ref{case1} and \ref{case2}, reversing the roles (and order) of $A,C$ and $B,D$.
 
 If $B$ and $D$ are homotopic restricted immersed arcs, then a different proof is needed.
It is convenient to put a complete hyperbolic metric on $P^*$ and to let $\HH\to P^*$ denote the universal cover. Let $\tB:\RR\to \HH$ be a lift of $B$.  Since $D$ is homotopic rel endpoints to $B$, there is a lift $\tD:\RR\to \HH$ of $D$ with the same limit points at the circle at infinity. The assumption that the limiting slopes at the endspoints are distinct and that $B$ and $D$ are transverse imply that   the closures of $\tB$ and $\tD$ in the closed disk $\overline \HH$  intersect in finitely many points. 

 If $\tB$ and $\tD$ are disjoint, then their closures bound a bigon in $\overline \HH$ with vertices on the circle at infinity.  If $\tB$ and $\tD$ intersect in one point, then clearly there are a pair of bigons with boundary in their closures, each including one point on the boundary of $\overline \HH$.  Finally, if $\tB \cap \tD$ consists of more than one point, then the result of D.B.A Epstein \cite[Lemma 3.2]{epstein} (see also 
\cite[Lemma A.10]{buser})   shows that  one can find an embedded bigon in  $\HH$ for    $\tB$ and $\tD$. 


In each of these three cases,   one can find a sequence of embedded arcs $\tB=\tB_0,\tB_1,\dots,\tB_r=\tD$ in $\overline \HH$ so that $\tB_k$ intersects $\tB_{k+1}$ transversely  in one interior point and in their endpoints. The argument is illustrated in Figure \ref{threecasefig}. The  figure on the left corresponds to the first case, when  $\tB \cap \tD = \emptyset$.    The figure on the right corresponds to the third case, where an interior bigon between $\tB$ and $\tD$ is used to construct an arc $\tB_1$ which intersects $\tB_0=\tB$ in one interior point and $\tD$ in two fewer points, providing the required induction step. If $\tB$ and $\tD$ started out with an odd number of intersection points in their interior, one eventually reaches $\tB_r$ which meets $\tD$ in one point in their interior, and if they started out with an even number of intersection points in their interior, one reaches $\tB_r$ whose interior misses $\tD$, in which case one can add one more step as in the first case.    

  \begin{figure}[h]
\begin{center}
\def\svgwidth{4.7in}
 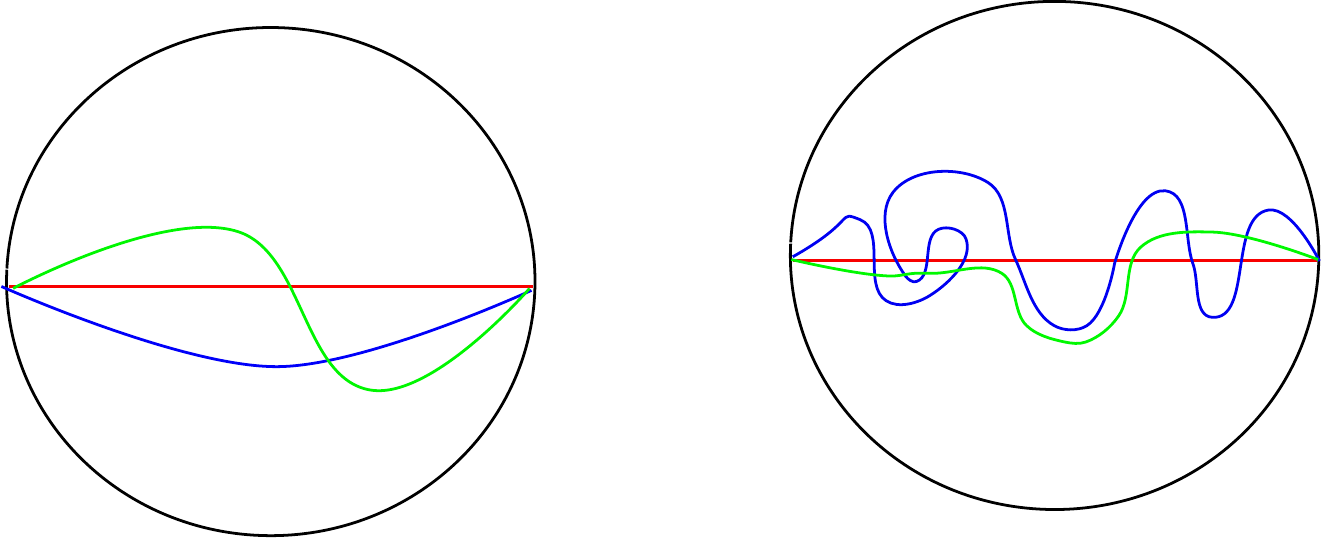
 \caption{  \label{threecasefig} }
\end{center}
\end{figure} 

Thus, proving $HF(A,B)=HF(A,D)$ reduces to the case when the lifts $\tB$ and $\tD$ intersect transversely in one interior point, as illustrated in Figure \ref{arcsfig} (after perhaps exchanging the notation for $B$ and $D$).
 \begin{figure}[h]
\begin{center}
\def\svgwidth{3.4in}
 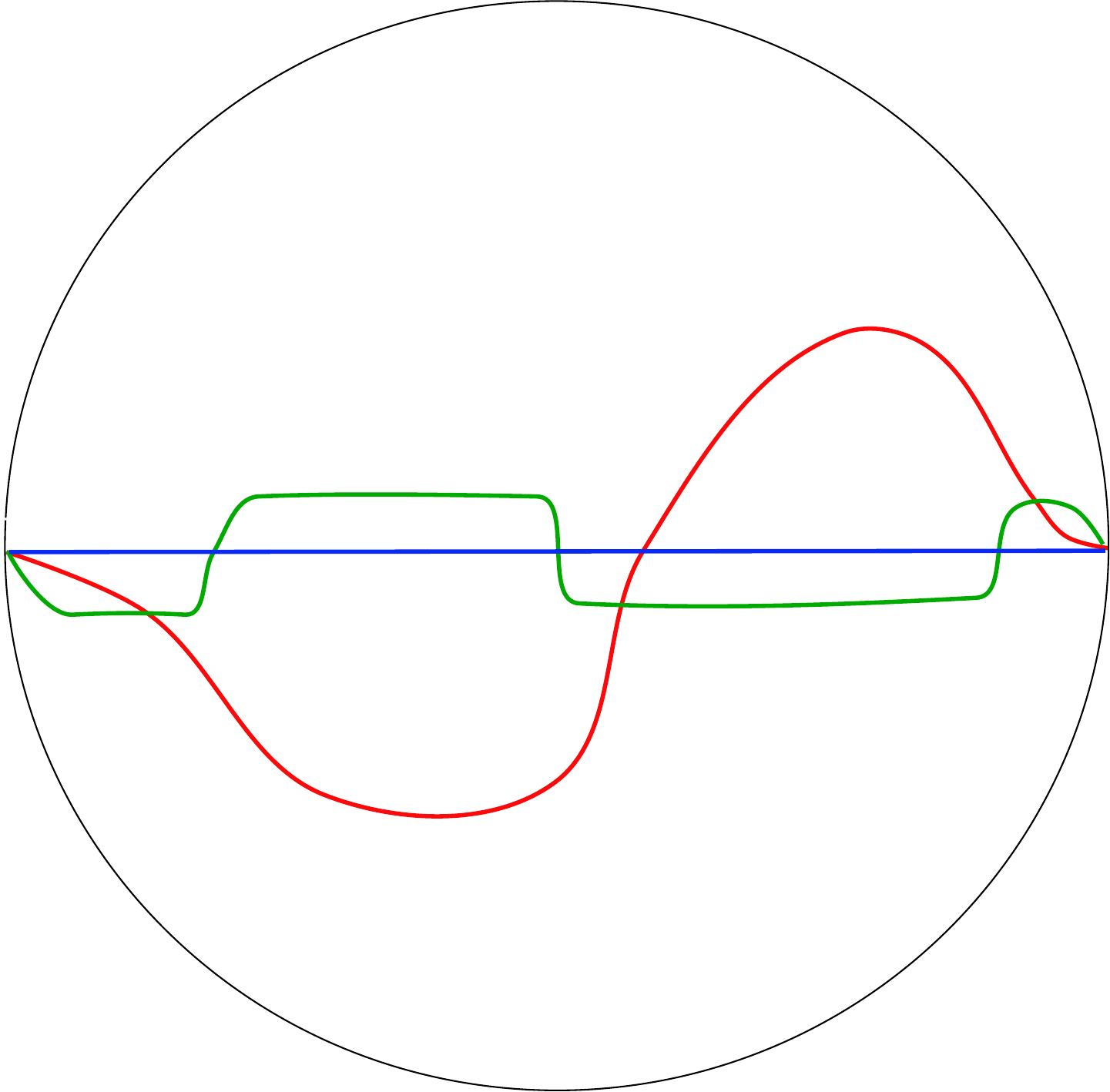
 \caption{  \label{arcsfig} }
\end{center}
\end{figure} 

  We choose an arc $\tD'$ close to and isotopic rel endpoints to $\tD$, as in Figure \ref{arcsfig}.  Not shown is the restricted immersed circle $\tA$, but we assume that there are no triple points of $A,B,D$.  This ensures that $\tA$ misses a neighborhood of $x$ containing $u$ and $w$. We also assume that  the bigons near the ends involving $e, g, f,$ and $z$ are small enough that   that $\tA$ misses them.   We use these facts, the fact that immersed $n$-gons lift to the universal cover,   and the $A_n$ relations to finish the proof of Lemma \ref{case3}.

The counterpart to Lemma \ref{welldef} is much simpler in the case of arcs.
Consider the ordered triple $(A,D,B)$. If $p\in A\cap D$, $q\in D\cap B$, and $r\in A\cap B$, and  then $\pi_2(p,q,r)$ is either empty or contains a single class.  If fact, if $(u_i,(\alpha_i,\gamma_i,\beta_i))\in \pi_2(p,q,r), ~i=1,2$, then since $B$ and $D$ are arcs we may use the homotopy extension property to assume $\beta_1=\beta_2$ and $\gamma_1=\gamma_2$. Then gluing $u_1$ to $u_2$ along $\alpha_i$ and $\beta_i$ shows that $\alpha_1\alpha_2^{-1}$ is nullhomotopic, so that we may assume further than $\gamma_1=\gamma_2$. Since $\pi_2(P^*)=0$, it follows that $u_1=u_2$ in  $\pi_2(p,q,r)$, As before, this implies that $\mathcal{M}(p,q,r)$ is finite, and hence $\mu_2:C(A,D)\times C(D,B)\to C(A,B)$ is defined.   
The same argument applies to the triple $(A,D,D')$ to show that $\mu_2:C(A,D)\times C(D,D')\to C(A,D')$ is defined, and to the 4-tuple $(A,D,B,D')$ to show that $\mu_3:C(A,D)\times C(D,B)\times C(B,D')\to C(A,D')$ is defined.

Using the same notation for  points in the universal cover and their images in $P^*$, we see that 
$\partial x= 0 \text{ in } C(D,B)$, hence the $A_2$ relation shows that 
$\mu_2(-,x):C(A,D)\to C(A,B)$ is a chain map.

Observe that $\partial (y+z)=2w=0$ in $C(B,D')$.  Also, 
$e+f=\mu_2(x, y+z)$ in $C(D,D')$.  Furthermore, $\partial(e+f)=2u=0$. Hence $\mu_2(-,e+f):C(A,D)\to C(A,D')$ is a chain map.  In fact , $\mu_2(-,e+f)$ is a chain isomorphism, since $\tD$ is close to $\tD'$ and $A$ misses neighborhoods of $\pm\infty$, so there is a unique intersection point $w'$ of $A$ with $D'$ for each intersection point $w$ of $A$ with $D$. If $w$ lies between $e$ and $u$, then $\mu_2(w,e)=w'$ and $\mu_2(w, f)=0$, and if $w$ lies between $x$ and $f$, then $\mu_2(w,f)=w'$ and $\mu_2(w, e)=0$.

Substituting these calculations into the $A_3$ relation  
$$
\begin{multlined}0=\mu_3(\partial w, x,y+z)+\mu_3( w, \partial x,y+z)+\mu_3(w, x,\partial (y+z))\\
+\mu_2(\mu_2(w,x),y+z)+\mu_2(w,\mu_2(x,y+z))+\partial\mu_3( w, x,y+z)
\end{multlined}$$
and defining $H:C(A,D)\to C(A,D')$ by $H(w)=\mu_3(w,x, y+z)$ yields
$$0=H\partial(w) +\mu_2(\mu_2(w,x),y+z)
 +\mu_2(w,e+f)+\partial H(w).
 $$
 So that the chain isomorphism $\mu_2(-,e+f):C(A,D)\to C(A,D')$ is chain homotopic to the 
 composite $C(A,D)\xrightarrow{\mu_2(-,x)}C(A,B)\xrightarrow{\mu_2(-,y+z)}C(A,D')$.
 Hence the chain map $\mu_2(-,x):C(A,D)\to C(A,B)$ is injective for all restricted immersed circles $A$ so that $(A,B)$ and $(A,D)$ are admissible.

 The reader can safely be left the task of showing that $\mu_2(-,x):C(A,D)\to C(A,B)$ is injective for all restricted immersed circles $A$,  by producing an embedded  $\tB'$ close to $\tB$ and constructing a right inverse to $\mu_2(-,x)$.    This is done by analogy with the symmetry  between Figures \ref{cylinderfig} and \ref{cylinder2fig},
The fact that the relative $\ZZ/4$ gradings is preserved is proved as before.

This completes the proof of Lemma \ref{case3}, and hence also the proof of Theorem \ref{rhi}. 
\qed

\begin{cor}\label{indepofep}
$HF(L_0^{\ep,g},L_1)$ is independent of $\ep>0$, the function $g$ and the homotopy class  of $L_1$.  \qed
\end{cor}
 
\section{Calculus}\label{special}
In this section, we make four technical but useful observations which  streamline the calculation of $HF(L_0,L_1)$ when $L_0=L_0^{\ep,g}$.   These calculations demonstrate the ease of working with the slope 1 line field $\ell_1$ and Equation (\ref{gradeeq}).

\subsection{} \label{pmpts} We show $gr(x_+,x_-)=1$ when $x$ is a transverse intersection of $L_1$ with the diagonal arc, and $x_+,x_-$ the two corresponding intersection points with $L_0=L_0^{\ep,g}$ for $\ep>0$ and $g$ small, as indicated in Figure \ref{plusfig}.  We take $\alpha_0$ to be  the path in $L_0$ starting at $x_+$ which heads down and to the left, around the bottom left corner, and back up to $x_-$.  The path $\alpha_1$ from $x_-$ back to $x_+$ is the short path which contains the diagonal intersection point $x$. 

In this case, $\mu(\alpha_0,\ell_1)=1$, $\mu(\alpha_1,\ell_1)=0$, $\tau(L_0,L_1,\ell_1)_{x_+}=0$,  $\tau(L_0,L_1,\ell_1)_{x_-}=1$, and $z(\alpha_0*\alpha_1)=1$ yielding $gr(x_+,x_-)=1$.
     \begin{figure}[h]
\begin{center}
\def\svgwidth{1.7in}
 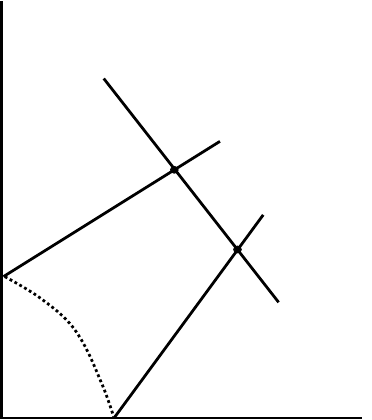
 \caption{ \label{plusfig} }
\end{center}
\end{figure}

\subsection{}\label{simp2} Suppose that $L_1$ intersects the diagonal arc $\Delta$ transversely and that $L_0=L_0^{\ep,g}$ for   $\ep>0$ and $g$  small. Suppose $p,q$ are intersection points between $L_0$ and $\Delta$.  Let $p_+,q_+$  be the corresponding intersection points with the part of $L_0$ in the front of $P$ which has slope slightly less than 1, as indicated in Figure \ref{plusfig}.

Then the arc $\alpha_0$  along $L_0$ from $p_+$ to $q_+$ can be chosen to lie entirely on the front face of the pillowcase, and has slope     slightly  than 1, hence $\mu( \alpha_0,\ell_1)=0$. Moreover, since the slope of $\ell_1$ equals 1,   $L_1$ is transverse to $\Delta$, and $\ep $ and $g$ are small, the triple index terms $\tau$ at $p_+$ and $q_+$ are both zero.  This leaves only the terms $\mu( \alpha_1,\ell_1)$ and $z(\alpha_0\alpha_1)$ in the formula for the grading.   Therefore,  {\em in this situation}
 \begin{equation}
\label{spcase}
gr(p_+,q_+)=\mu( \alpha_1 ,\ell_1)+z(\alpha_0*\alpha_1).
\end{equation}
A similar calculation applies to the other pair $p_-,q_-$.

\subsection{} If $R$ is an arc, $L_0\cap L_1$ may (and in our applications will) contain a distinguished point  which may be used to promote the relative grading to an absolute grading. To understand the meaning of the following lemma, the reader should locate the distinguished point $r_+^\ep$ in the examples illustrated below in Figures \ref{fig72}, \ref{fig34pert2}, \ref{fig35pert}, and \ref{fig45pert}.

\begin{lem} \label{dist}  Suppose that  $L_1:R\to P$ is a restricted immersed arc and one of the endpoints  (which we denote by $r_+$),   maps to the corner $(0,0)$ of $P$ with limiting slope bounded away from 1. Then for all small $\epsilon$, there is a unique continuously varying intersection point $r_+^\epsilon$ of $L_0^\epsilon$ and $L_1$ satisfying $\lim_{\epsilon\to 0}r_+^\epsilon=  r_+$.   
 \end{lem}
\begin{proof}
 This follows simply from the requirement that limiting slope bounded away from 1.\end{proof}

If $L_1$ satisfies the hypothesis of Lemma \ref{dist}, then, given the additional data of a choice of  $\sigma\in \ZZ/4$, one can endow  $HF(L_0,L_1)$ with the absolute $\ZZ/4$ grading which places $r_+^\epsilon$ in grading $\sigma$. 

In our applications below,    $L_1$  will be associated to a knot $K$ in a homology 3-sphere. In this setting the hypotheses of Lemma \ref{dist} hold and we take $\sigma$ to be the signature of the knot $K$.

\subsection{}  A special class of restricted immersed circles, which we call vertically monotonic,   arise in many of our examples and have particularly simple Lagrangian-Floer homology.

\begin{df} Suppose that $L_1:R\to P^*$ is a restricted immersed circle, with domain parameterized by $[0,2\pi]$, and let  
 $\tilde L_1=(\gamma(t),\theta(t)):[0,2\pi]\to \RR^2$ denote its lift to the branched cover (\ref{brcover}).

\begin{itemize}

\item  Call   $L_1:R\to P^*$  {\em vertically monotonic} if      $\tilde L_1$   misses the vertical line segments $\gamma=k\pi, k\in \ZZ$, and if its tangent slope satisfies $|\frac{d}{dt}\tilde L_1|>1$.  
Thus $L_1$ winds     around the pillowcase without intersecting the left or right edges, and is everywhere transverse to the line field $\ell_1$ (as well as the slope $-1$ line field).

\item Define the {\em vertical degree of $L_1$} to be the absolute value of the difference of the vertical coordinates of $\tilde L_1(2\pi)$ and $\tilde L_1(0)$. Thus$$d=\frac{1}{2\pi}| \theta(0)-\theta(2\pi)|.$$ Notice that the vertical degree is a homotopy invariant. In particular, it is well defined for any circle homotopic to a vertically monotonic circle.

 \end{itemize}
\end{df}

\begin{prop} \label{vm}  Suppose that $L_1:R\to P^*$ is a vertically monotonic restricted immersed circle. Then the vertical degree $d$ of $L_1$  is even. Moreover, for $\epsilon$ small enough, all differentials are zero      and $HF(L_0,L_1)$ has rank $\frac{d}{2}$ in each of the $4$ grading degrees. 
\end{prop}
\begin{proof}  Since $L_1:R\to P^*$ is vertically monotonic, it is transverse to the line field $\ell_1$ and hence $\mu(L_1(R),\ell_1)=0$.  Moreover, $L_0^{\ep,g}$ is transverse to $L_1$ for all small enough $\ep, g$. Fix a transverse $L_0=L_0^{\ep,g}$ with $\ep,g$ small.

Since $\tilde L_1:[0,2\pi]\to \RR^2$   misses the vertical line segments $\gamma=k\pi$,   $z(L_1(R))$ is equal to twice the vertical degree $d$ of $L_1$. Since $z(L_1(R))\equiv -\mu(L_1(R))$ mod 4, $2d\equiv 0$ mod 4 so that $d$ is even.

Let  $p,q$ be intersection points of $L_0$ and $L_1$ and suppose that there were a bigon $(u,\alpha_0, \alpha_1)$ from $p$ to $q$. The bigon misses the corners of $P$ and hence lifts to a bigon $(\tu, \tilde \alpha_0, \tilde \alpha_1)$ in $\RR^2$ with one edge along the preimage of $L_1$ and one along the preimage of $L_0$.  

For $\ep>0 $ sufficiently small, the connected components of the preimage in $\RR^2$ of $L_0$ are very close to lines of slope $1$ through $(0,k\pi)$, and hence $\tilde \alpha_0$ is nearly a straight segment of slope 1. On the other hand, since the tangent lines to $\tilde L_1$ have slope bounded away from 1, the lift $\tilde \alpha_1$, which  starts at $\tilde \alpha_0(1)$, cannot terminate at $\tilde \alpha_0(0)$, contradicting the fact that $\tilde \alpha_0$ and $\tilde \alpha_1$ bound a bigon.
 Hence, for sufficiently small $\epsilon$,  there are no bigons, so that all differentials are zero.  The circle $L_1(R)$ winds monotonically around  $P$, intersecting the diagonal arc $\Delta$ in $d$ points, and hence intersecting $L_0$ in $2d$ points, so that $HF(L_0,L_1)$ has rank $2d$.

It remains to calculate the relative gradings. Since $L_1$ is transverse to $\ell_1$, Equation \ref{spcase} shows that this reduces to calculating $z$. One vertical wind around $P$ encircles two corner points and therefore changes $z$ by $2$ mod 4. Hence the generators come in  $\frac{d}{2}$ pairs, and, as explained in Section \ref{pmpts} and indicated in Figure \ref{plusfig}, alternate between contributing in (relative) gradings $0,1$ and $2,3$. This completes the argument.
\end{proof}

  Proposition  \ref{vm}  can be strengthened by combining it with Corollary \ref{indepofep} as follows.

\begin{thm}
 Suppose that $L_1:R\to P^*$ is a restricted immersed circle which misses the left and right edges of the pillowcase. Then the vertical degree $d$ is defined and even, and  for any $\ep>0$ and $g$, $HF(L_0^{\ep,g}, L_1)$ has rank $ \tfrac d 2$ in each of the 4 grading degrees.
\end{thm}
\begin{proof} Since $L_1$ misses the right and left edges of the pillowcase, it is homotopic to a vertically monotonic immersed circle. The proof is therefore a consequence of Corollary \ref{indepofep}.
\end{proof}

\section{Traceless representation varieties of 2-stranded tangles}\label{2strand}

Hereafter,   a {\em restricted immersed 1-manifold} is a disjoint union  of a single restricted immersed arc and a finite number of restricted immersed circles.  This prompts us the make the following changes in our notation.  The first immersed curve, $L_0=L_0 ^{\ep,g}:S^1\to P$  henceforth has domain the circle, and the domain of $L_1$ has multiple components,  so we now use  the symbol $R_0$ to denote the arc component for $L_1$. More precisely, suppose $R$ is the disjoint union of an arc $R_0$ and finitely many circles $R_1,\dots,R_n$ $L_1:R\to P$ a map so that the restriction of $L_1$ to $R_0$ is a restricted immersed arc and the restriction to each $R_i,~i\ge 1$ is a restricted immersed circle.   The Lagrangian-Floer homology $H^\nat(L_0,L_1)$ is defined to be the direct sum of the homologies $HF(L_0^{\ep, g},L_1|_{R_i})$, where $\ep,g$ are chosen so that $(L_0^{\ep,g},L_1|_{R_i})$ form an admissible pair  for each $i$.    

Each summand is relatively $\ZZ/4$ graded, although initially there is neither a relative nor absolute  $\ZZ/4$ grading on all of $H^\nat(L_0,L_1)$ if $L_1$ is not connected.  The notation $H^\nat$ is adopted in order to indicate the relationship to the reduced instanton homology $IH^\nat$.

\bigskip

We now introduce traceless character varieties.  
We identify $SU(2)$ with the set of unit quaternions throughout. For the basic properties of the unit quaternions and their Lie algebra, we refer the reader to Section 2 of \cite{HHK}. A {\em traceless quaternion} means a quaternion with zero real part.

Given a pair $(A,B)$ consisting of a (compact) manifold and a properly embedded codimension 2 submanifold, call a  representation $\pi_1(A\setminus  B)\to SU(2)$   {\em traceless} if it sends all meridians of $B$ to the conjugacy class $C(\bbi)$ of $\bbi$. Define the {\em traceless character variety} (or traceless flat moduli space)
 $$R(A,B)=\{\rho:\pi_1(A\setminus B)\to SU(2)~|~\rho ~\text{is traceless}\}/_{\text{conjugation}}$$
 An embedding of pairs $(A_1,B_1)\subset (A_2,B_2)$ induces a restriction map $R(A_2,B_2)\to R(A_1,B_1)$.

\medskip

 Consider  a decomposition of a pair $(X,K)$, where $K$ is a knot (or link) in a homology 3-sphere $X$, and $X$ contains a    separating  2-sphere $S\subset X$ which intersects $K$ transversally in four points.    We assume that one of the two regions $S$ bounds is a 3-ball $D$, and that $D\cap K$ is a standard trivial 2-stranded tangle.  
\begin{equation}
\label{decom}
(X,K)=(Y,T)\cup _{(S,\{a,b,c,d\})}(D,U)
\end{equation}
We refer to (\ref{decom})   as {\em a 2-tangle  decomposition associated to the knot $(X,K)$.} 
We fix an identification $D=B^3$ so that $$U=\left\{\left(\pm\tfrac{1}{\sqrt 2},0,t\right)~\left. |~t\in\left[-\tfrac {1}{ \sqrt{ 2}},\tfrac {1}{ \sqrt{2}}\right]\right. \right\}$$
and fix an identification $ \partial(D,U)=(S^2,\{a,b,c,d\})$ with $a=\left(\tfrac 1 {\sqrt{2}}, 0,\tfrac 1 {\sqrt{2}}\right)$, $b=\left(-\tfrac 1 {\sqrt{2}}, 0,\tfrac 1 {\sqrt{2}}\right)$, $c=\left(-\tfrac 1 {\sqrt{2}}, 0,-\tfrac 1 {\sqrt{2}}\right)$, $d=\left(\tfrac 1 {\sqrt{2}}, 0, - \tfrac 1 {\sqrt{2}}\right)$.
We call  $(Y,T)$ a {\em 2-tangle associated to the knot $(X,K)$}.    

Observe that to recover $(X,K)$ from $(Y,T)$ requires only a choice of identification 
$$\iota: (Y,T)\cong(S^2,\{a,b,c,d\})$$ (in    the same way that a Dehn filling is determined by a manifold with torus boundary and an identification of its boundary with the boundary of a solid torus). 
 We will omit the choice of $\iota$ from the notation since the identification will be clear from context.

To a 2-tangle decomposition of $(X,K)$, one associates the diagram 
 \begin{equation}\label{SVKD}
\begin{diagram}\node[2]{R(S,\{a,b,c,d\})}\\
\node{R_{\pi}(Y,T)}\arrow{ne,t}{L_1}\node[2]{R^\nat_{\pi_\epsilon}(D,U)}\arrow{nw,t}{L_0}\\
\node[2]{R_{\pi''}^{\nat}(X,K)}\arrow{nw}\arrow{ne}
\end{diagram}
\end{equation}
 where $\pi, \pi_\epsilon$ refer  to certain {\em holonomy perturbations}  $\pi''=\pi\cup \pi_\ep$, and $R_\pi(M,L)$ denotes the corresponding {\em $\pi$-perturbed  traceless character variety}.   
 
Holonomy perturbations are used to make the Chern-Simons function (whose Morse theory defines instanton homology) have only non-degenerate critical points. They are constructed and explained in the context of the traceless character varieties in \cite[Section 7]{HHK} and also in Section \ref{perturbationsection} below. They were introduced in gauge theory by by Donaldson, Floer, Taubes, and others \cite{Donaldson1, Taubes,Floer}.

  It is not necessary for this article to understand precisely what the $\nat$ superscript means beyond knowing the statement of Theorem \ref{fig8} below. But roughly,
$R^\nat_{\pi_\epsilon}(D,U)$ refers to  $R_{\pi_\epsilon}(D,U\cup E)$, where $E$ is an additional small meridian component to one of the components of $U$ and one considers representations which come from flat connections on an $SO(3)$ bundle with $w_2$  dual to an arc spanning $U$ and $E$.  This construction,  introduced by Kronheimer-Mrowka  in  \cite{KM1}, is an ingredient in  the definition of  reduced instanton knot homology. 
 We refer to \cite{KM1,HHK} for the details. 
 
 \medskip

  The space $R(S^2,\{a,b,c,d\})$ is a pillowcase.  Indeed the following simple proposition is proved in \cite{HHK} (and elsewhere). In the statement we abuse notation and let $a,b,c,d$ also denote the oriented meridians of the four punctures.  
 
\begin{prop}[\cite{HHK}, Proposition 3.1]\label{pillow}
   There is a surjective quotient map
$$\psi: \RR^2\to R(S^2,\{a,b,c,d\})$$ 
given by 
$$\psi(\gamma, \theta):  a\mapsto \bbi,~ b\mapsto e^{\gamma\bbk}\bbi,~ c\mapsto e^{\theta\bbk}\bbi, ~d\mapsto e^{(\theta-\gamma)\bbk}\bbi.$$ 
The map $\psi$   induces a homeomorphism of the pillowcase $P$ with $R(S^2,\{a,b,c,d\})$.
The  four corner points are the  image under $\psi$ of the lattice $(\pi\ZZ)^2$,  and correspond  to abelian non-central representations.  All other points correspond  to  non-abelian representations.\qed
\end{prop}
We urge the reader not to confuse $(S^2,\{a,b,c,d\})$ with the pillowcase $P=R(S^2,\{a,b,c,d\})$, a homeomorphic space!

 \medskip
 
 The only fact we will need to recall about $R^\nat_{\pi_\ep}(D,U)$ is the following, which follows immediately  by combining    \cite[Theorem 7.1]{HHK} with  Theorem \ref{collar1}, proved below.

 \medskip
\begin{thm} \label{fig8} Given any $\ep>0$ and   $g\in \mathcal{X}$, there is a  holonomy perturbation $\pi_\ep$ depending on $\epsilon$  and $g$ so that $R^\nat_{\pi_\ep}(D,U)$ is a circle, and the restriction to the pillowcase (the northwest map in Diagram (\ref{SVKD})) is given by a map $L_0^{\ep, g}$ of Definition \ref{elzero}.
\qed\end{thm}

What Proposition \ref{pillow} and Theorem \ref{fig8} tell us is that a decomposition of a knot or link into two 2-tangles, one  of which  is trivial,  gives the pillowcase $P$ and a map $L_0:S^1\to P^*$. 

 \medskip
 
  The remaining input needed to define  a   Lagrangian-Floer homology as in Section \ref{LFT} is provided by the 1-manifold  $R=R(Y,T)$ and $L_1:R\to P$ the restriction map $R(Y,T)\to R(S^2,\{a,b,c,d\})$, as indicated in Diagram (\ref{SVKD}).    

Loosely speaking, $L_1:R(Y,T)\to R(S^2,\{a,b,c,d\})$ is generically    a union of a restricted immersed arc and some number of restricted immersed circles. The arc arises as one component of the space of traceless binary dihedral representations of $\pi_1(Y\setminus T)$, and the endpoints correspond to the two conjugacy classes of abelian traceless representations (Theorem  3.2 of \cite{FKP}.)

It is not always literally  true, however, that $L_1:R(Y,T)\to P$ is a restricted immersed 1-manifold. It is true for certain tangles associated to  2-bridge knots \cite[Section 10]{HHK}, and for some, but not all torus knots \cite{FKP}.

In fact, there exist decompositions of knots for which $R(Y,T)$, rather than being a smooth 1-manifold,  is instead a singular real algebraic variety of dimension greater than or equal to 1.
For example in \cite[Section 11]{HHK} (see Figure \ref{fig34}) it is shown that for a tangle associated to the $(3,4)$ torus knot, $R(Y,T)$ is a singular 1-dimensional variety, homeomorphic to the letter $\phi$. Many more examples are given in \cite{FKP}. In general one can construct   examples so that $R(Y,T)$ is highly singular and has strata of high dimension by placing local knots in one of the strands of a 2-tangle.  Hence the traceless character varieties $R(Y,T)$ must first be desingularized before we can apply the construction of the Lagrangian-Floer theory in the pillowcase.  In order to preserve the relationship to gauge theory and instanton homology, we use holonomy perturbations to smooth $R(Y,T)$.

The space $R(Y,T)$ for a certain natural  2-tangle decomposition of a torus knot is   typically singular  (\cite{HHK,FKP}).     We prove below that any torus knot admits a 2-tangle decomposition    and an arbitrarily small holonomy perturbation  $\pi$ so that  $R_\pi(Y,T)$ is a compact 1-manifold with two  boundary points, and $L_1$ an immersion which satisfies all the requirements to be a restricted immersed 1-manifold except possibly the requirement that it have no fishtails. Based on  index calculations and examples, it is reasonable to expect that   for any knot, arbitrarily small holonomy perturbations exist which make $L_1:R_\pi(Y,T)\to P$ a restricted immersed 1-manifold.

\begin{conj}\label{con1}  For any 2-tangle $(Y,T)$ in the 3-ball (or a homology 3-ball), there exist  arbitrarily small holonomy perturbations $\pi$  so that
 $R_\pi(Y,T)$ is a compact 1-manifold with two boundary points and the restriction map $L_1:R_\pi(Y,T)\to R(S^2,\{a,b,c,d\})$ is a restricted immersed 1-manifold  on each component  in the sense of Definition \ref{RLP}.  
 \end{conj}

Given a 2-tangle decomposition of a knot and a perturbation $\pi$  which satisfies the conclusion of Conjecture \ref{con1},  denote by 
$H^\nat(Y,T,\pi)$ the resulting  Lagrangian-Floer homology of the complex $C(L_0,L_1)$.
We will simplify this to $H^\nat(Y,T)$ if the  perturbation $\pi$ is clear from context.
\medskip

As explained in \cite{HHK}, if $\ep\ne 0$ is small and $L_0=L_0^{\ep,0}$ intersects $L_1$ transversely, then the  intersection points of   $L_0 $ and $ L_1$ also form generators of the reduced instanton homology $I^\nat(X,K)$. Theorem \ref{collar1} implies that this holds for $L_0=L_0^{\ep,g}$ for any   small   $g\in\mathcal X$.

We state this formally.
\begin{prop}\label{instanton} Given a small perturbation $\pi$ which makes $L_1:R_\pi(Y,T)\to P$ a restricted immersed 1-manifold and a transverse $L_0=L_0^{\ep,g}$ with $\ep$ and $g$ small, there is a (possibly different) differential
$$\partial_{KM}:C(L_0,L_1)\to C(L_0, L_1)$$ so that the homology of $(C(L_0,L_1),\partial_{KM})$ is   the reduced instanton homology $I^\nat(X,K)$. \qed
\end{prop}
The differential $\partial_{KM}$ is defined by Kronheimer-Mrowka in terms of singular instantons on cylinders $(X,K)\times \RR$.  There is a well known procedure for producing  approximate instantons from bigons in character varieties associated to lagrangian intersection   diagrams; see for example \cite[Section 4]{wehrheim}.   It is therefore not unreasonable to conjecture that there is a relationship between   $I^\nat(X,K)$ and $H^\nat(Y,T,\pi)$. Indeed, we have found these to be isomorphic  in every example we have computed.  We extend Conjecture \ref{con1} to an optimistic ``Atiyah-Floer'' type conjecture:
 
\begin{conj}\label{con3} Given a knot $(X,K)$ in a homology 3-sphere,  there exists a 2-tangle decomposition as in Equation (\ref{decom}), such that for suitably small generic perturbations $\pi$,  $L_1:R_\pi(Y,T)\to P$ is a restricted immersed 1-manifold and  $H^\nat(Y,T,\pi)$ is isomorphic to the reduced instanton homology $I^\nat(X,K)$. \end{conj}

In the remainder of this article we establish some partial results and carry out calculations which provide evidence for these conjectures.   The reader should realize, however, that there are no non-zero differentials in $C(L_0,L_1)$ between generators which lie on different path components of $L_1$. We know of no reason why this should be true for  the instanton complex.  It is likely that there are differentials in the instanton complex which don't appear in $C(L_0,L_1)$. For example, the pairs of generators $p_+,p_-$ near each intersection point $p$ of $L_1$ with the diagonal arc $\Delta$, described in Section \ref{pmpts}, arise from a holomony perturbation which ``tilts'' a Bott-Morse circle of critical points of the Chern-Simons function \cite{HHK}.    Analogy with finite-dimensional Morse theory suggests that there exists a  cancelling  pair of gradient flow lines (i.e., instantons) from $p_+$ to $p_-$ in the Kronheimer-Mrowka instanton complex, whereas there are no bigons   connecting these points of intersection.   

 \subsection{Absolute grading}\label{abgr}  We remark that, by construction, $H^\nat(Y,T,\pi)$ splits as the direct sum over the path components $R_0,R_1,\dots R_n$ of $R(Y,T)$:
 $$H^\nat(Y,T,\pi)=\oplus_{i}H^\nat(L_0,L_1(R_i)).$$
 and that each of the summands admits a relative $\ZZ/4$ grading.  The relative grading of the summand corresponding to the arc component $R_0$ can be promoted to an absolute $\ZZ/4$ grading for small perturbations, using the knot signature,  as follows.  

Assume that $(X,K):=(Y,T)\cup_{\iota}(D,U)$ is a 2-tangle decomposition of a knot $K$ in an integer homology sphere $X$. The signature of $K$, $\sigma(K)$, is an even   integer.   There are two traceless abelian representations of $\pi_1(Y\setminus T)$, $r_+$ and $ r_-$ distinguished by the property that $r_+$ extends to $\pi_1(X\setminus K)$ (and $r_-$ does not).  The point $r_+$ is a Morse critical point of the Chern-Simons function, and   a regular point of $R(X,K)$.  In particular, it remains regular after  small perturbations.

The points $r_+$ and $r_-$ are endpoints of an embedded arc of binary dihedral representations, which,  by Theorem 3.2 of \cite{FKP},  is the image in the pillowcase under the branched cover (\ref{brcover})  of an embedded linear segment joining two lattice points. This line segment has slope different from 1  (the slope is different from 1 since the 2-fold branched cover of a knot in a homology sphere is a rational homology sphere, so the integer $h(bc^{-1}) $ in Theorem 3.2 of \cite{FKP} is non-zero).  In particular, the arc of binary dihedral representations  is properly immersed (in fact, embedded) in $P$.

Small perturbations only change the slopes near the endpoints slightly, and one can keep them  bounded away from 1. By Lemma \ref{dist}  there is a unique intersection point $r_+^\ep$ of $R(Y,T)$ and $L^{\ep,g}_0$ for all small $\ep,g$.   
We promote the relative grading of the subcomplex corresponding to the component $R_0$ by declaring 
\begin{equation}\label{sigma}
gr(r_+^\ep)=\sigma(K)
\end{equation}
for small perturbations.

 We have not found an elementary approach to promote the relative grading of the generators of the subcomplexes associated to the circle components $R_i, i>0$, 
 and so we will use the following awkward definition as a consequence of Proposition \ref{instanton}:
 choose a generator on each circle component and declare its absolute grading to be the one assigned to it by Kronheimer-Mrowka in \cite[Proposition 4.4]{KM-khovanov}.

 \medskip
 
A proof that the  relative $\ZZ/4$ grading of  generators of $C(L_0,L_1)$   (Definition \ref{grading}) coincides with the   grading assigned the these generators (by Proposition \ref{instanton})   of singular instanton knot homology by \cite[Proposition 4.4]{KM-khovanov} is given by using splitting theorems for spectral flow \cite{nic, DK, BHKK}. We outline how this is done, referring to \cite{BHKK} for details.

First, the relative grading is defined to be the mod 4 reduction of the spectral flow of the Hessian of the Chern-Simons function (acting on singular connections)  along a path   joining a pair $a_0,a_1$ of critical points, i.e., perturbed flat connections.   If the restrictions of $a_0,a_1$ to $Y\setminus T$   can be joined by a smooth path of flat connections, i.e., by a smooth path in $R_\pi(Y,T)$,  then the approach of  \cite[Theorem 3.9]{BHKK} can be modified to show that the spectral flow 
equals the    Maslov index along the path of the tangent space of the immersed 1-manifold  $R_\pi(Y,T)\to P^*$ in the pillowcase, with respect to some {\em a priori} unknown   line field $\ell_{\operatorname{inst}}$, and hence is given as in  Definition \ref{grading}.

Changing the homotopy class of a line field  determines a difference class 
 $z\in H^1(P^*;\ZZ/4)$.   
The  identification of $\ell_{\operatorname{inst}}$ is therefore equivalent to the identification of $z$.   Its identification with the explicit class   of Definition \ref{theclassz} is completed by calculating a few examples of 2-bridge knots, whose instanton homology is known, to deduce the values of $z$ on a basis of 1-cycles in $H_1(P^*)$.

\medskip
 
This argument, combined with the additivity of spectral flow under composition of paths of self-adjoint operators, also shows that if $L_1:R_i\to P^*$ is an immersion of a smooth circle component $R_i\subset R_\pi(Y,T)$, then $L_1$ satisfies the condition $\mu(L_1(R_i),\ell_{\operatorname{inst}})\equiv 0$ mod 4 required of restricted immersed circles.   In this case, one uses the fact that the two smooth paths in $R_i$ joining $a_0$ to $a_1$ must give the same relative $\ZZ/4$ grading, since the relative grading in the singular instanton complex is well defined (and independent of the tangle decomposition).  As the proof of Proposition \ref{grade} shows, this is only possible if $L_1:R_i\to P^*$  satisfies $\mu(L_1(R_i),\ell_{\operatorname{inst}})\equiv 0$ mod 4.

 \section{Examples: 2-bridge knots}\label{2b}
 
   Two-bridge knots can be described as the union of two trivial tangles along a 4-punctured sphere.  We recall some of the results about their tangle decompositions from   \cite{HHK}.  In particular, we will show that  for such a tangle decompositions of a 2-bridge knot $K$, $L_1:R(Y,T)\to P$ is a restricted immersed (in fact linearly embedded) arc which meets $L_0^{\ep,0}$ transversely in $\det(K)$ points, and that all differentials in the Lagrangian-Floer complex are zero.  
   
These facts, together with the identification of the relative $\ZZ/4$ gradings in $C(L_0,L_1)$ and the singular instanton complex via a spectral flow splitting theorem as explained above,  imply that $H^\nat(Y,T)$ is isomorphic to the reduced instanton homology $I^\nat(S^3,K)$, which is known \cite{KM-khovanov} to equal the reduced Khovanov homology $Kh^{red}(K^m)$ of the mirror of $K$ for a 2-bridge knot. 
We conjectured in \cite{HHK} that placing the distinguished generator $r_+^\ep$ in grading $\sigma(K)$ agrees with Kronheimer-Mrowka's absolute grading \cite[Proposition 4.4]{KM-khovanov} (a conjecture   borne out in all our calculations) and, modulo this point,  for 2-bridge knots,  $H^\nat(Y,T)$, $I^\nat(S^3,K)$, and $Kh^{red}(K^m)$ (with its bigrading $(i,j)$ reduced to $i-j+1$ mod 4) contain the same information. In particular, Conjecture \ref{con3} holds for 2-bridge knots.

 \medskip

 Suppose that $(p,q)$ are relatively prime integers, with $p$ odd and positive. Then there is a 2-bridge knot $K= K(p,q)$ determined by the condition that its 2-fold branched cover is the lens space $L(p,q)$.  In  \cite[Figure 13]{HHK},     a 2-tangle decomposition $(S^3,K)=(Y,T)\cup(D,U)$ determined by a continued fraction expansion of $\frac{p}{q}$ is described. It is proved that $R(Y,T)$ is a smooth arc and the restriction map to the pillowcase is given, in $\RR^2$  coordinates, by 
 $$R(Y,T)\cong[0,\pi]\ni t\mapsto (qt, (q-p)t)\in P.$$

 Thus $L_1:R(Y,T)\to P$ is a restricted embedded arc.  In particular, no perturbation $\pi$ is needed to smooth $R(Y,T)$. Hence we can choose $L_0=L_0^{\ep,0}$ for a small $\ep$  and form the chain complex $C(L_0,L_1)$.  Since $L_1$  maps in linearly and $L_0$ is close to the linear arc $\Delta$, there can be no immersed bigons, and therefore all differentials are zero.

 There are $p$ intersection points of $L_1$ with $L_0$. In fact,   there are $\frac{p+1}{2}$ intersection points  of $(\gamma(t),\theta(t))=(qt, (q-p)t),~t\in [0,\pi]$ with the arc $\Delta$, these  occur  at  
 \begin{equation}
\label{points}x_\ell=\left(q \tfrac{2\pi \ell}{p}, (q-p)\tfrac{2\pi \ell}{p}\right),~ \ell =0,1,\dots, \tfrac{p-1}{2}.
\end{equation}
The points $x_\ell,~\ell>0$, each give rise to a pair of intersection  points $x_\ell^+,x_\ell^-$ of $L_1$ with $L_0^\epsilon$, and the point $x_0$ gives rise to the distinguished point $r_+^\epsilon$ of Lemma \ref{dist}. Hence $C(L_0,L_1)$ and $H^\nat(Y,T)$ have rank $p$. 
The  intersection points are illustrated in the case of $K=K(11,-5)$ ($7_2$ in the knot tables) in Figure \ref{fig72}.

\medskip

We show how to calculate the gradings.  First, the observation of Section \ref{pmpts} shows that $gr(x_\ell^+,x_\ell^-)=1$.  Next, recall that we promote the relative grading to an absolute grading by setting $gr(r_+^\epsilon)=\sigma(K)$ mod 4, where $\sigma(K)$ denotes the signature of the knot $K$.  Thus  the determination of all other gradings is reduced   to calculating $gr(r_+^\epsilon,x_\ell^+)$ for $\ell=1,\dots,p$.

  The  slope $\frac{q-p}{q} $ is not equal to $\pm 1$ since $p$ and $q$ are relatively prime.   There are four different cases to be considered, depending on the slope.  For simplicity we assume $\frac{q-p}{q} $ is positive and greater than 1; the other cases are treated similarly.

For each $x_\ell^+,\ell=1,\dots ,p$, one can find a path $\alpha_0$ in $L_0$ from $r_+^\ep$ to $x_\ell^+$ which lies on the front of the pillowcase.  One can then take a path $\alpha_1$ in $L_1$ from  $x_\ell^+$ back to $r_+^\ep$. Notice that since $L_1$ is an arc, the path $\alpha_1$ is unique.  

We are in the situation explained in Section \ref{simp2} and can calculate gradings using the  Equation (\ref{spcase}).   Since the  tangent line to $L_1$ is everywhere transverse to the slope one line field $\ell_1$, Equation (\ref{spcase}) simplifies further to $gr(r_+^\ep,x_\ell^+)=z(\alpha_0\alpha_1) $.

The same formula holds when  $\frac{q-p}{q} <-1$.  When $-1<\frac{q-p}{q}<1  $, an entirely similar calculation yields $gr(r_+^\ep,x_\ell^-)=z(\alpha_0\alpha_1)$. 

We summarize:
\begin{thm}\label{2bkt}
Let $K=K(p,q)\subset S^3$ be  a    2-bridge knot with $p>0$ odd,  and equip it with    the 2-tangle decomposition   described in   \cite[Figure 13]{HHK}. 
Then $L_1:R(Y,T)\to P$ is a linearly embedded arc of slope $\tfrac{q-p}{q}$, and hence a restricted immersed arc. Taking $L_0=L_0^{\ep,0}$ with $\ep>0$ small, 
$C(L_0,L_1)$ has rank $p$,  generated by the points $r^\ep_+$ and $x_\ell^+,x_\ell^-,\ell=1,\dots \tfrac{p-1}{2}$. The  $\ZZ/4$ grading is determined by 
 $$
gr(r^\ep_+)=\sigma(K),~gr(x_\ell^+,x_\ell^-)=1$$ and, letting
$x_\ell^\circ$ denote $x_\ell^+$ 
if  
$|\tfrac{q-p}{q}|>1$ and $x_\ell^-$ if  
$|\tfrac{q-p}{q}|<1$
$$ 
gr(r_+^\ep,x_\ell^\circ)=
z(\alpha)$$ where $\alpha$ is the loop in $P^*$ which starts at $r_+^\ep$, follows $L_0$ to $x_\ell^\circ$, then returns to $r_+^\ep$ along $L_1$, and $z\in H^1(P^*;\ZZ/4)$ is the class of Definition \ref {theclassz}. All differentials are zero and hence $C(L_0,L_1)\cong H^\nat(Y,T)$.\qed
\end{thm}

  \begin{figure}[h] 
\begin{center}
\def\svgwidth{2.1in}
 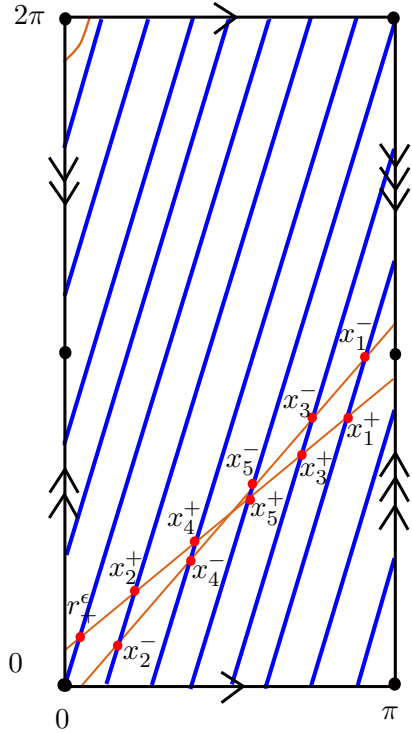
 \caption{The intersection of $L_0$ and $L_1$ in $P$ for the 2-bridge knot $K(11,-5)$. \label{fig72}}
\end{center}
\end{figure}

The knot $K(11,-5)$ has signature $\sigma=2$, and hence $gr(r_+^\ep)=2$.   The generators are illustrated in Figure \ref{fig72}.   The loop $\alpha$ which follows $L_0$ from $r_+^\ep$ to $x_1^+$ on the front of the pillowcase and then follows $L_1$ back to $r_+^\ep$ satisfies $z(\alpha)=2$. Hence
$gr(r_+^\ep, x_1^+)=2$.
A similar calculation applies to the other $x_\ell^+$ and yields  
$$ gr(r_+^\ep, x_2^+)=1 , gr(r_+^\ep, x_3^+)=0, 
gr(r_+^\ep, x_4^+)= 3,\text{~and~} gr(r_+^\ep, x_5^+)=2.$$
The gradings $gr(r_+^\ep,x_\ell^-)$ are computed using the fact that $gr(x_\ell^+,x_\ell^-)=1$.
Thus, in the notation introduced above,  $H^\nat(Y,T)=(3,2,3,3)$.

The choice $(p,q)=(11,6)$ gives a different tangle decomposition for the same knot $K=K(11,-5)=K(11,6)$.  The resulting homology is again $(3,2,3,3)$.

 The choice $(p,q)=(5,-3)$ yields a tangle decomposition of the Figure 8 knot. The map $L_1$ is illustrated (with different notation) in \cite[Figure 16]{HHK}. There are 5 generators, $r_+^\ep, x_1^\pm, x_2^\pm$, and computing gradings using $z$ yields $gr(r_+^\ep)=\sigma(K)=0 , gr(x_1^+)=3, gr(x_2^+)=2$, and hence $H^\nat=(1,1,2,1)$. This agrees with the calculation of reduced instanton homology and reduced Khovanov homology of the Figure 8 knot. 
 Choosing $(p,q)=(5,2)$ gives a different tangle decomposition for the Figure 8 knot, but again yields  $H^\nat=(1,1,2,1)$.
 
 The trefoil knot corresponds to $(p,q)=(3,-1)$; one calculates $H^\nat=(1,0,1,1)$. The same answer is obtained when taking instead $(p,q)=(3,2)$.

 \medskip
 
Theorem \ref{2bkt} can easily be used (and implemented in a computer algebra program) to compute $H^\nat(Y,T)\cong I^\nat(S^3,K)\cong Kh^{red}(K^m)$ for any 2-bridge knot $K$.  In particular, this gives a novel approach to computing the reduced Khovanov homology of 2-bridge knots (with its bigrading $(i,j)$ reduced to $i-j+1$ mod 4).

\medskip
 
 We point out that the main {\em new} ingredients contained in this discussion of 2-bridge knots which were not implicit in \cite{HHK} are first, the construction of the complex $C(L_0,L_1)$ associated to a tangle decomposition of a 2-bridge knot, and second, the use of the cohomological invariant $z\in H^1(P^*;\ZZ/4)$ and the Maslov index  to define and compute the relative grading.

\section{Some general properties of $R(Y,T)$}\label{general}

 \subsection{Structure of $R(Y,T)$ near the abelian points}   
   \bigskip
Suppose that $(Y,T)$ is a 2-tangle.  Our goal (not fully realized in this article) is to establish Conjecture \ref{con1}.  To this end, we start by showing that the two boundary  points of $R_\pi(Y,T)$ are well defined for small perturbations, and  correspond to the precisely two {\em abelian} representations in $R(Y,T)$, namely the conjugacy classes of the two representations
$$r_\pm:\pi_1(Y\setminus T)\to \{\pm 1,\pm\bbi\}\subset SU(2)$$ uniquely characterized (since $H_1(Y\setminus T)=\ZZ \oplus \ZZ$, generated by $\mu_1,\mu_2$) by 
 $$r_\pm(\mu_1)=\bbi,r_\pm(\mu_2)=\pm \bbi.$$

  The following proposition proves that   $r_+$ and $r_-$ each have a neighborhood in $R(Y,T)$  homeomorphic to a half-open interval. The outline of the argument is as follows: the space $R(Y,T)$ is identified with a subspace   of the space of conjugacy classes of  representations of the 2-fold branched cover (the {\em equivariant representations} in the sense of \cite{Saveliev}).  Then a Kuranishi model argument shows that the   representation space is locally a half open interval near the lifts of $r_\pm$.    
  
\begin{prop} \label{halfopen} For a 2-tangle $T$ in an integer homology ball $Y$, each of the two abelian traceless representations $r_\pm$ has a neighborhood in $R(Y,T)$  homeomorphic to a half-open interval.  The restriction map to the pillowcase $P$ properly embeds each half-open interval, taking the endpoints to distinct corners with limiting slope not equal to 1. \end{prop} 

The proof is an extension of \cite[Theorem 3.2]{FKP}.  That theorem identifies the subvariety $R^{tbd}(Y,T)\subset R(Y,T)$ of traceless binary dihedral representations with the disjoint union of one arc and a number of circles (the number determined by the torsion submodule of the homology of the 2-fold branched cover of $Y$ branched along $T$).  The endpoints of the arc are precisely $r_+$ and $r_-$,  and the arc of binary dihedrals is linearly embedded into the pillowcase with slope different from 1 (see Section \ref{abgr}).   Hence what must be shown is that there are no non-binary dihedral representations in small enough neighborhoods of $r_\pm$.  

\medskip

We begin with a lemma which permits us to transfer the problem to one about the 2-fold branched cover of $(Y,T)$. To this end, 
 Let  $c:\pi_1(Y\setminus T)\to \{\pm 1\}$ be the unique homomorphism sending both $\mu_1$ and $\mu_2$ to $-1$ (this is just the homomorphism $r_+^2=r_-^2$).  Let   $B\to Y$ denote the corresponding 2-fold branched cover. Denote the preimage of $T$ by $\widetilde T$.  Consider $\pi_1(B\setminus \widetilde T)$ as the index 2 subgroup of $\pi_1(Y\setminus T)$, i.e., as the kernel of $c$. Let $\tilde \mu_1,\tilde\mu_2$ denote the meridians of the two components of $\widetilde T$. Hence, in $\pi_1(Y\setminus T)$, $\tilde\mu_i=\mu_i^2$.
 
 Denote by $R_{\pm 1}(B,\widetilde T)$ the space of conjugacy classes of 
 representations of $\pi_1(B\setminus \widetilde T)$ which take the $\tilde\mu_i$ to $\pm1$.  Since the square of a traceless element of $SU(2)$ is $-1$, restriction to the index 2 subgroup defines a map
 $$R(Y,T)\to R_{-1}(B,\widetilde T).$$

\begin{lem}\label{orbitlem}
Pointwise multiplication by $c$ defines a $\ZZ/2$ action on $R(Y,T)$ with fixed points the traceless binary dihedral representations. The restriction map  $R(Y,T) \to R_{-1}(B,\widetilde T)$  is constant on $\ZZ/2$ orbits and embeds the quotient $R(Y,T)/\ZZ/2\subset R_{-1}(B,\widetilde T)$.
\end{lem}
 
 Assuming Lemma \ref{orbitlem}, the proof of Proposition \ref{halfopen}  can be completed as follows.

 \medskip
 
  Denote by by  $\tilde r_\pm$  the  restrictions   of $r_\pm$ to the index 2 subgroup $\pi_1(B\setminus \widetilde T)$.  Then $\tilde r_\pm$ takes values in the center $\{\pm1\}$ of $SU(2)$ and $\tilde r_\pm(\tilde \mu_i)=r_\pm(\mu_i^2)=-1$.  
  
It follows that pointwise multiplication of a representation by $\tilde r_+$ defines a continuous map $R_{-1}(B,\widetilde T)\to R_{1}(B,\widetilde T)$. This map is a homeomorphism (in fact real analytic isomorphism) with inverse given again by multiplication by $\tilde r_+$.
 
Let $\chi(B)$ denote the space of conjugacy classes of (all) $SU(2)$ representations of $\pi_1(B)$. The Seifert-Van Kampen theorem shows that the restriction $\chi(B)\to R_{1}(B,\widetilde T)$ is a homeomorphism. Hence we have a sequence of maps :
$$R(Y,T)\to R(Y,T)/\ZZ/2\subset R_{-1}(B,\widetilde T)\cong R_{1}(B,\widetilde T)\cong \chi(B).$$

It therefore suffices  to prove that a neighborhood of  $\tilde r_+\tilde r_\pm$   in $\chi(B)$ is homeomorphic to a half-open interval.   Notice that $\tilde r_+\tilde r_+:\pi_1(B)\to SU(2)$ is the trivial representation, and $\tilde r_+\tilde r_-:\pi_1(B)\to SU(2)$ is central but non-trivial (it takes $\mu_1\mu_2$ to $-1$).

The 3-manifold $B$ has torus boundary and has first homology isomorphic to $\ZZ\oplus O$ for $O$ an odd torsion abelian group, since $T$ is a tangle in a homology ball (see \cite[Section 3]{FKP} for details).

The Kuranishi method identifies a neighborhood of $c_\pm$ in $R(B)$ with $K^{-1}(0)/SU(2)$, where $K:H^1(B;su(2)_{ad ~ c_\pm})\to 
H^2(B;su(2)_{ad ~ c_\pm})$.  The adjoint action of $c_\pm$ is trivial since $c_\pm$ is central, and hence these are untwisted cohomology groups with coefficients in $su(2)=\RR^3$. The universal coefficient theorem gives $H^1(B;\RR^3)=\RR^3$ and $H^2(B,\RR^3)=0$, so that $c_\pm$ has a neighborhood homeomorphic to $\RR^3/SU(2)=\RR^3/SO(3)\cong [0,1)$, as desired.  

\bigskip

 \noindent{\sl Proof of Lemma \ref{orbitlem}.}  First, if $\rho$ represents a conjugacy class in $R(Y,T)$, then the function $c\rho(\gamma)=c(\gamma)\rho(\gamma)$ is again a representation, since $c$ takes values in the center $\{\pm 1\}$. Moreover, since $\ker c=\pi_1(B\setminus\widetilde T)$, the restrictions of $\rho$ and $c\rho$ to $\pi_1(B\setminus\widetilde T)$ agree.  Since $c^2=1$, this shows that multiplication by $c$ defines a $\ZZ/2$ action on $R(Y,T)$ and the restriction $R(Y,T)\to R_1(B,\widetilde T)$ factors through the quotient of this $\ZZ/2$ action.

 Conversely, suppose $\rho_1,\rho_2:\pi_1(Y\setminus T)\to SU(2)$ are two traceless representations whose restriction to the index 2 subgroup $\pi_1(B\setminus \widetilde T)$ are equal. For clarity, denote this   restriction by $\bar \rho$, so $\bar\rho=\rho_1|_{\ker c}=\rho_2|_{\ker c}$.  
 
 We claim that, perhaps after conjugating $\rho_2$ without changing its restriction to $\ker c$,  $\rho_1(\mu_1)$ and $\rho_2(\mu_1)$ commute.  To see this, first note that for each $\tau\in \ker c$,
 \begin{equation}\label{both}
\rho_1(\mu_1\tau\mu_1^{-1})=\bar\rho(\mu_1\tau\mu_1^{-1})=\rho_2(\mu_1\tau\mu_1^{-1})
\end{equation}
so that 
 \begin{equation}
\label{commute}
[\rho_2(\mu_1)^{-1}\rho_1(\mu_1),\bar\rho(\tau)]=1\text{~ for all ~} \tau \in \ker c
\end{equation}

If $\bar\rho$ has non-abelian image, Equation (\ref{commute}) implies that $\rho_2(\mu_1)^{-1}\rho_1(\mu_1)$ is central, so that $\rho_2(\mu_1)=\pm \rho_1(\mu_1)$ and hence they commute.
If $\bar\rho$ has central image, then conjugating $\rho_2$ by any element of $SU(2)$ does not change its restriction to $\ker c$, and since $\rho_1(\mu_1)$ and $\rho_2(\mu_1)$ are traceless, they are conjugate. Hence $\rho_2$ can be conjugated so that $\rho_1(\mu)=\rho_2(\mu)$ and their restrictions to $\ker c$ agree.

Consider as a final case that $\bar\rho$ has abelian non-central image. We show that again $\rho_2$ can be conjugated without changing its restriction to $\ker c$ to make $\rho_1(\mu_1)$ and $\rho_2(\mu_1)$ commute.   Choose a traceless quaternion ${\bf q}$  so that the image of $\bar\rho$ lies in the circle subgroup ${\bf S}:=\{e^{\theta{\bf q}}\}$.  Then Equation  (\ref{commute})  shows that 
$ \rho_2(\mu_1)^{-1}\rho_1(\mu_1)$ lies in {\bf S}.   If one of $ \rho_1(\mu_1)$ or $\rho_2(\mu_1)$ lies in {\bf S} then they both do since their product does, and hence they commute. Suppose that neither lies in {\bf S}.  Equation (\ref{both}) shows that conjugation by $\rho_1(\mu_1)$ and $\rho_2(\mu_1)$ leaves the circle {\bf S} invariant.   This in turn shows that there exists an element of {\bf S}  which conjugates $\rho_2(\mu_1)$ to $\rho_1(\mu_1)$. This conjugation leaves {\bf S} fixed, so that we have shown that in this final case,  $\rho_2$ can be conjugated without changing its restriction to $\ker c$ to make $\rho_1(\mu_1)$ and $\rho_2(\mu_1)$ commute.  
 
Define $f:\pi_1(Y\setminus T)\to \{\pm 1\}$ by the formula
$$f(\gamma)=\begin{cases} 1&\text{ if } \gamma\in \ker c\\ \rho_1(\mu_1)\rho_2(\mu_1)^{-1}&\text{ if } \gamma\not \in \ker c.\end{cases}$$
Then it is easy to see that $f$ is a homomorphism (using the fact that $\rho_1(\mu_1)$ and $\rho_2(\mu_1)$ commute).  Moreover, a simple calculation shows that
$$\rho_2(\gamma)=f(\gamma)\rho_1(\gamma)\text{ for all }\gamma\in \pi_1(Y\setminus T).$$ 
Note that there are exactly two possibilities for $f$ since $\ker c$ has order 2. In fact, the two possibilitites are the trivial homomorphism and $c$.  This proves that the restriction map $R(Y,T)\to R(B, \widetilde T)$  factors through  an injective map on the orbit space of this $\ZZ/2$ action.

It remains to prove that the fixed points are exactly the traceless binary dihedral representations (\cite[Definition 3.1]{FKP}).  Suppose that $c\rho$ is conjugate to $\rho$ for $\rho\in R(Y,T)$.
Thus there exists $g\in SU(2)$ so that $g\rho(\gamma)g^{-1}=c(\gamma)\rho(\gamma)$ for all $\gamma\in \pi_1(Y\setminus T)$. In particular, $g\rho(\gamma)g^{-1}=\rho(\gamma)$ for all $\gamma\in \ker c$. Since $c(\mu_1)=-1 $, $g\ne \pm 1$, so that $g$ lies in a unique circle subgroup which we denote ${\bf S}$. 

If $\rho$ sends every $\gamma\in \ker c$ to the center $\{\pm 1\}$, then the image of $\rho$ lies in the subgroup $\{\pm 1,\pm \rho(\mu_1)\}$  of order $4$, and $\rho$ is traceless binary dihedral.

On the other hand, if there exists $\gamma\in\ker c$ such that $\rho(\gamma)\ne \pm 1$, then $g$ and $\rho(\gamma)$ commute, and hence $\rho$ sends all of $\ker c$ into ${\bf S}$.   Furthermore, for each $\gamma$ in the non-trivial coset, $\rho(\gamma)^{-1}g\rho(\gamma)=-g$, which implies that $\rho(\gamma)$ is traceless and ${\bf S}\cup \rho(\gamma) {\bf S}$ is (a conjugate of) the binary dihedral subgroup containing the image of $\rho$. Hence $\rho$ is traceless binary dihedral.
\qed

\subsection{Perturbations}\label{perturbationsection}

Proposition \ref{halfopen} shows that $R(Y,T)$ is a 1-manifold with boundary near the  two abelian representations $r_\pm$.  The space $R(Y,T)$ is a real algebraic variety,  but in general it may be singular.  To prove Conjecture \ref{con1} for some $(Y,T)$ one must first  desingularize $R(Y,T)$. 

There are various ways to smooth the singular space $R(Y,T)$; we restrict attention to  holonomy perturbations since these have a gauge theoretical counterpart which   permits us to compare our constructions to those of \cite{KM1, KM-khovanov}. In particular, with this choice of perturbations, Proposition \ref{instanton}   identifies the generators of the reduced knot instanton homology chain complex with the intersection points of $R_\pi(Y,T)$ and $L_0$ in the pillowcase for any appropriate perturbation $\pi$.  

We recall how to understand holonomy perturbations on the level of representations.  What follows can be taken as a definition. The reader should keep in mind, however, that the perturbed equations we give below arise from a perturbation of the Chern-Simons functional on the space of traceless $SU(2)$ connections. In particular, what we call a perturbation function is essentially the derivative of the conjugacy invariant function on $SU(2)$ which is used to perturb the Chern-Simons function.

\medskip

A holonomy perturbation is associated to  a pair $\pi=(E,f)$, where 
\begin{enumerate}
\item $E$ is an  embedding $E:S^1\times D^2\subset Y\setminus T$ (we use $E$ also as notation for the image $E(S^1\times D^2)$), and
\item $f$ is a perturbation function, i.e., $f\in \mathcal X=\{f\in C^\infty(\RR,\RR)~|~  f\text{~ is odd,  $2\pi$ periodic}\}$.
\end{enumerate}
Call a representation 
$\rho:\pi_1(Y\setminus(T\cup E))\to SU(2)$  a {\em $\pi$-perturbed traceless representation} if $\rho$ takes the meridians of $T$ to $C(\bbi)$, the conjugacy class of $\bbi$, and satisfies the {\em perturbation condition} on the meridian $\mu_E=E(\{1\}\times \partial D^2)$ and longitude $\lambda_E=E(S^1\times\{1\})$:
\begin{equation}
\label{pc}\rho(\lambda_E)=e^{\alpha Q}\text{~implies~} \rho(\mu_E)=e^{f(\alpha) Q}
\end{equation}
for $Q\in su(2)$.  Then define the {\em perturbed traceless flat moduli space} 
 $R_\pi(Y,T)$ to be the space of conjugacy classes of  $\pi$-perturbed traceless representations.   We refer the reader to \cite{Donaldson, Floer, Taubes}; expositions tailored to  our notation can be found in  \cite[Lemma 61]{herald1} and  \cite[Section 7]{HHK}.

More generally, one can choose a collection $E_i, i=1,\dots ,n$ of disjoint embeddings, and corresponding functions $f_i$ define $\pi=\{E_i, f_i\}$, and take  $R_\pi(Y,T)$ to be the space of conjugacy classes of  $\pi$-perturbed traceless representations, defined by requiring the perturbation condition (\ref{pc}) to hold for each $i$.
One useful choice is  $f_{i}(x)=\epsilon_i\sin(x)$ for some small $\epsilon_i$. 

The following proposition shows that the two abelian representations are stable with respect to (sup norm of $f_i$) small perturbations.

\begin{prop} \label{abs} For small enough perturbations, $R_\pi(Y,T)$ contains exactly two conjugacy classes of abelian representations. These are sent to distinct corner points in the pillowcase by the restriction map $R_\pi(Y,T)\to R(S^2,\{a,b,c,d\})$.
 \end{prop}
 
\begin{proof}
 Let $\mu_1,\mu_2$ denote  meridians of the two components of $T$.  Let $\mu_{E_1},\dots,\mu_{E_n}$ denote the meridians of the pertubation curves.  Then $\mu_1,\mu_2,\mu_{E_1},\dots,\mu_{E_n}$ generate $H_1(Y\setminus T)$.  
 
 Let $\ell_i(\mu_1,\mu_2,\mu_{E_1},\dots,\mu_{E_n})$, $i=1,\dots, n$, express the longitude $\lambda_{E_i}$ in $H_1(Y\setminus T)$ as a linear combination of the meridians of the meridians.     
 
 Identify the diagonal maximal torus in $SU(2)$ with the circle $S^1$ and let $T^{n+2}=(S^1)^{n+2}$.
  For each $\delta\ge 0$, Let $\pi(\delta)$ denote the perturbation data obtained by multiplying each $f_i$ by $\delta$.  Then  
  define a self-map of $T^{n+2}$
   $$Q(\delta):T^{n+2}\to T^{n+2}$$
 as follows. The first two coordinates encode the traceless condition and are given by
 $$Q(\delta)_1 (e^{\theta_1\bbi},e^{\theta_2\bbi}, e^{\alpha_1\bbi},\dots, e^{\alpha_n\bbi })= -e^{ 2\theta_1\bbi},~Q(\delta)_2(e^{\theta_1\bbi},e^{\theta_2\bbi}, e^{\alpha_1\bbi},\dots, e^{\alpha_n\bbi })= -e^{ 2\theta_2\bbi}. 
$$ 
The remaining coordinates encode the perturbation condition:
$$Q(\delta)_{i+2}(e^{\theta_1\bbi},e^{\theta_2\bbi}, e^{\alpha_1\bbi},\dots, e^{\alpha_n\bbi })=e^{ \alpha_i \bbi}e^{-\delta f_i(\ell_i(\theta_1,\theta_2,\alpha_1,\dots,  \alpha_n))\bbi}.
 $$
 Then $Q(\delta)^{-1}(1,\dots,1)$ parameterizes the perturbed traceless abelian representations (not conjugacy classes) with values in the diagonal maximal torus of $SU(2)$, with respect to the functions $\delta f_i$: the  point $(e^{\theta_1\bbi},e^{\theta_2\bbi}, e^{\alpha_1\bbi},\dots, e^{\alpha_n\bbi })\in Q(\delta)^{-1}(1,\dots,1)$
corresponds to the representation $$\pi_1(B\setminus (T\cup E))\to H_1(B\setminus (T\cup E)) \to S^1\subset SU(2)$$ sending each meridian  to its corresponding coordinate. 

The proof is completed by observing that $Q(0)$ is a covering map, hence a submersion. Since submersions are stable, $Q(\delta)^{-1}(1,\dots,1)$ varies by an isotopy for small $\delta$.
 \end{proof}
 
For any perturbation $\pi$, restricting to the boundary punctured sphere induces a map to the pillowcase
$$R_\pi(Y,T)\to R(S^2,\{a,b,c,d\}).$$ 
The two abelian representations guaranteed to persist after small perturbations by Proposition \ref{abs} necessarily are mapped to corners of the pillowcase, since the restriction of an abelian representation is abelian, and non-corner points are non-abelian, as one can see from Proposition \ref{pillow}.  

Putting the Propositions \ref{pillow}, \ref{halfopen} and \ref{abs}, together, we conclude the following.

\begin{thm}\label{endpointsgood} Let   $(Y,T)$ be a 2-tangle in a $\ZZ$-homology ball. Then, for any sufficiently small holonomy perturbation $\pi$, there are two abelian perturbed flat representations $r_\pm\in R_\pi(Y,T)$ with neighborhoods $U_\pm$ in $R_\pi(Y,T)$   half-open intervals.
The restriction map
 $$R_\pi(Y,T)\to R(S^2,\{a,b,c,d\})$$
 restricts to an immersion on $U_+\cup U_-$ which takes $r_\pm$ to distinct corners of the pillowcase, with slope $\ne$ 1.\qed \end{thm}

 Thus Theorem \ref{endpointsgood} reduces the problem of defining $H^\nat(Y,T)$ for a 2-tangle $T$ to 
 finding  an (arbitrarily) small holonomy  perturbation $\pi$ so that
 \begin{enumerate}
\item $R_\pi(Y,T)\setminus \{r_+,r_-\}$ is a smooth 1-manifold.
\item The restriction of $L_1:R_\pi(Y,T)\to P$ to   the arc component is an immersion into $P^*$   containing no fishtails.
\item The restriction of $L_1:R_\pi(Y,T)\to P^*$ to each circle component   is an immersion into $P^*$    containing no fishtails.
\end{enumerate}

It is well known that calculations of Zariski tangent spaces  using Poincar\'e-Lefschetz duality  show that  if  $R_\pi(Y,T)$ is a smooth 1-manifold away from the two endpoints, then the restriction map $L_1:R_\pi(Y,T)\to P$  immerses $R_\pi(Y,T)\setminus \{r_+,r_-\}$ into $P^*$.   Thus for a given $(Y,T)$, what is needed is a holonomy perturbation which desingularizes $R(Y,T)$ so that the resulting restriction to $P$ has no fishtails.

\section{Perturbing near the 2-sphere} \label{chris}

 In this section we construct    holonomy perturbations in $(S^2,\{a,b,c,d\})\times I$ which induce a family of Hamiltonian isotopies of the pillowcase. These were used in Sections \ref{RLITP} and \ref{2strand} to make $L_0$ and $L_1$ transverse.  These will alsobe used for other purposes below and in further work.  
 
 \medskip

Consider the product pair $$(S^2\times I, \{a,b,c,d\}\times I).$$ Its traceless character variety is $P$, and the traceless character variety of its boundary $$(S^2\times\{0,1\},  \{a,b,c,d\}\times\{0,1\})$$
is $P\times P$.  The restriction map  
$$R(S^2\times I, \{a,b,c,d\}\times I)\to R(S^2\times\{0,1\},  \{a,b,c,d\}\times\{0,1\})$$ is the diagonal map $P\to P\times P$, which we consider as the graph of the identity map $P\to P$.

Given suitable perturbation data $\pi$ for $(S^2\times I, \{a,b,c,d\}\times I)$, the restriction map  
$$R_\pi(S^2\times I, \{a,b,c,d\}\times I)\to R(S^2\times\{0,1\},  \{a,b,c,d\}\times\{0,1\})$$
gives a {\em Lagrangian correspondence}   $c_\pi:P\to P$.  Choosing a path from the trivial perturbation to $\pi$ gives a homotopy  of the identity to $c_\pi$.   We focus on a special class of $\pi$ for which  $c_\pi$ is an explicitly defined diffeomorphism.

\medskip

 Figure \ref{2scurve} shows the 4-punctured 2-sphere with the four based meridian generators $a,b,c,d$ based at a point $s$. An additional curve $e$ is also indicated. 
               \begin{figure}
\begin{center}
\def\svgwidth{2.5in}
 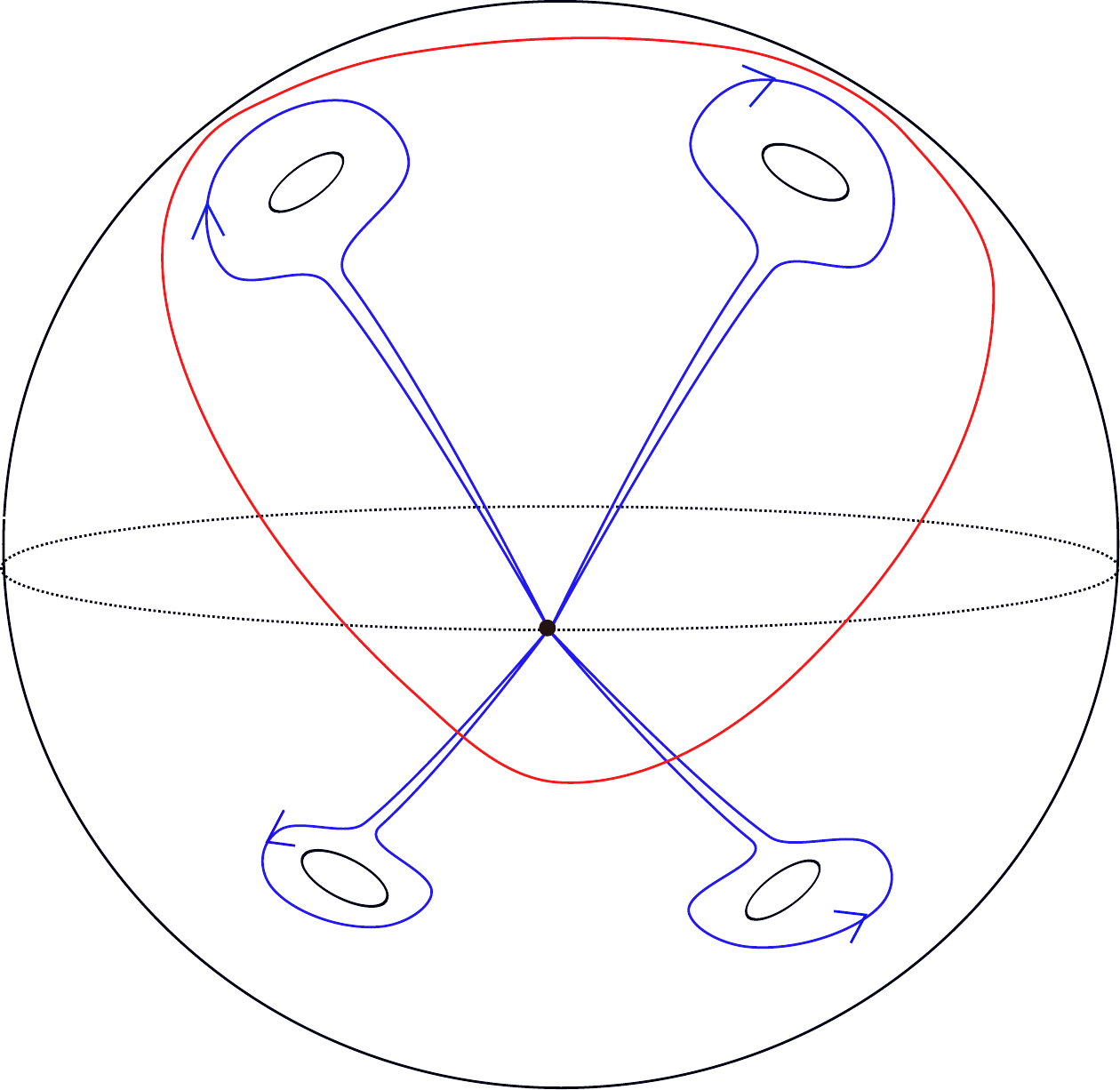
 \caption{\label{2scurve}}
\end{center}
\end{figure}

      Let  $E:S^1\times D^2\to S^2\setminus \{a,b,c,d\}\times I$ be a tubular neighborhood of   the curve obtained by pushing $e$ into the interior of $S^2\times I$.    Fix a perturbation function $f\in \mathcal{X}$ and let $\delta=(E,f)$ denote the perturbation data.  Recall that $f$ can be any smooth odd, $2\pi$ periodic function.

\begin{thm}\label{collar1} With perturbation data $\delta=(E,f)$, the map 
$$P\to R_\delta(S^2\times I, \{a,b,c,d\}\times I)$$ induced by the inclusion $S^2\times \{0\}\to S^2\times I$ is a homeomorphism, and the composite
$$P\to R_\delta(S^2\times I, \{a,b,c,d\}\times I)\to
 R(S^2\times\{0,1\},  \{a,b,c,d\}\times\{0,1\})=P\times P$$ is the graph of the self homeomorphism (smooth away from the corners) of the pillowcase
\begin{equation}
\label{horizontal}
c_\delta:P\to P,~ c_\delta(\gamma,\theta)=(\gamma,\theta+ 2f(\gamma+\pi)). 
\end{equation}  Using the 1-parameter family of perturbations $tf,t\in[0,1]$ gives an isotopy from the identity $Id:P\to P$ to $c_\delta:P\to P$.
\end{thm}

\begin{proof}  Let $a',b',c',d'$ and $\mu_E$ be based loops in $\pi_1(S^2\times I\setminus (\{a,b,c,d\}\times I\cup E),s)$ so that $a',b',c',d'$ represent the meridians of the punctures in the other boundary component $S^2\times \{1\}$, and $\mu_E$ denotes the meridian to the perturbation curve $E$. 
These curves are illustrated   in Figure \ref{s2xIfig}, where, for convenience,  the four-punctured sphere is identified with a three-punctured disk.

 \begin{figure}
\begin{center}
\def\svgwidth{2.7in}
 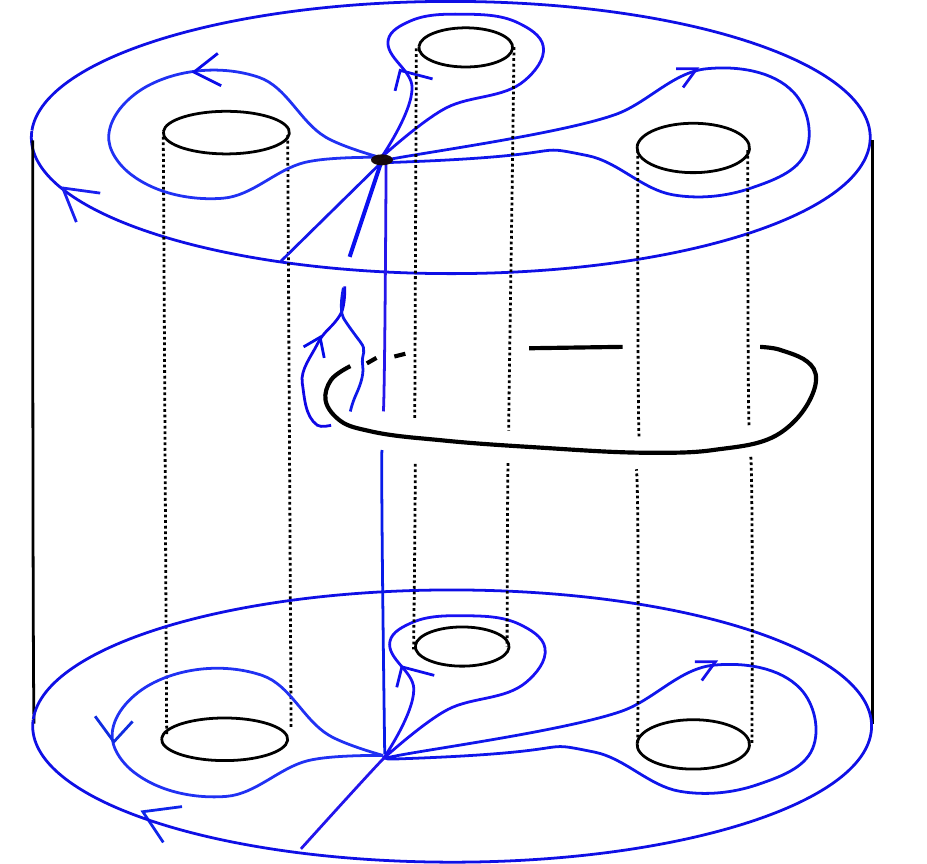
 \caption{\label{s2xIfig}}
\end{center}
\end{figure}

 The curves $a,b,c,d,\mu_E$ generate $\pi_1(S^2\times I\setminus (\{a,b,c,d\}\times I\cup E),s)$ 
 and the relations
 $$ba=cd, a'=a,  b'=b,c'=\mu_E c\bar\mu_E,d'=  \mu_E  d    \bar\mu_E,$$
hold.   
The natural longitude  $\lambda_E$  for $E$ is represented by the homotopy class $ba$.

As explained in \cite[Proposition 3.1]{HHK} (see Proposition \ref{pillow} above), any  representation of  $\langle a,b,c,d~|~ba=cd\rangle$ taking $a,b,c,d$ to traceless elements is  conjugate to one given by 
 \begin{equation}
\label{eq8.1}a\mapsto\bbi, b\mapsto e^{\gamma\bbk}\bbi,c\mapsto e^{\theta\bbk}\bbi,d\mapsto e^{(\theta-\gamma)\bbk}\bbi,
\end{equation}
for some $(\gamma,\theta)\in P$. Thus, to any representation $\rho:\pi_1(S^2\times I\setminus (\{a,b,c,d\}\times I\cup E))\to SU(2)$ sending $a,b,c,d$ to traceless elements, one can associate  $(\gamma,\theta)\in P$.  Then $\rho(\lambda_E)=\rho(ba)=-e^{\gamma\bbk}=e^{(\gamma+\pi)\bbk}$.

If  $\rho\in R_\delta(S^2\times I, \{a,b,c,d\}\times I)$, then $\rho$ satisfies the perturbation condition (see Equation (\ref{pc})):
\begin{equation}
\label{eq8.2}\rho(\mu_E)=e^{f(\gamma+\pi)\bbk}.
\end{equation}
Hence \begin{equation}\label{eq8.3}
\rho(a')=\bbi,~ \rho(b')=e^{\gamma\bbk}\bbi, ~\rho(c')=e^{f(\gamma+\pi)\bbk}e^{\theta\bbk}\bbi e^{-f(\gamma+\pi)\bbk}=e^{\theta+2f(\gamma+\pi)\bbk}
\end{equation}
Conversely, given any $(\gamma,\theta)\in P$ and $e^{\alpha\bbk}$ there exists a  unique traceless representation $\rho:\pi_1(S^2\times I\setminus (\{a,b,c,d\}\times I\cup E))\to SU(2)$ and
$
\rho(\mu_E)=e^{\alpha \bbk}.$  This satisfies the perturbation condition, and hence $\rho\in R_\delta(S^2\times I, \{a,b,c,d\}\times I)$ provided  $e^{\alpha\bbk}=e^{f(\gamma+\pi)\bbk}$.

We have shown that to each $(\gamma,\theta)\in P$ there exists a unique $\rho\in 
 R_\delta(S^2\times I, \{a,b,c,d\}\times I)$, given by (\ref{eq8.1}),  (\ref{eq8.2}), and  (\ref{eq8.3}). Moreover the restriction 
 $$R_\delta(S^2\times I, \{a,b,c,d\}\times I)\to
 R(S^2\times\{0,1\},  \{a,b,c,d\}\times\{0,1\})=P\times P$$
 has image $(\gamma,\theta, \gamma, \theta + 2f(\gamma+\pi))$

This shows that $d=(E,f)$ induces the map $c_\delta(\gamma,\theta)=(\gamma,\theta+ 2f(\gamma+\pi))$, as asserted. This map is invertible, with inverse $(\gamma,\theta)\mapsto(\gamma,\theta-2f(\gamma+\pi))$, and hence is a homeomorphism. \end{proof}

We stated Theorem \ref{collar1} for a specific curve $e$ in $S^2\setminus\{a,b,c,d\}$  but one may conjugate by any diffeomorphism $\phi$ of the punctured sphere to replace $e$ by $\phi(e)$, generating many more homeomorphisms of the pillowcase. Although   not used in the rest of this article, these perturbations will be important in forthcoming work.

\begin{thm} \label{collar2} Given any relatively prime pair of integers $p,q$ and $\phi\in \mathcal{X}$, there exists a holonomy perturbation along a single curve in $(S^2\times I, \{a,b,c,d\}\times I)$ inducing the homeomorphism
$$c_{p,q,f}(\gamma,\theta)=(\gamma-q\phi(p\gamma+q\theta),\theta+p\phi(p\gamma+q\theta))$$ of the pillowcase. This homeomorphism is Hamiltonian isotopic to the identity.\end{thm}

\begin{proof} Given $(\gamma,\theta)\in \RR^2$, Let $\psi({\gamma,\theta}):\pi_1(S^2\setminus \{a,b,c,d\})\to SU(2)$ be the 
traceless representation  of  Proposition \ref{pillow}.

Let $g:(S^2,\{a,b,c,d\})\to(S^2,\{a,b,c,d\})$ be the half-Dehn twist diffeomorphism supported in 
the hemisphere containing $c$ and $d$ which sends $c$ to $d$ and $d$ to $c$. (Thus $g^2$ is the Dehn twist about the curve labeled $e$ in Figure \ref{2scurve}.) 	 Choosing a base point near $a$, the induced automorphism on $\pi_1(S^2\setminus \{a,b,c,d\})$ is given by 
$$g_*(a)=a, g_*(b)=b, g_*(c)=d, g_*(d)=d^{-1}cd.$$
Then 
$\psi(\gamma,\theta)(g_*(b))=e^{\gamma\bbk}\bbi$ and $\psi(\gamma,\theta)(g_*(c))=e^{(\theta-\gamma)\bbk}\bbi$, i.e., 
$$g^*\psi(\gamma,\theta)=(\gamma,\theta-\gamma).$$
In other words, $g$ induces the linear map on the pillowcase:
\begin{equation}
\label{Ag}
g^*\begin{pmatrix}\gamma\\ \theta\end{pmatrix}=A_g\begin{pmatrix}\gamma\\ \theta\end{pmatrix},~A_g=\begin{pmatrix}1&0\\ -1&1\end{pmatrix}.
\end{equation}

Let $h:(S^2,\{a,b,c,d\})\to(S^2,\{a,b,c,d\})$ be the diffeomorphism which fixes (a neighborhood of) $a$, and cyclically permutes $b,c,d$.  This can be chosen to induce the automorphism 
 $\pi_1(S^2\setminus \{a,b,c,d\})$ is given by 
$$h_*(a)=a, h_*(b)=c^{-1}, h_*(c)=d, h_*(d)=(cd)^{-1}b^{-1}(cd)$$
Then 
$\psi(\gamma,\theta)(h_*(b))=-e^{\theta\bbk}\bbi=e^{(\theta+\pi)\bbk}\bbi$ and $\psi(\gamma,\theta)(g_*(c))=e^{(\theta-\gamma)\bbk}\bbi$, i.e.,  $h$ induces the affine map on the pillowcase:
\begin{equation}
\label{Ah}
h^*\begin{pmatrix}\gamma\\ \theta\end{pmatrix}=A_h\begin{pmatrix}\gamma\\ \theta\end{pmatrix}+ {\bf v},~A_h=\begin{pmatrix}0&1\\ -1&1\end{pmatrix}, {\bf v}=\begin{pmatrix}\pi \\ 0\end{pmatrix}.
\end{equation}

The matrices
$$S=A_gA_h^4=\begin{pmatrix}0&-1\\ 1&0\end{pmatrix}\text{~and~} T=A_hA_g^{-1}=\begin{pmatrix}1&1\\ 0&1\end{pmatrix}$$
are the standard generators of the modular group. It follows that given any relatively prime pair of integers $p,q$, there exists a word $w=w(g,h)$ in $g$ and $h$ so that the resulting diffeomorphism $w:(S^2,\{a,b,c,d\})\to(S^2,\{a,b,c,d\})$  satisfies
\begin{equation}
\label{w}
w^*\begin{pmatrix}\gamma\\ \theta\end{pmatrix}=A\begin{pmatrix}\gamma\\ \theta\end{pmatrix}+ {\bf u},~A=\begin{pmatrix}p&q\\ r&s\end{pmatrix}
\end{equation}
where $ps-qr=1$ and ${\bf u}$ is a vector whose entries are integer multiples of $\pi$.

The diffeomorphism $w$ induces a level preserving diffeomorphism $$w:(S^2\times I,\{a,b,c,d\}\times I)\to (S^2\times I,\{a,b,c,d\}\times I).$$
This diffeomorphism takes perturbed flat connections with respect to the perturbation curve $E$ of Theorem \ref{collar1} to perturbed flat connections with respect to $w(E)$.  

To simplify notation, write $\phi(x)=2f(x+\pi)+u_1$ where $f\in\mathcal{X}$ is the function used in Theorem \ref{collar1}, and $u_1$ is the first component of the vector ${\bf u}$. Note that $\phi\in \mathcal{X}$ if and only if $f\in\mathcal{X}$.  Then the self-homeomorphism of the pillowcase given by perturbing along $w(E)$  is the conjugate
$(w^*)^{-1} \circ c_\delta \circ w^*$, which we compute
\begin{eqnarray*}
\left((w^*)^{-1} \circ c_\delta \circ w^*\right)\begin{pmatrix}\gamma\\ \theta\end{pmatrix}&=&
(w^*)^{-1} \circ c_\delta \left(A\begin{pmatrix}\gamma\\ \theta\end{pmatrix}+{\bf u}\right)\\
&=&(w^*)^{-1}    \left(A\begin{pmatrix}\gamma\\ \theta\end{pmatrix}+{\bf u}+\begin{pmatrix}0\\ \phi(p\gamma+q\theta+u_1)\end{pmatrix}\right)\\
&=&A^{-1}\left(A\begin{pmatrix}\gamma\\ \theta\end{pmatrix}+{\bf u}+\phi(p\gamma+q\theta+u_1)\begin{pmatrix}0\\ 1\end{pmatrix}\right)-A^{-1}{\bf u}\\
&=&\begin{pmatrix}\gamma\\ \theta\end{pmatrix}+\phi(p\gamma+q\theta+u_1)\begin{pmatrix}-q\\ p\end{pmatrix}.
\end{eqnarray*}
If $u_1$ is an even multiple of $\pi$, then we are done, since these are pillowcase coordinates. If $u_1$ is an odd multiple of $\pi$, replace $\phi(x)$ by $\phi(x+\pi)$; this induces a bijection of $\mathcal{X}$.

\end{proof}

 \section{holonomy perturbations to smooth the traceless character variety of a 2-tangle decomposition of a torus knot}\label{torussec}

For the rest of this article we provide a detailed study of the traceless character varieties associated to a certain 2-tangle decomposition  of a torus knot $T_{p,q}$.   In the present section we establish that $(S^3,T_{p,q})$ admits a 2-tangle decomposition as in Equation \ref{decom} which verifies Conjecture \ref{con1}  except possibly for the absence of fishtails.  

 In the next section we identify $L_1:R_\pi(Y,T)\to P$ for a number of $T_{p,q}$ and verify that in all our examples, $L_1$ is indeed a restricted immersed 1-manifold, and that $H^\nat(Y,T)$ is either isomorphic to $I^\nat(S^3,T_{p,q})$, or, in examples where the calculation of $I^\nat(S^3,T_{p,q})$ is unknown, that the calculation of  $H^\nat(Y,T)$ combined with Conjecture \ref{con3} is consistent with the  conjecture  \cite{KM-s}
  that the ranks of $I^\nat(S^3,K)$ and   knot Heegaard-Floer homology  $\widehat{HFK}(K)$ are equal.
 
 \bigskip
 
 We recall the description of the traceless $SU(2)$ character variety of a tangle associated to the $(p,q)$-torus knot from \cite{HHK, FKP}.  Figure \ref{torusknot} illustrates a 3-component link $H_A\cup H_B\cup K$ in $S^3$,  with the component $K$ intersecting a 3-ball $D$ in a trivial 2-tangle $U$.  Integers $r,s$ satisfying  $pr+qs=1$ are fixed throughout.

Performing $-\frac s p$ Dehn surgery on the component labeled $H_A$ and  $\frac q r$ Dehn surgery on the component labeled $H_B$ yields $S^3$ again, and the knot labeled $K$ becomes the $(p,q)$ torus knot. In this $S^3$,  let $Y$ denote the complement of the illustrated 3-ball $D$,  and $T$ the part of the $(p,q)$ torus knot contained in $Y$.  Precisely, $Y$ is obtained from $S^3\setminus D$ by performing  $-\frac s p$ and $\frac q r$  surgery on $H_A$ and $H_B$, and $T\subset Y$ denotes that part of $K$ which lies in $Y$.   Note that $Y$ is itself diffeomorphic to a 3-ball. 

 Let $P_A$ and $P_B$ in $Y$ be the cores of the Dehn surgery solid tori which are added after neighborhoods of $H_A$ and $H_B$ are removed.  We will perform holonomy perturbations along these curves in $Y$.

Generators $ a,b,c,d,x,y$ of the fundamental group 
 $$\pi_1(Y\setminus (T\cup P_A\cup P_B))=\pi_1(S^3\setminus (D\cup K\cup H_A\cup H_B))$$ are illustrated.

   \begin{figure}
\begin{center}
\def\svgwidth{2.5in}
 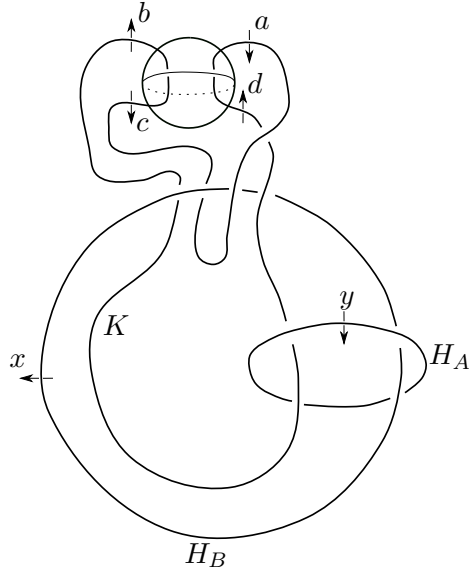
 \caption{The $(p,q)$ torus knot \label{torusknot}}
\end{center}
\end{figure} 

One computes (see \cite{HHK}) that 
$$\pi_1(Y\setminus (T\cup P_A\cup P_B))=
\langle x,y,a,b,c~|~ c=\bar x a x, a d \bar a=y x b  \bar x\bar y, [y,xb]=1, [x,d \bar a y]=1\rangle.
$$

The curves $$A_1=(xb)^qy^r\text{ and }A_2=(xb)^{-p}y^s$$ form a longitude-meridian pair for the component $P_A$.   The curves $$B_1=(d\bar a y)^{-s}x^p\text{ and }B_2=(d\bar a y)^rx^q$$ form a (commuting) longitude-meridian pair for the component $P_B$. 
In particular, $\pi_1(Y\setminus T)$ is obtained from 
$\pi_1(Y\setminus (T\cup P_A\cup P_B))$ by killing $A_2$ and $B_2$.

Working with the presentation of $\pi_1(Y\setminus (T\cup P_A\cup P_B))$ given above, together with the fact that $pr+qs=1$ yields:
$$A_1^s=(xb)^{qs}y^{rs}=(xb)(xb)^{-pr}y^{rs}=xb((xb^{-p}y^s)^r=xbA_2^r$$
and so $$xb=A_1^{s}A_2^{-r}.$$  
Similar calculations give
$$y=A_2^q A_1^p,~ x= B_2^sB_1^r,~ d\bar a y= B_1^{-q}B_2^p,$$
from which one obtains
\begin{equation}
\label{aandb}
a=yxb\bar x (d \bar a y)^{-1}=A_1^{s+p}A_2^{q-r}B_1^{q-r}B_2^{-(s+p)}             \text{ and }b=\bar x x b= B_1^{-r}B_2^{-s}A_1^sA_2^{-r}.
\end{equation}

Since $c=\bar x a x$ and $d=\bar a yxb\bar x \bar y a$, it follows that the four elements $A_1,A_2,B_1,B_2$ generate $\pi_1(Y\setminus (T\cup P_A\cup P_B))$.   A simple extension of the observation in \cite{HHK} that  $\pi_1(Y\setminus T)$ is free on $A_1$ and $B_1$ (they are labeled $A$ and $B$ in that article)  shows that 
$\pi_1(Y\setminus (T\cup P_A\cup P_B))$ is the free product of the free abelian group generated by $A_1,A_2$ and the free abelian group generated by $B_1,B_2$. 

We use the  perturbation functions $\epsilon_A\sin x$ on $P_A$ and $\epsilon_B\sin x$ on $P_B$ for some $(\epsilon_1,\epsilon_2)\in \RR^2$.   Recall from  Equation (\ref{pc}) that with this choice, perturbed-flat connections modulo gauge are identified with representations $\rho:\pi_1(Y\setminus (T\cup P_A\cup P_B))\to SU(2)$ which satisfy 
the {\em perturbation conditions}
\begin{equation}
\label{pert}
\rho(A_2)=e^{\epsilon_A \sin u \, Q_A} \text{ if } \rho(A_1)=e^{u Q_A} \text{ for some } Q_A\in C(\bbi)\end{equation}  
$$\rho(B_2)=e^{\epsilon_B \sin v \,Q_B} \text{ if } \rho(B_1)=e^{v Q_B}~ \text{ for some }Q_B\in C(\bbi).$$

If $(\eA,\eB)=(0,0)$, then perturbed-flat connections send $A_2$ and $B_2$ to $1\in SU(2)$, hence by the Seifert-Van Kampen theorem correspond exactly to $SU(2)$ representations of $\pi_1(Y\setminus T)$.
 
As above, we define the {\em perturbed traceless flat moduli space} 
 $$R_{\eA,\eB}(Y,T)=\{\rho:\pi_1(Y\setminus (T\cup P_A\cup P_B))\to SU(2)~|~\text{$\rho$ traceless,  satisfying   (\ref{pert})}\}/_{\text{conjugation}}$$

\begin{thm} \label{torusknotsgood} There exists a neighborhood  $\mathcal{O}\subset\RR^2$ of $(0,0)$ such  that    for any $(\eA,\eB)\in \mathcal{O}$, the space $R_{\eA,\eB}(Y,T)$ is a smooth compact 1-manifold with two boundary points and such that the restriction map to the pillowcase $
R_{\eA,\eB}(Y,T)\to P$ satisfies the conditions to be a restricted immersed 1-manifold except possibly the absence of fishtails. \end{thm}

   In extensive calculations we have not
found any small non-zero $\eA,\eB$   for which $
R_{\eA,\eB}(Y,T)\to P$ is {\em not} a restricted immersed 1-manifold.

\bigskip

The strategy to prove Theorem \ref{torusknotsgood} is standard: we form a parameterized moduli space, prove it is a smooth manifold, and apply Sard's theorem to the projection to $\RR^2$.   We start with  a {\em gauge fixing} theorem which identifies      $ R_{\eA,\eB}(Y,T)$ with a subset of 
the box $[0,\pi]\times[0,\pi]\times[-1,1]$.

   A representation $\rho:\pi_1(Y\setminus (T\cup P_A\cup P_B))\to SU(2)$ satisfying the perturbation conditions with respect to $( \eA,\eB)$  is traceless   if and only if 
$\rho(a)$ and $\rho(b)$ are traceless. From Equation (\ref{aandb}) this holds if and only if
$$\Real(\rho(A_1^{s+p}A_2^{q-r}B_1^{q-r}B_2^{-(s+p)} ) )=0\text{ and }\Real(\rho(B_1^{-r}B_2^{-s}A_1^sA_2^{-r}))=0.$$ Assuming that $ \rho(A_1)=e^{u Q_A}$ and $\rho(B_1)=e^{v Q_B}$  for some pair of purely imaginary unit quaternions $Q_A, Q_B\in C(\bbi)$,  these can be expressed equivalently as 
$$\Real\left(e^{((s+p)u+(q-r)\eA\sin u)Q_A}e^{((q-r)v-(s+p)\eB\sin v)Q_B}\right)=0
$$
and
$$\Real\left(e^{(-rv-s\eB\sin v)Q_B} e^{(su-r\eA\sin u)Q_A)}\right)=0
$$
or equivalently (see \cite[Proposition 2.1]{HHK}) as $\Psi(\eA,\eB,u,v,\tau)=(0,0)$, where 
$\Psi=(\Psi_1,\Psi_2)$
is defined by 
\begin{equation}\label{psi1}
\begin{multlined}
\Psi_1(\eA,\eB,u,v,\tau)=\cos((q-r)v-(s+p)\eB\sin v)\cos((s+p)u+(q-r)\eA\sin u) \\
\qquad -\sin((q-r)v-(s+p)\eB\sin v)\sin((s+p)u+(q-r)\eA\sin u)\tau
\end{multlined}
\end{equation}
and
\begin{equation}
\label{psi2}
\begin{multlined}
\Psi_2(\eA,\eB,u,v,\tau)=\cos(-rv-s\eB\sin v)\cos(su-r\eA\sin u)\qquad \\
\qquad  -\sin(-rv-s\eB\sin v)\sin(su-r\eA\sin u)\tau,
\end{multlined}
\end{equation}
with $\tau$  the cosine of the angle between $Q_A$ and $Q_B$.

\begin{thm}\label{WR} Fix $(\eA,\eB)\in \RR^2$ and let
 $$W_{\eA,\eB}=\{(u,v,\tau)\in [0,\pi]\times[0,\pi]\times[-1,1]
 ~|~ \Psi( \eA,\eB,u,v,\tau)=0   \}.$$
 Then the assignment
 $$A_1\mapsto e^{u \bbi}, A_2\mapsto e^{\epsilon_A \sin u \bbi}
 , B_1\mapsto e^{v e^{\arccos \tau\bbk}\bbi}
 B_2\mapsto e^{\epsilon_B \sin v e^{\arccos \tau \bbk}\bbi }$$
 defines a surjection 
 $$f:W_{\eA,\eB}\to R_{\eA,\eB}(Y,T)$$
 whose fiber over $f(u,v,\tau)$ is the  single point $\{(u,v,\tau)\}$ unless   $ \sin u=0$ or   $\sin v=0$, in which case the fiber is the arc $\{(u,v)\}\times [-1,1]$. The map $f$ sends the points of $W_{\eA,\eB}$ interior to the box to non-abelian perturbed representations and boundary points to perturbed abelian representations.
\end{thm}
The proof of this lemma is identical to \cite[Theorem 11.1]{HHK} (see also \cite[Theorem 4.2]{FKP}).  The essential point is that if $\sin u\ne0\ne \sin v$  then any representation can be uniquely conjugated so that $A_1$ is sent to $e^{ui}$ and $B_1$ is sent to 
$e^{v e^{\arccos \tau \bbk}\bbi}$ for $\tau$ the cosine of the angle between   this representation's $Q_A$ and $Q_B$.   The perturbation condition then determines where $A_2$ and $B_2$ are sent.  We leave the details to the reader.

A point $(u,v,\tau)\in W_{\eA,\eB}$ which lies on the boundary of the box  corresponds to a representation which  sends $A_1$ and $A_2$ to the center $\pm 1$ if $\sin u=0$, sends $B_1$ and $B_2$ to  $\pm 1$ if $\sin v=0$,  and sends $A_1,A_2,B_1,B_2$ to commuting elements if $|\tau|=1$.  It follows that points in $W_{\eA,\eB}$ meeting the boundary of the cube  correspond exactly to abelian representations (i.e., representations with abelian image). There are two conjugacy classes of (unperturbed) traceless abelian representations.  This property is stable with respect to small perturbations, as shown in Proposition \ref{abs}.

 The result needed to prove complete the proof of Theorem \ref{torusknotsgood} is the following. It says that the map $\Psi$ of Equations  (\ref{psi1}), (\ref{psi2}) is submersive near non-abelian points of the unperturbed traceless character variety $R(Y,T)$.  Hence for generic small $(\epsilon_A,\epsilon_B)$, the non-abelian part of 
 $R_{\epsilon_A,\epsilon_B}(Y,T)$ is smooth.

\begin{lem}\label{sublemma}  Suppose that $u,v,\tau$ are chosen so that $\Psi(0,0,u,v,\tau)=0$, with $\sin u\ne 0, \sin v\ne 0,$ and $  |\tau|\ne 1$. 
Then $d\Psi_{(0,0,u,v,\tau)}:\RR^5\to \RR^2$ is surjective, and hence $\Psi$ is a submersion near $(0,0,u,v,\tau)$. 
\end{lem}
\begin{proof} The proof is essentially a lengthy second-year calculus computation, and we recommend the reader skip it.    

Consider first $\Psi_2$.  To clarify, we adopt the following notation
$$A=\cos(-rv), B=\sin(-rv), C=\cos(su), D=\sin(su).$$ 
From Equation (\ref{psi2}), the partial derivatives of $\Psi_2$ at the point $(0,0,u,v,\tau)$ are given by
$$\left. \frac{\partial \Psi_2}{\partial \eA}\right|_{(0,0,u,v,\tau)}=A(-D)(-r)\sin u-BC(-r)\tau\sin u=r(AD+BC\tau)\sin u
$$
 $$\left. \frac{\partial \Psi_2}{\partial \eB}\right |_{(0,0,u,v,\tau)}=-B(-s\sin v)C -A(-s\sin v)D\tau=s(BC+AD\tau)\sin v$$
 $$\left. \frac{\partial \Psi_2}{\partial u}\right|_{(0,0,u,v,\tau)}=A(-D)s-BCs\tau=-s(AD+BC\tau)$$
  $$\left. \frac{\partial \Psi_2}{\partial v}\right|_{(0,0,u,v,\tau)}= -B(-r)C-A(-r)D\tau=r(BC+AD\tau)$$
  $$\left. \frac{\partial \Psi_2}{\partial \tau}\right|_{(0,0,u,v,\tau)}=-BD$$
Moreover, the equation $\Psi_2(0,0,u,v,\tau)=0$ is equivalent to 
$$AC-BD\tau=0.$$  
  Similarly, adopting the notation
$$E=\cos((q-r)v), F=\sin((q-r)v), G=\cos((s+p)u), H=\sin((s+p)u),$$ we obtain
$$\left. \frac{\partial \Psi_1}{\partial \eA}\right|_{(0,0,u,v,\tau)}=-(q-r)(EH+FG\tau)\sin u
$$
 $$\left. \frac{\partial \Psi_1}{\partial \eB}\right|_{(0,0,u,v,\tau)}=(s+p)(FG+EH\tau)\sin v$$
 $$\left. \frac{\partial \Psi_1}{\partial u}\right|_{(0,0,u,v,\tau)}=-(s+p)(EH+FG\tau)$$
  $$\left. \frac{\partial \Psi_1}{\partial v}\right|_{(0,0,u,v,\tau)}=-(q-r)(FG+EH\tau)$$
  $$\left. \frac{\partial \Psi_1}{\partial \tau}\right|_{(0,0,u,v,\tau)}=-FH$$
  Moreover, the equation $\Psi_1(0,0,u,v,\tau)=0$ is equivalent to 
$$EG-FH\tau=0.$$ 

Suppose that $\Psi$ were not a submersion at $(0,0,u,v,\tau)$.  Then $d\Psi_1$ and $d\Psi_2$ are linearly dependent.    
 
 If $d\Psi_2=0$, then  $AD+BC\tau=0$, $BC+AD\tau=0$ and $BD=0$. Therefore $AD=AD\tau^2$. Since $|\tau|\ne 1$, this implies $AD=0$ and $BC=0$.  If $A=\cos(-rv)=0$, then $B=\sin(-rv)\ne 0$, so $\cos(su)=C=0$ and hence $\sin(su)=D\ne0$, contradicting $BD=0$.  But if $A\ne 0$,  $\sin(su)=D=0$ and so $\cos(su)=C\ne 0$ and hence $B=0$. But this contradicts $AC-BD\tau=0$.  Therefore, $d\Psi_2\ne 0$.  
 Similarly, $d\Psi_1\ne 0$.

 Since neither $d\Psi_1$ nor $d\Psi_2$ is zero, there there exists a non-zero $\alpha\in \RR$ so that $\alpha d\Psi_1=d\Psi_2$. Comparing the first columns, (i.e., $\frac{\partial\ \ }{\partial \eA}$) and using the fact that   $\sin u\ne 0$ one sees 
 $$AD+BC\tau=\alpha\left(\frac{r-q}{r}\right)(EH+FG\tau).$$
 Similarly, comparing third columns gives
 $$AD+BC\tau=\alpha\left(\frac{s+p}{s}\right)(EH+FG\tau).$$
Since $\frac{(r-q)s}{(s+p)r}=1-\frac{1}{r(s+p)}\ne 1$,  and $\alpha\ne 0$, this implies that
$EH+FG\tau=0,$ and hence also $AD+BC\tau=0$.
Comparing the second and fourth columns and applying the same argument  yields 
$FG+EH\tau=0$ and $AD+BC\tau=0$.

Then $FG=-EH\tau=FG\tau^2$. Since $|\tau|\ne 1$, $FG=0=EH$ and, similarly, $AD=BC=0$. Since   neither $d\Psi_1$ nor $d\Psi_2$ is zero, $BD\ne 0$ and $FH\ne 0$. 
Thus $A=C=0$ and $G=E=0$.  

Recalling their definitions, this says that
 \begin{equation}
\label{cos}\cos(-rv)=\cos(su)=\cos((q-r)v)=\cos((s+p)u)=0.
\end{equation}
Hence there exist odd integers $k,\ell$ so that $-rv=k\frac{\pi}{2}$ and $(q-r)v=\ell \frac{\pi}{2}$. Thus $(q-r)k=-r\ell $. Since $r$ and $q-r$ are relatively prime, there exist odd integers
$m,n$ so that $k=rm$ and $\ell=(q-r)n$. Thus $v=\frac{k}{-r}\frac{\pi}{2}=-m\frac{\pi}{2}$.  Similarly, $u$ is an odd multiple of $\frac \pi 2$. Equation (\ref{cos}) then implies that $r,s,q-r,s+p$ are all odd, but then $p,q$ are both even, contradicts the fact that $p$ and $q$ are relatively prime.  Thus the assumption that $\Psi$ is not a submersion leads to a contradiction.
 
  Hence $\Psi$ is  a submersion at $(0,0,u,v,\tau)$, and so also near $(0,0,u,v,\tau)$. \end{proof}


\noindent{\sl Proof of Theorem \ref{torusknotsgood}.}
Recall that $W_{0,0}$ denotes the preimage in the box $[0,\pi]\times [0,\pi]\times [-1,1]$ of $0$ by the map $(u,v,\tau)\mapsto \Psi(0,0,u,v,\tau)$. Let $V$ be the intersection of a small 
open neighborhood of $W_{0,0}$ with the interior of the box.

Lemma \ref{sublemma}  implies that (after perhaps choosing a smaller $V$), there is a neighborhood $\mathcal{O}$ of $0$ in $\RR^2$ so that $\Psi: \RR^2\times [0,\pi]\times [0,\pi]\times [-1,1]\to \RR^2 $ restricts to a submersion on $\mathcal{O}\times V$. The {\em parameterized moduli space} $P:=\mathcal{O}\times V\cap \Psi^{-1}(0)$ is a smooth submanifold of $\mathcal{O}\times V$. By Sard's theorem, there exist   regular values $(\eA,\eB)$ of the composite $P\subset \mathcal{O}\times V\xrightarrow{\pi_\mathcal{O}} \mathcal{O}$ arbitrarlily close to 0. Its preimage in the interior of the box is identified with the non-abelian part of $R_{\eA,\eB}(Y,T)$ by Theorem \ref{WR}.  The structure near the two abelian representations was identified in Theorem \ref{endpointsgood}. 
\qed


 \section{Calculations for torus knots}\label{examples}

In this section we carry out calculations of $C^\nat(L_0,L_1)$  for some torus knots, including examples with non-trivial differentials. In what follows, we continue to use the description of the torus knot $K=T_{p,q}$  illustrated in Figure \ref{torusknot}, where we perform $-\frac s p $ Dehn  surgery on the component labeled $H_A$ and $\frac q r $ surgery on the component labeled $H_B$.  Figure \ref{torusknot} illustrates a decomposition
$$(S^3,T_{p.q})=(Y,T)\cup_{(S^2,\{a,b,c,d\})}(D,U).$$
Recall that this decomposition depends on the choice of integers $r,s$ satisfying $pr+qs=1$.  Different choices of $r,s$ lead to different pairs $R(Y,T)$ and restriction maps $L_1$.  

The identification of the spaces $R_\pi(Y,T)$ and their image in the pillowcase was done using a computer algebra package. 
\medskip

 Recall that Kronheimer-Mrowka prove (\cite{KM-khovanov,KM-Alexander}, also Lim  \cite{lim})  that the rank of the reduced singular instanton homology $I^\nat(S^3,K)$ is bounded below by the sum of the absolute value of coeffcients of the Alexander polynomial, which we denote by $|\Delta_K|$. 

It is conjectured that the reduced instanton homology of a torus knot $K$ has rank {\em equal} to  $|\Delta_K|$. This is a special case of a more sweeping conjecture which relates singular instanton homology of a knot and its  Heegaard-Floer homology.

If Conjecture \ref{con3} is true, then the rank of  $H^\nat(Y,T,\pi)$ must be at least as large as $|\Delta_K|$ for a tangle decomposition of a torus knot $K$, and if in addition  the conjecture of the previous paragraph is true, then rank$(H^\nat(Y,T,\pi))=|\Delta_K|$. The reader can verify that in all the examples given below, the rank of $H^\nat(Y,T,\pi)$ equals $|\Delta_K|$.

\medskip

 In  the following calculations, we take $L_0=L_0^{\ep,0}$ for a small $\ep>0$.   We calculate Maslov indices relative to the slope 1 line field $\ell_1$, making use of  Equation (\ref{spcase}) to simplify grading calculations. We make frequent use of the calculus described in Section \ref{special}.

\medskip

  \subsection{The $(5,11)$ torus knot}\label{5-11tk}
Consider the tangle associated to the $(5,11)$ torus knot, corresponding to the  choice $(5,11,9,-4)$. The {\em unperturbed} traceless character variety $R(Y,T)$  is a restricted immersed 1-manifold. 

In fact, $R(Y,T)$ is a union of an arc $R_0$ and   four circles $R_1,R_2,R_3,R_4$. 
 The arc $R_0$ embeds linearly with slope 2.     Two of the circles, say $R_1,R_2$ are vertically monotonic with vertical degree 2. The remaining circles $R_3,R_4$ each map precisely in the way illustrated in the  example of Figures \ref{re1}, \ref{PP}, \ref{re2}, and \ref{re3}.

  Since the signature of the $(5,11)$ torus knot is $-24\equiv 0$ mod 4, $R_0$ contributes $(1,0,0,0)$ to $H^\nat(Y,T)$ by our absolute grading convention.  Proposition \ref{vm} implies that the two vertically monotonic circles $R_1,R_2$  contribute $(1,1,1,1)$ each to $H^\nat(Y,T)$.   
 
The contributions  of $R_3$ and $R_4$ to $C(L_0,L_1)$ and $H^\nat(Y,T)$ were computed in detail in 
 Section \ref{examplecalc}; it was shown that each contributes a $(2,2,2,2)$ summand to    $C(L_0,L_1)$, each has two bigons contributing to the differential, and each contributes $(1,1,1,1)$ to the homology $H^\nat(Y,T)$.
Thus 
the differential $\partial:C(L_0,L_1)\to C(L_0,L_1)$ has rank 8 and
$$H^\nat(Y,T)=(1,0,0,0)\oplus_{i=1}^4(1,1,1,1)=(5,4,4,4).$$
In particular, the rank of $H^\nat(Y,T)$ is 17, which equals $|\Delta_K|$. The calculation of $I^\nat(S^3,T_{5,11})$ is unknown, but Conjecture \ref{con3} would imply that    the Kronheimer-Mrowka lower bound is attained for $T_{5,11}$.

\subsection{The  (3,7)  torus knot} Taking the decomposition of the $(3,7)$ torus knot corresponding to the choice $(p,q,r,s)=(3,7,5,-2)$, it is shown in \cite{FKP} that the space $R(Y,T)$ is the disjoint union of an arc $R_0$ (consisting of binary dihedral representations) and two circles $R_1$, $R_2$.  In particular, $L_1:R(Y,T)\to P$ is a restricted immersed 1-manifold.

The restriction  $L_1|_{R_0}:R_0\to P$ is the linear embedding of slope 2 (\cite[Theorem 4.1]{FKP}), given by 
$[0,\pi]\ni t\mapsto (t,2t)\in P$. The restrictions of $L_1$ to $R_1$ and $R_2$ have the same image, and each is a vertically monotonic circle of vertical degree 2.

The line segment of slope 2 has a unique intersection point with $L_0$, namely the point $r_+^\ep$. The signature of $T_{3,7}$ is $-8\equiv 0 $ mod 4 and so   $R_0$ contributes $(1,0,0,0)$ to $H^\nat(Y,T)$. 

 Since $R_1$ and $R_2$ are vertically monotonic with vertical degree 2,  Proposition \ref{vm} implies that $R_1$  and $R_2$ each contribute   a summand  $(1,1,1,1)$ to $C(L_0,L_1)$ and the differential is zero in these summands.
Hence 
$$H^\nat(Y,T)=(1,0,0,0)\oplus_{i=1}^2(1,1,1,1)=(3,2,2,2).$$ This agrees with the calculation of the reduced instanton homology of the $(3,7)$ torus knot $I^\nat(S^3, T_{3,7})$ (see \cite{HHK}), as well as the calculation of the ($\ZZ/4$ graded) reduced Khovanov homology $Kh^{red}(T_{3,7}^m)$.

\subsection{The  (5,7)  torus knot}  It is established in \cite{FKP} that  taking $(p,q,r,s)=(5,7,3,-2)$, $R(Y,T)$ is smooth, and has two components, an arc $R_0$ and a circle $R_1$. The restriction of $L_1$ to   $R_0$  is linear with slope 2.  The restriction of $L_1$ to the circle is a 2-1 cover onto its image which winds four times vertically around $P^*$.  In particular, $L_1:R_1\to P$ is vertically monotonic with vertical degree $d=8$.

The arc component $R_0$ has a unique intersection point with $L_0$, namely the point $r_+^\ep$.   The signature of the $(5,7)$ torus knot is $-16\equiv 0$ mod 4, and hence   $R_0$ contributes $(1,0,0,0)$ to $C(L_0,L_1)$ and $H^\nat(Y,T)$.  Since  $R_1$ is vertically monotonic with vertical degree $8$, Proposition \ref{vm} implies that $R_1$ contributes    $(4,4,4,4)$ to $C(L_0,L_1)$. Moreover, all differentials are zero in this summand.  Thus we conclude that all differential vanish and 
$$H^\nat(Y,T)=(5,4,4,4).$$ This agrees with the calculation of the reduced instanton homology $I^\nat(S^3, T_{5,7})$. Moreover, the reduced Khovanov homology $Kh^{red}(T_{5,7}^m)$   equals $(8,6,7,8)$, which has stricly larger rank,  corresponding to the fact that there are non-trivial higher differentials in the Kronheimer-Mrokwa spectral sequence from Khovanov to instanton homology \cite{KM-khovanov}.

\subsection{The  (5,12) torus knot}  Taking $(p,q,r,s)=(5,12,5,-2)$, it is shown in \cite{FKP} that $R(Y,T)\to P$ is a restricted immersed 1-manifold composed of one arc $R_0$ mapping with slope 6,  and two vertically monotonic circles $R_1,R_2$ each  of vertical degree 6. 

Proposition \ref{vm} shows that the components $R_1$ and $R_2$ each contribute $(3,3,3,3)$  to $C(L_0,L_1)$ and $H^\nat(Y,T)$. For the component $R_0$, we calculate in exactly the same manner as was done for 2-bridge knots. There are 5 representations, $r_+^\ep, x_1^\pm, y_1^\pm$, and $gr(r_+^\ep, x_1^+)=2$, $gr(x_1^+,x_2^+)=2$, and $gr(x_i^+,x^-_i)=1$. The signature of $K_{5,12}$ is $-28$, and hence $gr(r_+^\ep)=0$ mod 4. Since $R_0$ maps linearly with slope 6 to the pillowcase, there can be no bigons, and thus all differentials are zero. Therefore,  the component $R_0$ contributes $(2,1,1,1)$ to $H^\nat(Y,T)$, generated by $\{r_+^\ep, x_2^+\}, \{x_1^-\}, \{x_1^+\}, \{x_2^-\}$ respectively.   We conclude that 
$$H^\nat(Y,T)=(2,1,1,1)\oplus_{i=1}^2(3,3,3,3)=(8,7,7,7).$$ This agrees with the calculation of $I^\nat(S^3,T_{5,12})$ in \cite{HHK}, and is smaller than the reduced Khovanov homology $Kh^{red}(T_{5,12}^m)=(20,19,19,19)$.

\subsection{The   (5,17) torus knot}   Taking $(5,17,7,-2)$,  it is shown in \cite{FKP} that $R(Y,T)\to P$  is a restricted immersed 1-manifold, the union of an arc $R_0$ mapping linearly with slope 2 and three circles $R_1,R_2,R_3$ are vertically monotonic  with vertical degrees $6,6$  and $8$. The signature equals $-40\equiv 0$ mod 4, and so the single intersection $r_+^\ep$  on the arc $R_0$ contributes $(1,0,0,0)$ to $C(L_0,L_1)$ and $H^\nat(Y,T)$. Proposition \ref{vm} implies that the circles $R_1, R_2, R_3$ contribute  $(3,3,3,3),(3,3,3,3)$ and $(4,4,4,4)$  respectively to $C(L_0,L_1)$ and $H^\nat(Y,T)$. Hence
 $$H^\nat(Y,T)=(11,10,10,10),$$ which agrees with the calculation of the instanton homology $I^\nat(S^3,T_{5,17})$ in \cite{HHK}.     Two of the circles (the degree 6 circles) have the same image and the degree 8 circle double covers its image.

  \subsection{The (3,4) torus knot}     We next give an example of a tangle decomposition so that the unperturbed traceless character variety $R(Y,T)$ is not smooth.  After smoothing it using a holonomy perturbation, $R_\pi(Y,T)$ becomes a restricted immersed 1-manifold and one finds a non-trivial differential  in the summand corresponding to the arc component.   
  
  Take $(p,q,r,s)=(3,4,3,-2)$, it was shown in \cite[Proposition 11.4]{HHK} that $R(Y,T)$ is a singular space, obtained from 3 arcs $I_0\cong[0,\pi]$, $I_\pm\cong[\frac{\pi}{6},\frac{5\pi}{6}]$ by identifying the endpoints of $I_+$ and $I_-$ to $I_0$ at $\frac{\pi}{6}$ and $\frac{5\pi}{6}$ to form a singular variety homeomorphic to the letter  $\phi$.  The restriction map to the pillowcase takes the arc $I_0$ (consisting of binary dihedral representations) to the arc of slope $-2$ via $[0,\pi]\ni t\mapsto (\pi-t,-2t)$ and takes each of the two arcs $I_\pm$ to linear arcs of slope 4.  The space $R(Y,T)$ and its image in the pillowcase is illustrated in Figure \ref{fig34}.

    \begin{figure}
\begin{center}
\def\svgwidth{3.5in}
 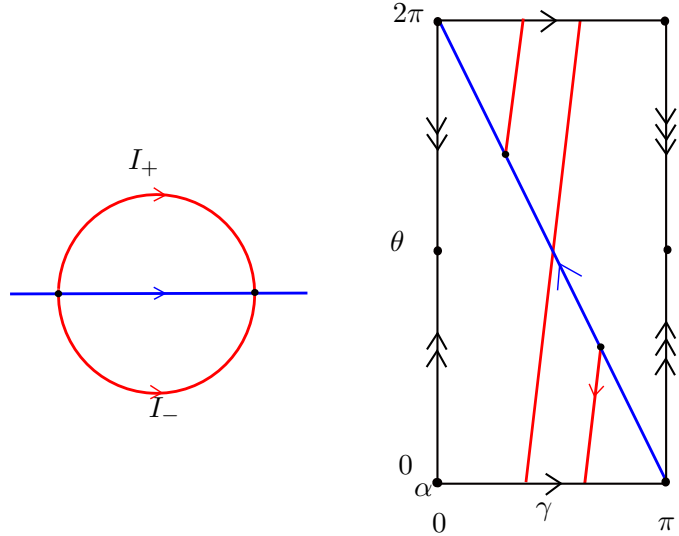
 \caption{$R(Y ,T)$ and its image in the pillowcase for the  (3,4) torus knot. \label{fig34}}
\end{center}
\end{figure} 

Applying Theorem \ref{torusknotsgood}, we can find an arbitrarily small perturbation so that $R_{\eA,\eB}(Y,T)$ is smooth.  A lengthy calculation (or using a computer algebra package) reveals that  $R_{\eA,0}(Y,T)$ is the union of an arc and a circle for any small non-zero $\eA$, as illustrated in Figure \ref{fig34pert}.

    \begin{figure}
\begin{center}
\def\svgwidth{3.5in}
 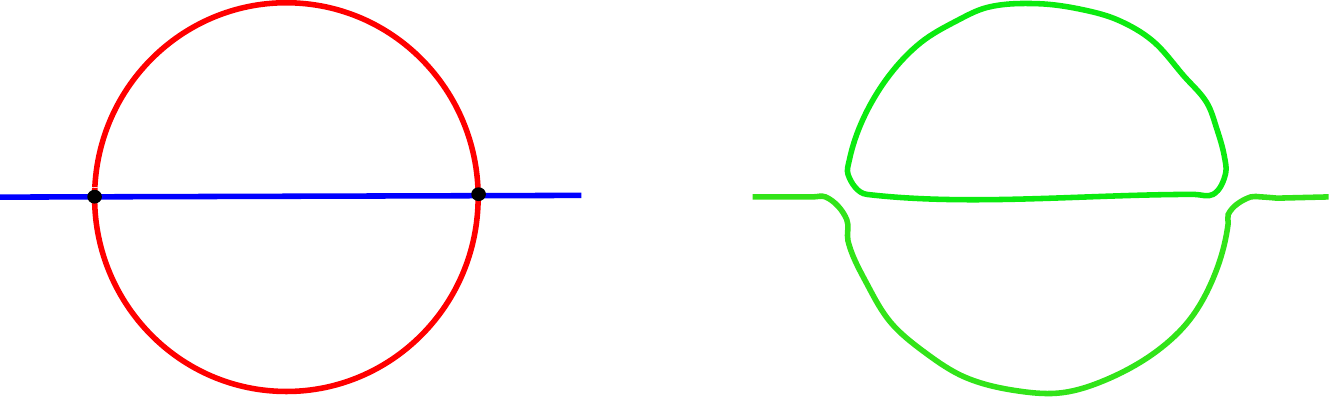
 \caption{$R(Y ,T)$ resolving to $R_{\eA,0}(Y,T)$ for the  (3,4) torus knot. \label{fig34pert}}
\end{center}
\end{figure}
The image of the perturbed character variety, a restricted immersed 1-manifold composed of one arc and one circle,  is illustrated in Figure \ref{fig34pert2}, along with the image of $L_0$. The arc component $R_0$ intersects $L_0$ in three points, $r_+^\ep,  x_1^+$ and $x_1^-$. The signature of the $(3,4)$ torus knot is $-6\equiv 2$ mod 4, and so $gr(r_+^\ep)=2$. The relative gradings are  $gr(x_1^+,x_1^-)=1$ and $gr(x_1^-,r_+^\ep)=1$. Hence $gr(x_1^-)=3$ and $gr(x_1^+)=0$.

A bigon from $x_1^-$ to $r_+^\ep$ is indicated in in Figure \ref{fig34pert2}.  This is the only bigon and gives the non-zero differential $\partial x_1^-=r_+^\ep$.  Thus $R_0$ contributes $(1,0,0,0)$ to $H^\nat(Y,T,\pi)$. The circle component $R_1$ is vertically monotonic  of vertical  degree 2 and hence contributes $(1,1,1,1)$ to $C(L_0,L_1)$ and $H^\nat(Y,T)$. Hence $$H^\nat(Y,T,\pi)=(2,1,1,1),$$ which is isomorphic   to the reduced instanton homology $I^\nat(S^3,T_{3,4})$, as well as the reduced Khovanov homology $Kh^{red}(T_{3,4}^m)$.

    \begin{figure}
\begin{center}
\def\svgwidth{4in}
 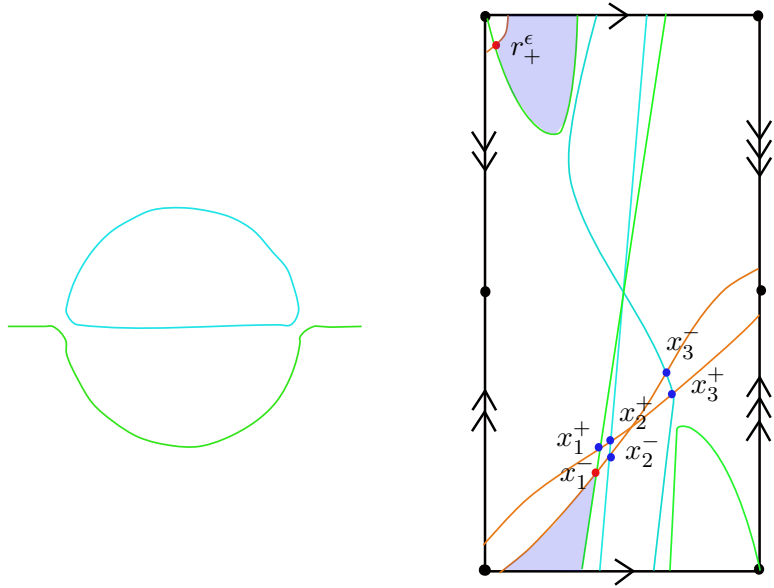
 \caption{The image of $R_{\eA,0}(Y,T)$ in the pillowcase for the (3,4) torus knot.  The bigon giving a non-trivial differential is shaded. \label{fig34pert2}}
\end{center}
\end{figure}

 \subsection{The (3,5) torus knot}     Taking $(p,q,r,s)=(3,5,2,-1)$, the space $R(Y,T)$ is again a (singular) $\phi$ curve, made up of an arc $I_0=[0,\pi]$ of traceless binary dihedral representations which maps to the bottom edge of the pillowcase with slope 0: $ [0,\pi]\ni t\mapsto (t,0)$ and  two arcs $I_\pm
 =[\tfrac\pi 6,\tfrac{5\pi} 6]$ whose interiors consist of non-binary dihedral representations and which map to the   pillowcase linearly with slope 6: $I_\pm
 =[\tfrac\pi 6,\tfrac{5\pi} 6]\ni t\mapsto (t, 6(t-\tfrac{\pi} 6))$.

However,   the singularities resolve differently than in the case of the $(3,4)$ torus knot which was illustrated in Figure \ref{fig34pert}: the perturbed variety $R_{\eA,\eB}(Y,T)$  is a single arc $R_0$.  The image of this arc in the pillowcase is illustrated in 
Figure \ref{fig35pert}.     
     \begin{figure}
\begin{center}
\def\svgwidth{4.4in}
 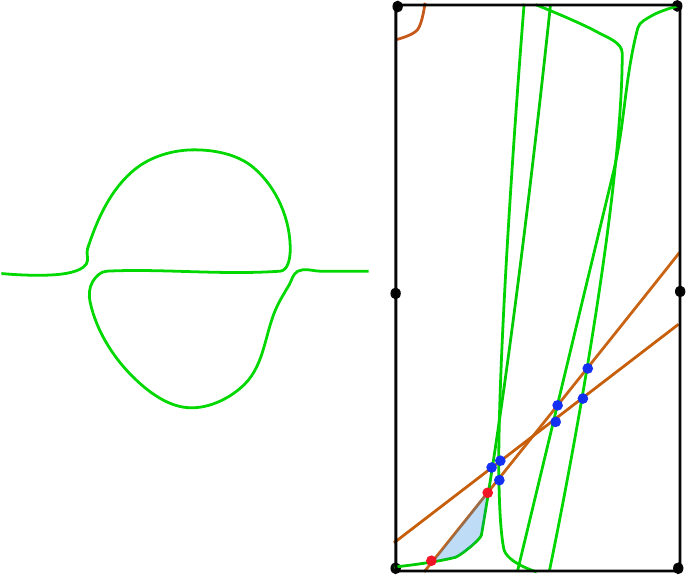
 \caption{The image of $R_{\eA,\eB}(Y,T)$ and $L_0$  in the pillowcase for the (3,5) torus knot. The bigon giving a non-trivial differential is shaded. \label{fig35pert}}
\end{center}
\end{figure}        
 
One can easily see a bigon from  $x_1^-$ to $r_+^\ep$.  This is the only bigon, so that $\partial x_1^-=r_+^\ep$ is the only non-zero differential, and $gr(x_1^-,r_+^\ep)=1$. The signature of the $(3,5)$ torus knot is $-8$, so that $gr(r_+^\ep)=0$. One computes $gr(x_1^-,x_2^-)=gr(x_2^-,x_3^-)=gr(x_3^-,x_4^-)=2$. Together with $gr(x_\ell^+,x_\ell^-)=1$, this gives $C(L_0,L_1)=(3,2,2,2)$ and 
$$H^\nat(Y,T,\pi)=(2,1,2,2).$$  Once again, this agrees as an absolutely $\ZZ/4$ graded group, with $I^\nat(S^3,T_{3,5})$ (and $Kh^{red}(T_{3,5}^m)$).
 

      \subsection{The (4,5) torus knot}    This example is interesting in the context of instanton homology, as it was shown by Kronheimer-Mrowka \cite{KM-filtrations} that there is a non-trivial  higher differential in their spectral sequence from Khovanov homology to instanton homology.

The description of $R(Y,T)$ for the $(4,5)$ torus knot, corresponding to the tangle decomposition associated to  $(p,q,r,s) =(4,5,4,-3)$,  is analyzed in detail in \cite[Section 4]{FKP}.  We refer the reader to that article, where it is shown that $R(Y,T)$ is again a $\phi$ curve, and its restriction map to $P$ is illustrated, along with its nine intersection points with $L_0$ generating $C(L_0,L_1)$  and the reduced instanton complex.

  \begin{figure}
\begin{center}
\def\svgwidth{2in}
 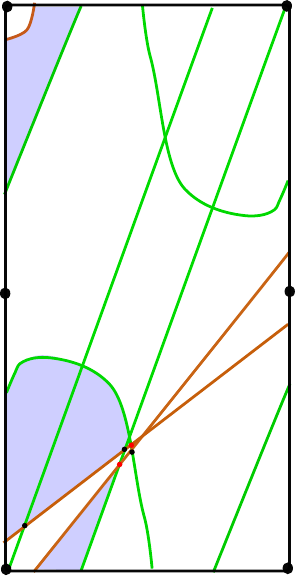
 \caption{The image of the immersed arc $R_0\subset R_{\eA,\eB}(Y,T)$ and $L_0$  in the pillowcase for the (4,5) torus knot. The bigon giving a non-trivial differential is shaded. \label{fig45pert}}
\end{center}
\end{figure}  
 
 A computer-aided calculation shows that  $R_{\eA,0}(Y,T)$ is the union of an arc $R_0$ and a circle $R_1$ and that the restriction to the pillowcase is a restricted immersed 1-manifold.   The circle $R_1$ is vertically monotonic of vertical degree 2. Hence $R_1$ contributes $(1,1,1,1)$ to $C(L_0,L_1)$ and $H^\nat(Y,T)$.

      The image of  the arc $R_0\subset R_{\eA,0}(Y,T)$ in the pillowcase is illustrated in Figure \ref{fig45pert}.  There is only one bigon and, in contrast to the examples of the $(3,5)$ and $(4,5)$ torus knots given above, the non-trivial differential does not involve the canonical generator $r^\ep_+$.  The signature of the $(4,5)$ torus knot is $-8$, so that $gr(r_+^\ep)=0$.

   One computes that $R_0$  contributes  $(2,1,1,1)$ to  $C(L_0,L_1)$. This uses the observation that $R_0$ has two slope 1 tangencies.  The differential takes a generator in grading 1 to a generator in grading 0. Hence $R_0$ contributes $(1,0,1,1)$ to $H^\nat(Y,T,\pi)$, so that 
   $$H^\nat(Y,T,\pi)=(2,1,2,2).$$ Once again, this is isomorphic as a $\ZZ/4$ graded group to $I^\nat(S^3,T_{4,5})$, computed in  \cite{KM-filtrations}.  

\medskip

It is worth contrasting this calculation with the one Kronheimer-Mrowka give of the instanton homology $I^\nat(S^3,K)$. They start with the count of the nine generators and their relative grading to get  the relatively  graded chain complex with ranks (up to cyclic reordering) $(3,2,2,2)$. They then compare this to $Kh^{red}(T_{4,5})=  (2,1,3,3)$ to conclude, from the incompatibility of gradings, that there must be a non-trivial  differential. A further non-trivial argument identifies this differential. By contrast, the differential is manifest in our pictures. Of course, Conjecture \ref{con3} may be false, and so one should be cautious in drawing   conclusions.

     \subsection{The (4,7) torus knot}         
We take the tangle decomposition of the $(4,7)$ torus knot determined by the choice $(p,q,r,s)=(4,7,2, -1)$. The singular variety  $R(Y,T)$ and its image in the pillowcase is illustrated in  \cite[Figure 9]{FKP}.   The smoothed variety $R_{\eA,\eB}(Y,T)$ is the union of an arc $R_0$ and two circles $R_1,R_2$. The map to the pillowcase is a restricted immersed 1-manifold.

         \begin{figure}
\begin{center}
\def\svgwidth{5.5in}
 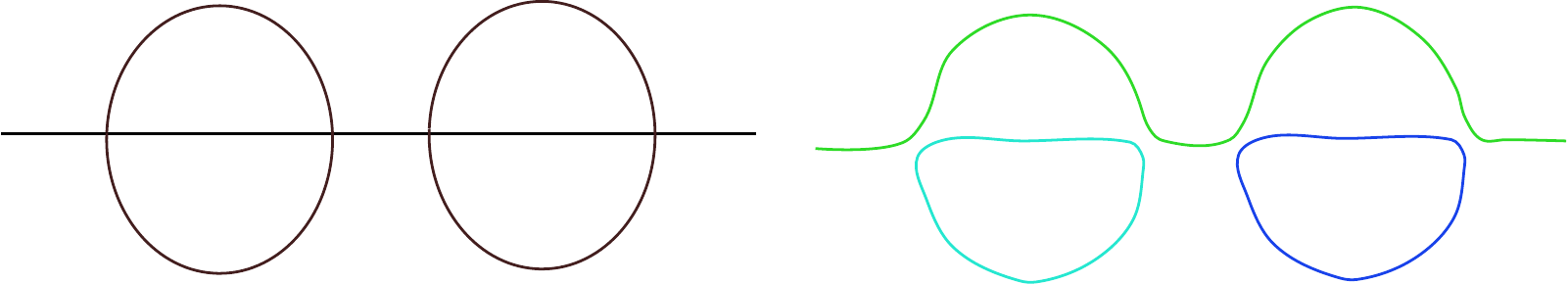
 \caption{$R(Y,T)$ and its smoothing $R_{\eA,\eB}(Y,T)=R_0\cup R_1\cup R_2$ for the (4,7) torus knot.  \label{fig47pert}}
\end{center}
\end{figure}

              \begin{figure}
\begin{center}
\def\svgwidth{1.7in}
 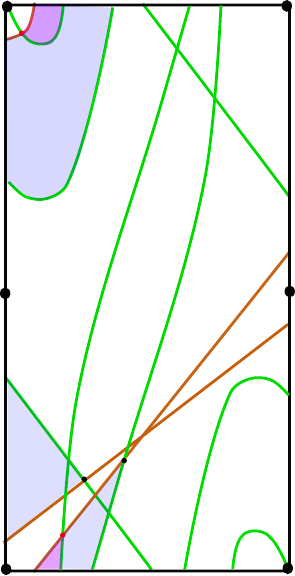
 \caption{The image of the immersed arc $R_0\subset R_{\eA,\eB}(Y,T)$ in the pillowcase for the $(4,7)$ torus knot. The two bigons are shaded.  \label{fig47pert2}}
\end{center}
\end{figure}  

The two circles each map to vertically monotonic circles of vertical degree 2; each contributes $(1,1,1,1)$ to $H^\nat(Y,T,\pi)$. The restriction map of the arc $R_0$ to the pillowcase  is illustrated in Figure \ref {fig47pert2}. There are seven intersection points, and two bigons are shaded.   Notice also the four points where the tangent line of $L_1(R_0)$ is tangent to the slope 1 line  field $\ell_1$.  
The $(4,7)$ torus knot has signature $-14\equiv 2$ mod 4, so that $gr(r_+^\ep)=2$.  From this one computes that the contribution of $R_0$ to $C(L_0,L_1)$ is $(2,1,2,2)$ and to $H^\nat(Y,T,\pi)$ is $(1,0,1,1)$. Therefore, 
$$H^\nat(Y,T,\pi)= (3,2,3,3).$$
 In particular, the rank is 11, equal to the sum of the absolute value of the coefficients of the Alexander polynomial of $T_{4,7}$.

The calculation of $I^\nat(S^3,T_{4,7})$ is unknown to us. Notice that $C(L_0,L_1)=(4,3,4,4)$ (with one non-trivial differential from grading 1 to 0 and the other from grading 3 to 2). This is consistent  with the relative gradings computed for the instanton chain complex in \cite{HHK}. 
In  that article we computed gradings  using a spectral flow splitting formula   approach (suggested in \cite{KM-khovanov}), based, not on a tangle decomposition as in the current article, but rather on the decomposition of the form
$(S^3,K)=(S^3\setminus N(K),\phi)\cup_{T^2}(N(K),K)$.  
This yielded  $CI^\nat(S^3,T_{4,7})=(1,0,0,0)_a\oplus (4,4,3,3)_b$, where the subscripts denote a possible cyclic reordering. We conjecture in that article that $a=\sigma(K)$ and $b=3$, which implies that 
$CI^\nat(S^3,T_{4,7})=(0,0,1,0)\oplus(4,3,3,4)=(4,3,4,4)$. Thus, although we do not know the instanton homology, we do see that the  generators of the instanton chain complex occur in the same gradings as for  $C(L_0,L_1)$.

\subsection{Changing $\ep$ to cancel bigons} 

  Consider the effect of varying $\ep$ in the definition for $L^{\ep,0}_0:S^1\to P^*$ in Definition \ref{tilfixed} and illustrated in Figure \ref{Lzero}. 

 For very small $\ep>0$, there are $2n+1$  intersection points of $L_1(R(Y,T))$ and $L_0^\ep(R(D,U))$, where $n$ corresponds to the number of intersections of $L_1(R(Y,T))$ with the diagonal arc $\Delta$, or equivalently (see \cite{HHK})  $n$ equals the number of conjugacy classes of non-abelian traceless representations of the corresponding knot complement. This is doubled to account for the fact that $L_0^{\ep,0}$ is a figure 8 close to $\Delta$, and the extra intersection point corresponds to the the perturbation of the unique abelian traceless representation which is  mapped to the corner.   An illustration of this is given in Figure \ref{fig72}, where one sees  $10=2\cdot 5$ points, labeled $x_i^\pm$, corresponding to the five intersection points with the interior of $\Delta$, and one extra point $r_+^\ep$ which converges to the bottom left corner as $\ep\to 0$.
 
   Increasing  the parameter $\ep$ in  the holonomy perturbation function   makes the figure 8 $L_0^{\ep,0}$ wider (see \cite{HHK}).  In some circumstances, the regular homotopy of $L_0^{\ep,0}$ obtained by increasing $\ep$ can be used to cancel pairs of intersection points, and hence reduce the rank $C(L_0,L_1)$. 
 
 For example, for the $(3,5)$ torus knot, a bigon is illustrated in Figure \ref{fig35pert}.  As $\ep$ increases, the pair of intersection points $r_+^\ep$ and $x_1^-$ get closer together and eventually cancel.    Explicitly, when $\ep=0.2, \ep_A=0.1$, and $\ep_B=0$, $L_1 (R_{\eA, 0}(Y,T))$ intersects $L_0$ in 7 points. This is the minimum possible by the lower bound   given by the sum of absolute value of the coefficients of the Alexander polynomial ${t}^{8}-{t}^{7}+{t}^{5}-{t}^{4}+{t}^{3}-t+1$ of the $(3,5)$ torus knot, since, by Proposition \ref{instanton}, the rank of $C(L_0^{\ep,g},L_1)$ cannot be smaller than the rank of the instanton homology. 
 
 The same method works  for   the tangle decomposition of the $(3,4)$ torus knot with $R(Y,T)$  illustrated in Figure \ref{fig34}; increasing $\ep$ to $0.8$ removes the two generators spanned by a bigon.   
This shows that a suitably perturbed Chern-Simons function on the configuration space of the $(3,4)$ torus knot is perfect.  

We summarize:
\begin{prop}
 There exists a holonomy perturbations $CS+h$ of the Chern-Simons function on the orbit space of singular connections on   $(S^3,T_{3,4})$ and on $(S^3,T_{3,5})$   so that $CS+h$ is  perfect, and hence all differentials in the singular instanton complex vanish. \qed
\end{prop}
   
In general, simply increasing $\ep$ does not eliminate every pair of generators spanned by a  bigon.  For example,  in the 2-tangle decomposition of the $(4,5)$ torus knot, increasing $\ep$ increases the number of intersection points of $L_0^\ep$ with $L_1$. For the $(4,7)$  torus knot, one pair ($r^\ep_+$ and $x_3^-$ in Figure \ref{fig47pert2}) can be eliminated by increasing $\ep$ but the second pair ($x_2^+$ and $x_1^-$) cannot.

 Theorem \ref{rhi}  implies that  $H^\nat(L_0,L_1)$ is unchanged by a homotopy of $L_0$, and, in particular,  bigons can be used as guides to regularly homotop away pairs of intersection points.  For example, in the case of the $(4,7)$ torus knot, one can easily find a curve $L_0'$ homotopic to $L_0^\ep$ so that all differentials in the resulting complex are zero. It is not clear whether such an $L_0'$ can can be found only using holonomy pertubations.
 
\section{Loose ends}

Several problems remain to be settled before our approach can be considered as producing a functioning invariant of knots.  Conjectures \ref{con1} and \ref{con3} need to be further investigated. 

An important 
first problem is to determine the extent to which $H^\nat(Y,T,\pi)$ depends on the perturbation $\pi$.  Given two perturbations $\pi$, $\pi'$ for which $ R_\pi(Y,T)\to P$ and $ R_{\pi'}(Y,T)\to P$ are both restricted immersed 1-manifolds, they are not necessarily related by a regular homotopy, but rather by a {\em Legendrian cobordism} \cite{herald1}.  For example, in the calculations with torus knots described in Section \ref{examples},  
  we used the perturbation corresponding to $\eA>0, \eB=0$ with $\eA$ small to smooth $R(Y,T)$.  Typically, using  $\eA<0,\eB=0$ resolves the normal crossing  singularities along the arc of binary dihedrals in the opposite way. For these examples, the resulting $H^\nat(Y,T,\pi)$ is unchanged by reversing the sign of $\eA$.
But in general, Legendrian cobordisms need not preserve Lagrangian-Floer homology.

A closely related question concerns the existence of fishtails, which obstruct $\partial^2=0$.
We would like to know that there are no fishtails for small perturbations $\pi$.

A third question concerns the relationship of $H^\nat(Y,T,\pi)$ to reduced Khovanov homology, a question already solved for Heegaard-Floer theory in \cite{OS} and for singular instanton homology in \cite{KM-khovanov}.   In forthcoming work we explore this question,  extending the definition of $H^\nat(Y,T)$  to include links, and we have established a 
skein exact triangle for $H^\nat(Y,T,\pi)$.  We expect this to lead to a spectral sequence from Khovanov homology to $H^\nat(Y,T,\pi)$ and to an approach to prove Conjecture \ref{con1}.

A fourth question concerns the promoting of the constructions of this article to $n$-tangle decompositions of knots and links. Some related work   includes \cite{JR}, which studies   decompositions of a knot into two trivial $n$-tangles.  
The symplectic variety corresponding to the pillowcase in this setting is no longer 2-dimensional, making it much more difficult to understand and compute with.

In a different direction, the rich collections of isotopies of the pillowcase described in  Theorems \ref{collar1} and \ref{collar2} are induced by holonomy perturbations, which also induce analytically appropriate perturbations of the Chern-Simons functional for the construction of instanton homology.     These should prove useful in isotoping $L_0^\ep$ to reduce the number of generators of the instanton chain complex.

  
\end{document}